\documentclass[11pt, reqno]{amsart}

\usepackage{amsmath}

\usepackage{tabu}
\usepackage{tikz}
\usepackage[margin=1.25in]{geometry}
\usepackage[matrix,arrow,curve,frame]{xy}
 \usepackage{hyperref}
\usepackage{amsthm}
\usepackage{amsfonts}
\usepackage{amssymb}
\usepackage{wasysym}
\usepackage{mathrsfs}
\usepackage{mathtools}

\definecolor{my-linkcolor}{rgb}{0.75,0,0}
\definecolor{my-citecolor}{rgb}{0.1,0.57,0}
\definecolor{my-urlcolor}{rgb}{0,0,0.75}
\hypersetup{
	colorlinks,
	linkcolor={my-linkcolor},
	citecolor={my-citecolor},
	urlcolor={my-urlcolor}
}

\newtheorem{thm}{Theorem}[section]
\newtheorem{cor}[thm]{Corollary}
\newtheorem{lemma}[thm]{Lemma}
\newtheorem{prop}[thm]{Proposition}
\theoremstyle{definition}\newtheorem{defn}[thm]{Definition}
\theoremstyle{definition}\newtheorem{remark}[thm]{Remark}
\theoremstyle{definition}
\theoremstyle{definition}\newtheorem{assum}[thm]{Assumption}

\newcommand{\Id}{\mathrm{Id}}
\newcommand{\id}{\mathrm{Id}}

\newcommand{\tens}{\boxtimes}
\newcommand{\vac}{\mathbf{1}}

 \DeclareMathOperator{\im}{Im}
 \DeclareMathOperator{\coker}{Coker}
 
 \DeclareMathOperator{\rep}{Rep}
 \let\ker\relax
 \let\hom\relax
 \DeclareMathOperator{\ker}{Ker}
 \DeclareMathOperator{\hom}{Hom}

 \newcommand{\til}[1]{\widetilde{#1}}

\def\CC{\mathbb{C}}

\def\FF{\mathbb{F}}

\def\NN{\mathbb{N}}

\def\RR{\mathbb{R}}

\def\ZZ{\mathbb{Z}}

\newcommand\cA{\mathcal{A}}
\newcommand\cB{\mathcal{B}}
\newcommand\cC{\mathcal{C}}
\newcommand\cD{\mathcal{D}}
\newcommand\cE{\mathcal{E}}
\newcommand\cF{\mathcal{F}}
\newcommand\cG{\mathcal{G}}

\newcommand\cM{\mathcal{M}}

\newcommand\cP{\mathcal{P}}

\newcommand\cR{\mathcal{R}}
\newcommand\cS{\mathcal{S}}

\newcommand\cU{\mathcal{U}}
\newcommand\cV{\mathcal{V}}
\newcommand\cW{\mathcal{W}}
\newcommand\cX{\mathcal{X}}
\newcommand\cY{\mathcal{Y}}

\newcommand\tilM{\widetilde{M}}

\newcommand\tilW{\widetilde{W}}
\newcommand\tilX{\widetilde{X}}

\newcommand\tilg{\widetilde{g}}
\newcommand\tilh{\widetilde{h}}

\newcommand\tilk{\widetilde{k}}
\newcommand\tilm{\widetilde{m}}

\newcommand\tilw{\widetilde{w}}

\renewcommand\a\alpha
\renewcommand\b\beta
\newcommand\g\gamma
\renewcommand\d\delta
\newcommand\D\Delta

\begin{document}
    \bibliographystyle{alpha}

\title[Deligne products and vertex operator algebras]{Deligne tensor products of categories of modules for vertex operator algebras}
 \author{Robert McRae}
\date{}

 \address{Yau Mathematical Sciences Center, Tsinghua University, Beijing 100084, China}
  \email{rhmcrae@tsinghua.edu.cn}
  

 \numberwithin{equation}{section}


\begin{abstract}
We show that if $\mathcal{U}$ and $\mathcal{V}$ are locally finite abelian categories of modules for vertex operator algebras $U$ and $V$, respectively, then the Deligne tensor product of $\mathcal{U}$ and $\mathcal{V}$ can be realized as a certain category $\mathcal{D}(\mathcal{U},\mathcal{V})$ of modules for the tensor product vertex operator algebra $U\otimes V$.  We also show that if $\mathcal{U}$ and $\mathcal{V}$ admit the braided tensor category structure of Huang--Lepowsky--Zhang, then $\mathcal{D}(\mathcal{U},\mathcal{V})$ does as well under mild additional conditions, and that this braided tensor structure is equivalent to the natural braided tensor structure on a Deligne tensor product category. These results hold in particular when $\mathcal{U}$ and $\mathcal{V}$ are the categories of $C_1$-cofinite $U$- and $V$-modules, if these categories are closed under contragredients, in which case we show that $\mathcal{D}(\mathcal{U},\mathcal{V})$ is the category of $C_1$-cofinite $U\otimes V$-modules. If $U$ and $V$ are $\mathbb{N}$-graded and $C_2$-cofinite, then we may take $\mathcal{U}$ and $\mathcal{V}$ to be the categories of all grading-restricted generalized $U$- and $V$-modules, respectively. Thus as an application, if the tensor categories of all modules for two $C_2$-cofinite vertex operator algebras are rigid, then so is the tensor category of all modules for the tensor product vertex operator algebra. We use this to prove that the representation categories of the even subalgebras of the symplectic fermion vertex operator superalgebras are non-semisimple modular tensor categories.
\end{abstract}

\maketitle

\tableofcontents

\allowdisplaybreaks

\section{Introduction}

Vertex operator algebras are algebraic structures which appear in several areas of mathematics and physics. They feature prominently among mathematically rigorous approaches to two-dimensional conformal quantum field theories, and through Huang's theorem \cite{Hu-rig-mod} that the representation category of a ``strongly rational'' vertex operator algebra is a semisimple modular tensor category, they also have connections to topological quantum field theories and $3$-manifold invariants. While Huang's theorem concerns vertex operator algebras with semisimple representation theory, non-semisimple categories of modules for vertex operator algebras are also of interest in logarithmic conformal field theory, $3$- and $4$-manifold invariants, and non-semisimple topological quantum field theories \cite{CR-review, HL-review, CG, FG, CDGG}. Understanding non-semisimple representation theory of a vertex operator algebra is difficult, however. For example, although Huang, Lepowsky, and Zhang \cite{HLZ1}-\cite{HLZ8} have found sufficient conditions for module categories for vertex operator algebras to admit natural braided tensor category structure, it is not known whether every vertex operator algebra has a non-trivial category of modules satisfying these conditions.

In this paper, we study tensor product vertex operator algebras. In particular, for vertex operator algebras $U$ and $V$, suppose $\cU$ and $\cV$ are categories of $U$-modules and $V$-modules which  admit the braided tensor category structure of Huang--Lepowsky--Zhang. Then one would expect there to be a braided tensor category of modules for the tensor product vertex operator algebra $U\otimes V$ which is determined somehow by $\cU$ and $\cV$, and whose braided tensor category structure is related in a natural way to that on $\cU$ and $\cV$. In fact, there is a natural braided tensor category determined by $\cU$ and $\cV$, namely, their Deligne tensor product \cite{De}; the question is whether this is equivalent to a category of $U\otimes V$-modules.

For any $\CC$-linear abelian categories $\cU$ and $\cV$, a Deligne tensor product of $\cU$ and $\cV$ is a $\CC$-linear abelian category $\cU\otimes\cV$ equipped with a bilinear functor $\otimes:\cU\times\cV\rightarrow\cC$ which is right exact in both variables, such that for any $\CC$-linear abelian category $\cC$, composition with $\otimes$ yields an equivalence between right exact $\CC$-linear functors $\cU\otimes\cV\rightarrow\cC$ and bilinear functors $\cU\times\cV\rightarrow\cC$ which are right exact in both variables. For example, if $\cU$ and $\cV$ are the categories of finite-dimensional modules for finite-dimensional algebras $A$ and $B$, respectively, then the category of finite-dimensional $A\otimes B$-modules is a Deligne tensor product of $\cU$ and $\cV$. If a Deligne tensor product of $\cU$ and $\cV$ exists, then it is unique up to equivalence. If $\cU$ and $\cV$ are locally finite, that is, morphism spaces are finite dimensional and all objects have finite length, then $\cU\otimes\cV$ indeed exists and is locally finite \cite{De, LF}.

Now if $\cU$ and $\cV$ are locally finite abelian categories of modules for vertex operator algebras $U$ and $V$, respectively, then one expects $\cU\otimes\cV$ to be equivalent to some category of $U\otimes V$-modules. Indeed, by \cite[Theorem 5.5]{CKM2}, if one of $\cU$ and $\cV$ is semisimple, then $\cU\otimes\cV$ is equivalent to the category of finite direct sums of $U\otimes V$-modules $M\otimes W$, where $M\in\mathrm{Ob}(\cU)$ and $W\in\mathrm{Ob}(\cV)$. However, if neither $\cU$ nor $\cV$ is semisimple, then this category of $U\otimes V$-modules need not be abelian. Thus we define a different category of $U\otimes V$-modules to obtain the first main result of this paper (see Theorem \ref{thm:D(U,V)_is_Del_prod} below):

\begin{thm}
Let $\cU$ and $\cV$ be locally finite abelian categories of grading-restricted generalized $U$- and $V$-modules, respectively, which are closed under subquotients, and define $\cD(\cU,\cV)$ to be the category of finite-length grading-restricted generalized $U\otimes V$-modules $X$ such that any vector in $X$ generates a $U$-submodule which is an object of $\cU$ and a $V$-submodule which is an object of $\cV$. Then $\cD(\cU,\cV)$ is a Deligne tensor product of $\cU$ and $\cV$, with $\otimes: \cU\times\cV\rightarrow\cD(\cU,\cV)$ given by the vector space tensor product of $U$- and $V$-modules.
\end{thm}

The key step to prove this theorem is showing that if $P$ is projective in $\cU$ and $Q$ is projective in $\cV$, then $P\otimes Q$ is projective in $\cD(\cU,\cV)$. This is proved in Theorem \ref{thm:proj_covers_in_C}, and then one can use properties of projective objects to show that $\cD(\cU,\cV)$ satisfies the universal property of the Deligne tensor product in the special case that $\cU$ and $\cV$ have enough projectives. To prove that $\cD(\cU,\cV)$ is a Deligne tensor product in general, we use this special case together with the fact (see \cite[Proposition 2.14]{De}) that any locally finite abelian category is the union of abelian full subcategories which have enough projectives.

Now suppose that $\cU$ and $\cV$ are locally finite categories of $U$- and $V$-modules which admit the vertex algebraic braided tensor category structure of \cite{HLZ1}-\cite{HLZ8}, with fusion product bifunctors $\tens_U: \cU\times\cU\rightarrow\cU$ and $\tens_V: \cV\times\cV\rightarrow\cV$. (We call $\tens_U$ and $\tens_V$ fusion products to avoid confusion with vector space tensor product $U\otimes V$-modules.) Then the universal property of the Deligne tensor product $\cD(\cU,\cV)$ implies that $\cD(\cU,\cV)$ is also a braided tensor category (see \cite[Proposition 5.17]{De}, \cite[Section 4.6]{EGNO}, or the discussion in Section \ref{subsec:prelim_tens_cat} below), such that the fusion product bifunctor $\tens:\cD(\cU,\cV)\times\cD(\cU,\cV)\rightarrow\cD(\cU,\cV)$ satisfies
\begin{equation}\label{eqn:fus_dist_over_tens_prod}
(M_1\otimes W_1)\tens(M_2\otimes W_2) = (M_1\tens_U M_2)\otimes(W_1\tens_V W_2)
\end{equation}
for $M_1, M_2\in\mathrm{Ob}(\cU)$ and $W_1,W_2\in\mathrm{Ob}(\cV)$. But it is not clear that this braided tensor structure on $\cD(\cU,\cV)$ agrees with that of \cite{HLZ1}-\cite{HLZ8}. In particular, if we take $\tens$ to be the vertex algebraic fusion product of $U\otimes V$-modules, defined as in \cite[Definition 4.15]{HLZ3}, then the relation \eqref{eqn:fus_dist_over_tens_prod} has been proved before only in cases where one of $\cU$ and $\cV$ is semisimple (see \cite[Lemma 2.16]{Lin}, \cite[Proposition 3.3]{CKLR}, and \cite[Theorem 5.2]{CKM2}).

In Theorem \ref{thm:dist_fus_over_tens}, we prove that the vertex algebraic fusion product $\tens$ of $U\otimes V$-modules satisfies \eqref{eqn:fus_dist_over_tens_prod} under quite general conditions; especially, we do not assume either $\cU$ or $\cV$ is semisimple. This is the key step for showing that, under mild conditions, the braided tensor category structure on the Deligne tensor product $\cD(\cU,\cV)$ is indeed the vertex algebraic braided tensor structure specified in \cite{HLZ1}-\cite{HLZ8}; see Theorem \ref{thm:D(U,V)_braid_equiv} for the detailed statement. We also remark that if $\cU$ and $\cV$ are rigid tensor categories, with duals given by the contragredient modules of \cite{FHL}, then $\cD(\cU,\cV)$ is also rigid.

 Now take $\cU$ and $\cV$ to be the categories $\cC_U^1$ and $\cC_V^1$ of $C_1$-cofinite $U$- and $V$-modules, respectively. The $C_1$-cofiniteness condition on modules for a vertex operator algebra was first introduced by Nahm \cite{Na}, and results of Huang \cite{Hu-diff-eqns} and Miyamoto \cite{Mi-C1-cofin} suggest that the category of $C_1$-cofinite modules for a vertex operator algebra is likely to admit the braided tensor category structure of \cite{HLZ1}-\cite{HLZ8} in general. In fact, from \cite{CJORY, CY}, $\cC_U^1$ and $\cC_V^1$ are indeed braided tensor categories if they are closed under contragredient modules \cite{FHL}, in which case they are also locally finite abelian categories (see Theorem \ref{thm:C1_cofin_prop} below). In Theorem \ref{thm:D(U,V)_for_C1_cofinite}, we show that if $U$ and $V$ are $\NN$-graded by conformal weights and if $\cC_U^1$ and $\cC_V^1$ are locally finite abelian categories, then $\cD(\cC_U^1,\cC_V^1)$ is precisely the category of $C_1$-cofinite $U\otimes V$-modules. Then Theorem \ref{thm:D(U,V)_braid_equiv} becomes (see Theorem \ref{thm:C1_UxV}):
\begin{thm}\label{thm:intro_C1_braid_tens}
If $U$ and $V$ are $\NN$-graded by conformal weights and $\cC_U^1$ and $\cC_V^1$ are closed under contragredients, then $\cC_{U\otimes V}^1$ admits the braided tensor category structure of \cite{HLZ1}-\cite{HLZ8} and is braided tensor equivalent to $\cC_U^1\otimes\cC_V^1$.
\end{thm}

In this theorem, we can replace $\cC_U^1$ and $\cC_V^1$ with locally finite tensor subcategories $\cU\subseteq\cC_U^1$ and $\cV\subseteq\cC_V^1$, in which case we replace $\cC_{U\otimes V}^1$ with the locally finite abelian full subcategory $\cD(\cU,\cV)\subseteq\cC_{U\otimes V}^1$ (see Corollary \ref{cor:C1_subcat_braid_tens}). If $U$ and $V$ satisfy the $C_2$-cofiniteness condition (originally introduced by Zhu \cite{Zh} to prove modular invariance of characters of modules for a vertex operator algebra),  then every grading-restricted generalized $U$- and $V$-module is $C_1$-cofinite. Thus Theorem \ref{thm:intro_C1_braid_tens} yields:
\begin{cor}\label{cor:intro_C2_cofin}
If $U$ and $V$ are $\NN$-graded $C_2$-cofinite vertex operator algebras, then the category of grading-restricted generalized $U\otimes V$-modules is braided tensor equivalent to the Deligne tensor product of the categories of grading-restricted generalized $U$- and $V$-modules.
\end{cor}

It is conjectured \cite{Hu-conj, CG, GR2} that the category of grading-restricted generalized modules for a simple $\NN$-graded self-contragredient $C_2$-cofinite vertex operator algebra is a (not necessarily semisimple) modular tensor category, generalizing Huang's result \cite{Hu-rig-mod} in the semisimple case. To prove this conjecture, by \cite[Main Theorem 1]{McR-rat}, it is enough to show that the tensor category of modules for such a vertex operator algebra is rigid. Thus by Corollary \ref{cor:intro_C2_cofin}, if the module categories for $U$ and $V$ are non-semisimple modular tensor categories, then so is the category of $U\otimes V$-modules. Further, by vertex operator algebra extension theory \cite{KO, HKL, CKM1}, if $A$ is a vertex operator algebra which contains $U\otimes V$ as a vertex operator subalgebra, then the category of grading-restricted generalized $A$-modules is also a non-semisimple modular tensor category under certain conditions (especially, the categorical dimension of $A$ in the category of $U\otimes V$-modules should be non-zero). This is in fact one of the main motivations of this paper, to make it possible to use extension theory to study the representations of a vertex operator algebra that contains a tensor product of well-understood vertex operator algebras as a subalgebra.

In Section \ref{sec:C2_cofinite_examples}, we look at a family of $C_2$-cofinite vertex operator algebras discussed in \cite[Section 4.1.2]{CKM1}. They are simple current extensions of tensor products of triplet $W$-algebras $\cW_p$ \cite{Ka, AM, TW}, and we show in Theorem \ref{thm:ext_of_Wp_tens} that their module categories are non-semisimple modular tensor categories. The best-known examples from this family are the even subalgebras $SF_d^+$ of the symplectic fermion superalgebras $SF_d$, $d\in\ZZ_+$, which were introduced in the physics literature \cite{Ka2, Ka3, GK} and are the affine vertex operator superalgebras associated to a $2d$-dimensional purely odd abelian Lie superalgebra equipped with a symplectic form. The algebra $SF_d^+$ is a simple current extension of $\cW_2^{\otimes d}$, with $SF_1^+\cong\cW_2$ itself \cite{Ka3, GK}. In Theorem \ref{thm:sym_ferm_properties}, we use the rigid tensor structure on the category of $\cW_2$-modules \cite{TW} (see also \cite{MY}), Corollary \ref{cor:intro_C2_cofin}, and vertex operator algebra extension theory to classify simple and projective $SF_d^+$-modules and to compute all fusion products involving simple and projective modules. Note that simple $SF_d^+$-modules were already classified in \cite{Ab}, and fusion rules involving simple $SF_d^+$-modules were calculated in \cite{AA}, though these fusion rules do not give the complete fusion products.

In \cite{Ru}, Runkel constructed a braided tensor category $\cS\cF_d$ which he conjectured to be braided tensor equivalent to the category of grading-restricted generalized $SF_d^+$-modules. The category $\cS\cF_d$ is braided tensor equivalent to the finite-dimensional representation category of a factorizable ribbon quasi-Hopf algebra \cite{GR, FGR}, and thus is a non-semisimple modular tensor category. So our result that the category of $SF_d^+$-modules is indeed a non-semisimple modular tensor category, stated as Corollary \ref{cor:symplectic_fermions} below, is consistent with Runkel's conjecture. We remark that this conjecture has recently been proved in the $d=1$ case \cite{GN, CLR}; it may be possible to use Corollary \ref{cor:symplectic_fermions} and Theorem \ref{thm:sym_ferm_properties} of the present work, together with the methods of \cite{CLR}, to prove Runkel's conjecture for all $d$.

There are also many non-$C_2$-cofinite vertex operator algebras which are known to satisfy the hypotheses of Theorem \ref{thm:intro_C1_braid_tens}, including all Virasoro vertex operator algebras \cite{CJORY}, many affine vertex operator algebras \cite{CY}, and the singlet $W$-algebras $\cM_p$, $p\in\ZZ_{\geq 2}$ \cite{CMY2, CMY-typ-sing}. Thus one can now use extension theory to study the representations of vertex operator (super)algebras which contain tensor products of such vertex operator algebras as subalgebras. For example, in upcoming joint work with Thomas Creutzig, Shashank Kanade, and Jinwei Yang, we plan to study the representation theory of many affine vertex operator superalgebras at level $1$, which are simple current extensions of tensor products of multiple copies of the singlet algebra $\cM_2$ and Heisenberg vertex operator algebras. 

Other interesting examples include the chiral universal centralizer algebras of \cite{Ar}, which contain tensor products of two affine $W$-algebras associated to a simple Lie algebra $\mathfrak{g}$ as subalgebras. For $\mathfrak{g}=\mathfrak{sl}_2$, the $W$-algebras are Virasoro vertex operator algebras. Thus once the Virasoro tensor categories from \cite{CJORY} are thoroughly understood, it should be possible to use Theorem \ref{thm:intro_C1_braid_tens} and extension theory to study the representation theory of chiral universal centralizer algebras of $\mathfrak{sl}_2$ at rational levels (see \cite{MY2} for the case of irrational levels and level $-1$, in which case the relevant Virasoro categories are semisimple). Since the chiral universal centralizer algebras are $\ZZ$-graded conformal vertex algebras (with infinite-dimensional conformal weight spaces and no lower bound on conformal weights), extension theory is perhaps the best way to study their representations; methods such as the Zhu algebra \cite{Zh}, which apply only to $\NN$-gradable modules, are not so useful here.

We note that almost all the results in this paper are, for simplicity, stated for tensor products of two vertex operator algebras, but these results apply to tensor products of any finite number $N$ of vertex operator algebras by induction on $N$.

\vspace{3mm}
 
 \noindent{\bf Acknowledgments.} I would like to thank Thomas Creutzig, Shashank Kanade, and Jinwei Yang for comments and discussions.

\section{Preliminaries}

In this section, we review basic definitions and results for (tensor) categories, and especially tensor categories of modules for vertex operator algebras.

\subsection{Locally finite abelian categories}\label{subsec:prelim_loc_fin}

For the definitions and properties here, we mainly use \cite{De, EGNO} as references. Recall that a category $\cC$ is $\CC$-linear additive if all morphism sets in $\cC$ are $\CC$-vector spaces such that composition of morphisms is bilinear, $\cC$ has a zero object, and every finite set of objects in $\cC$ has a direct sum. A functor between $\CC$-linear additive categories is $\CC$-linear if it induces $\CC$-linear maps on morphisms.

A $\CC$-linear additive category $\cC$ is \textit{abelian} if every morphism in $\cC$ has a kernel and cokernel, every monomorphism in $\cC$ is a kernel, and every epimorphism in $\cC$ is a cokernel. The category $\cC$ is in addition \textit{locally finite} if every morphism set in $\cC$ is a finite-dimensional $\CC$-vector space and every object in $\cC$ has finite length. A \textit{projective cover} of $W\in\mathrm{Ob}(\cC)$ is a surjection $p_W: P_W\twoheadrightarrow W$ such that $P_W$ is projective, and such that for any surjection $p: P\twoheadrightarrow W$ with $P$ projective, there is a surjection $f: P\twoheadrightarrow P_W$ such that the diagram
\begin{equation*}
\xymatrixcolsep{2pc}
\xymatrix{
& P \ar[d]^p  \ar[ld]_f\\
P_W \ar[r]_{p_W} & W \\
}
\end{equation*}
commutes. A locally finite $\CC$-linear abelian category $\cC$ is called \textit{finite} if $\cC$ has finitely many simple objects up to isomorphism and every simple object in $\cC$ has a projective cover.

\begin{remark}\label{rem:loc_fin_coalgebra}
Any essentially small locally finite $\CC$-linear abelian category $\cC$ is equivalent to the category of finite-dimensional comodules for some coalgebra $C$ \cite{Ta}; see also \cite[Section 1.10]{EGNO}. If $\cC$ is finite, then $C$ is finite dimensional, in which case $\cC$ is also equivalent to the category of modules for the algebra $C^*$.
\end{remark}

If $\cC$ is locally finite, then because every object in $\cC$ has finite length, every simple object of $\cC$ has a projective cover if and only if $\cC$ has \textit{enough projectives}, that is, every object of $\cC$ is the image of a morphism from a projective object, or equivalently every object in $\cC$ is a cokernel of some morphism between projective objects. The following proposition is elementary, and it implies that every projective object in a locally finite abelian category is a finite direct sum of projective covers:
\begin{prop}\label{prop:proj_cover_indecomposable}
Let $\cC$ be a locally finite $\CC$-linear abelian category, and let $W$ be a simple object of $\cC$. Then a surjection $p_W: P_W\twoheadrightarrow W$ is a projective cover of $W$ if and only if $P_W$ is an indecomposable projective object of $\cC$.
\end{prop}
\begin{proof}
If $p_W: P_W\twoheadrightarrow W$ is a projective cover of $W$, then $P_W$ is projective. To show $P_W$ is also indecomposable, suppose $P_W=P_1\oplus P_2$ with inclusions $i_1: P_1\rightarrow P$ and $i_2: P_2\rightarrow P$. Since $p_W\neq 0$, we have $p_W\circ i_k\neq 0$ for either $k=1$ or $k=2$. Without loss of generality assume $k=1$, so $p_W\circ i_1$ is surjective since $W$ is simple. Also, $P_1$ is projective since it is a direct summand of a projective object. Thus because $p_W: P_W\twoheadrightarrow W$ is a projective cover, there is a surjection $f: P_1\twoheadrightarrow P_W$ such that $p_W\circ f = p_W\circ i_1$. It follows that $P_1$ and $P_W$ have the same finite length, and thus $P_2$ has length $0$. Then $P_2=0$ and $P_W$ is indecomposable.

Conversely, suppose $p_W: P_W\twoheadrightarrow W$ is a surjection such that $P_W$ is projective and indecomposable. Then for any surjection $p: P\twoheadrightarrow W$ in $\cC$ with $P$ projective, there are morphisms $f: P\rightarrow P_W$ and $g: P_W\rightarrow P$ such that
\begin{equation*}
p_W\circ f =p,\qquad p\circ g = p_W,
\end{equation*}
because both $P_W$ and $P$ are projective. It follows that for any $N\in\NN$,
\begin{equation*}
p_W\circ(f\circ g)^N = p_W.
\end{equation*}
Thus $f\circ g$ is an endomorphism of $P_W$ which is not nilpotent since $p_W\neq 0$. Since $P_W$ is indecomposable and has finite length, Fitting's Lemma implies that $f\circ g$ is an isomorphism; in particular, $f$ is surjective. Thus $p_W: P_W\twoheadrightarrow W$ is a projective cover of $W$. 
\end{proof}

Any locally finite $\CC$-linear abelian category $\cC$ is the union of finite abelian full subcategories. Indeed, for any object $X$ in $\cC$, let $\langle X\rangle\subseteq\cC$ denote the full subcategory consisting of objects which are subquotients of $X^{\oplus n}$, $n\in\NN$. Then $\langle X\rangle$ is an abelian subcategory of $\cC$ with finitely many simple objects, and it also has enough projectives by \cite[Proposition 2.14]{De}. Thus $\langle X\rangle$ is a finite abelian subcategory, and any object of $\cC$ is contained in such a subcategory. So although $\cC$ may not have any non-zero projective objects, any object of $\cC$ (or indeed any finite collection of objects, or any morphism, in $\cC$) is contained in a subcategory which does have enough projectives.

Now let $\cC$ and $\cD$ be two $\CC$-linear abelian categories. A \textit{Deligne tensor product} of $\cC$ and $\cD$ \cite[Section 5]{De} (see also \cite[Section 1.11]{EGNO}) is a $\CC$-linear abelian category $\cC\otimes\cD$ equipped with a $\CC$-bilinear functor
\begin{equation*}
\otimes: \cC\times\cD\longrightarrow\cC\otimes\cD
\end{equation*}
which is right exact in both variables and satisfies the following universal property: For any $\CC$-linear abelian category $\cE$ and $\CC$-bilinear functor
\begin{equation*}
\cB: \cC\times\cD\longrightarrow\cE,
\end{equation*}
there is a unique (up to natural isomorphism) right exact $\CC$-linear functor
\begin{equation*}
\cF: \cC\otimes\cD\longrightarrow\cE
\end{equation*}
such that $\cF\circ\otimes = \cB$. Clearly $\cC\otimes\cD$ is unique up to equivalence if it exists. If $\cC$ and $\cD$ are locally finite, then indeed $\cC\otimes\cD$ exists, and it is also locally finite \cite[Proposition 5.13]{De}; see also \cite[Proposition 22]{LF} for a complete proof.

\begin{remark}\label{rem:Del_prod_union_of_finite}
If $\cC$ and $\cD$ are equivalent to the categories of finite-dimensional comodules for coalgebras $C$ and $D$, respectively, as in Remark \ref{rem:loc_fin_coalgebra}, then by \cite[Proposition 1.11.2]{EGNO}, we may take $\cC\otimes\cD$ to be the category of finite-dimensional $C\otimes D$-comodules.
\end{remark}

\begin{remark}
The Deligne tensor product of $\cC$ and $\cD$ is usually denoted $\cC\tens\cD$, but here we use the notation $\otimes$ instead of $\tens$ because we will soon consider locally finite $\CC$-linear abelian categories of modules for vertex operator algebras $U$ and $V$, and we will show that the Deligne tensor product of such categories is given by a certain category of modules for the tensor product vertex operator algebra $U\otimes V$. We will reserve $\tens$ to denote fusion products in braided monoidal (or tensor) categories of modules for vertex operator algebras.
\end{remark}

If $\cF:\cC\rightarrow\til{\cC}$ and $\cG: \cD\rightarrow\til{\cD}$ are right exact $\CC$-linear functors between locally finite $\CC$-linear abelian categories, then there is a unique (up to natural isomorphism) functor
\begin{equation*}
\cF\otimes\cG: \cC\otimes\cD\longrightarrow\til{\cC}\otimes\til{\cD}
\end{equation*}
such that the diagram
\begin{equation*}
\xymatrixcolsep{4pc}
\xymatrix{
\cC \times \cD \ar[r]^{\cF\times\cG} \ar[d]^{\otimes} & \til{\cC}\times\til{\cD} \ar[d]^{\otimes}\\
\cC\otimes\cD \ar[r]^{\cF\otimes\cG} & \til{\cC}\otimes\til{\cD}\\
}
\end{equation*}
commutes. Moreover, the Deligne tensor product is a commutative and associative operation. In particular, there is an equivalence $\sigma:\cC\otimes\cD\rightarrow\cD\otimes\cC$ such that the diagram
\begin{equation*}
\xymatrixcolsep{3pc}
\xymatrix{
\cC\times\cD \ar[r]^{\cong} \ar[d]^\otimes & \cD\times\cC \ar[d]^{\otimes}\\
\cC\otimes\cD \ar[r]^{\sigma} & \cD\otimes\cC \\
}
\end{equation*}
commutes, and there is an equivalence $\cA: \cC\otimes(\cD\otimes\cE)\rightarrow(\cC\otimes\cD)\otimes\cE$ such that the diagram
\begin{equation*}
\xymatrixcolsep{4pc}
\xymatrixrowsep{1.5pc}
\xymatrix{
\cC\times(\cD\times\cE) \ar[r]^{\cong} \ar[d]^{\Id_\cC\times\otimes} & (\cC\times\cD)\times\cE \ar[d]^{\otimes\times\Id_{\cE}}\\
\cC\times (\cD\otimes\cE) \ar[d]^{\otimes} & (\cC\otimes\cD)\times\cE \ar[d]^{\otimes}\\
\cC\otimes(\cD\otimes\cE) \ar[r]^{\cA} & (\cC\otimes\cD)\otimes\cE\\
}
\end{equation*}
commutes. To see the existence of $\cA$, one can use realizations of $\cC$, $\cD$, and $\cE$ as categories of finite-dimensional comodules for certain coalgebras, as in Remarks  \ref{rem:loc_fin_coalgebra} and \ref{rem:Del_prod_union_of_finite}. Thus from now on, we will suppress parentheses from the notation for Deligne tensor products of more than two locally finite $\CC$-linear abelian categories.

We end this subsection with a useful elementary lemma on right exact sequences; we provide a proof for completeness.
\begin{lemma}\label{lem:cokernel_of_tens_prod}
Let $\cC$, $\cD$, and $\cE$ be abelian categories, and let $\cB: \cC\times\cD\rightarrow\cE$ be a functor which is right exact in both variables. Then given two right exact sequences
\begin{equation*}
W \xrightarrow{f} \til{W} \xrightarrow{c} C\rightarrow 0,\qquad X\xrightarrow{g} \til{X}\xrightarrow{d} D\rightarrow 0
\end{equation*}
in $\cC$ and $\cD$, respectively, the following is a right exact sequence in $\cE$:
\begin{equation*}
\cB(W,\til{X})\oplus\cB(\til{W},X) \xrightarrow{F} \cB(\til{W},\til{X}) \xrightarrow{\cB(c,d)} \cB(C,D)\longrightarrow 0,
\end{equation*}
where $F=\cB(f,\Id_{\til{X}})\circ\pi_1+\cB(\Id_{\til{W}},g)\circ\pi_2$ and $\pi_1$, $\pi_2$ denote the obvious projections.
\end{lemma}
\begin{proof}
Since $c$ and $d$ are surjective and $\cB$ is right exact in both variables, $\cB(c,d)$ is also surjective. Also, $\cB(c,d)\circ F=0$, so $\mathrm{Im}\,F\subseteq\ker\,\cB(c,d)$. To show $\ker\,\cB(c,d)\subseteq\mathrm{Im}\,F$, consider the following commutative diagram with right exact rows and columns:
\begin{equation*}
\xymatrixcolsep{5pc}
\xymatrixrowsep{2pc}
\xymatrix{
\cB(W,X) \ar[r]^{\cB(\Id_W,g)} \ar[d]^{\cB(f,\Id_X)} & \cB(W,\til{X}) \ar[r]^{\cB(\Id_W,d)} 
\ar[d]^{\cB(f,\Id_{\til{X}})} & \cB(W, D) \ar[r] \ar[d]^{\cB(f,\Id_D)} & 0\\
\cB(\til{W},X) \ar[r]^{\cB(\Id_{\til{W}}, g)} \ar[d]^{\cB(c,\Id_X)} & \cB(\til{W},\til{X}) \ar[r]^{\cB(\Id_{\til{W}}, d)} \ar[d]^{\cB(c,\Id_{\til{X}})} & \cB(\til{W}, D) \ar[r] \ar[d]^{\cB(c,\Id_{D})} & 0\\
\cB(C, X) \ar[r]^{\cB(\Id_{C}, g)} \ar[d] & \cB(C, \til{X}) \ar[r]^{\cB(\Id_{C}, d)} \ar[d] & \cB(C, D) \ar[r] \ar[d] & 0\\
0 & 0 & 0 & \\
}
\end{equation*}
Then
\begin{align*}
\cB(\Id_{\til{W}}, d)\left(\ker\,\cB(c,d)\right) \subseteq\ker\,\cB(c,\Id_D) =\mathrm{Im}\,\cB(f,\Id_D)=\mathrm{Im}\,\cB(f,d).
\end{align*}
Assuming for convenience that objects of $\cE$ have the structure of vector spaces, this means that for any $b\in\ker\,\cB(c, d)$, there exists $b_1\in \cB(W,\til{X})$ such that
\begin{equation*}
\cB(\Id_{\til{W}}, d)(b) = \cB(\Id_{\til{W}}, d)\big(\cB(f,\Id_{\til{X}})(b_1)\big).
\end{equation*}
That is,
\begin{equation*}
b-\cB(f,\Id_{\til{X}})(b_1)\in\ker\,\cB(\Id_{\til{W}}, d)=\mathrm{Im}\,\cB(\Id_{\til{W}}, g),
\end{equation*}
which implies that $b\in\mathrm{Im}\,F$ as required.
\end{proof}

\subsection{Braided tensor categories}\label{subsec:prelim_tens_cat}

Let $\cC$ be a monoidal category, which means there is a functor $\tens: \cC\times\cC\rightarrow\cC$, a unit $\vac\in\mathrm{Ob}(\cC)$, left and right unit isomorphisms $l: \vac\tens\bullet\rightarrow\Id_\cC$ and $r: \bullet\tens\vac\rightarrow\Id_\cC$, and associativity isomorphisms $\cA: \tens\circ(\Id_\cC\times\tens)\rightarrow\tens\circ(\tens\times\Id_\cC)$ which satisfy the triangle and pentagon identites (see for example \cite[Definition 2.1.1]{EGNO}). If there is also a natural braiding isomorphism $\cR:\tens\rightarrow\tens\circ\sigma$ (where $\sigma$ exchanges the factors of $\cC\times\cC$) which satisfies the hexagon identities (see for example \cite[Definition 8.1.1]{EGNO}), then $\cC$ is a \textit{braided monoidal category}. We call the functor $\tens$ in a (braided) monoidal category the ``fusion product'' on $\cC$, to avoid confusion with tensor products of vector spaces and Deligne tensor products of categories. In this paper, all monoidal categories will be $\CC$-linear and additive, and $\tens$ will be $\CC$-bilinear and right exact in both variables. If such a (braided) monoidal category is abelian, we call it a \textit{(braided) tensor category}.

Unlike in \cite{EGNO}, we do not use ``tensor category'' to refer only to monoidal categories that are rigid. However, rigidity is an important property of many of the tensor categories that we will consider. An object $X$ in a monoidal category is \textit{(left) rigid} if it has a \textit{(left) dual}  $(X^*, e_X, i_X)$ where $X^*\in\mathrm{Ob}(\cC)$, $e_X: X^*\tens X\rightarrow\vac$ is the \textit{evaluation} morphism, $i_X:\vac\rightarrow X\tens X^*$ is the \textit{coevaluation} morphism, and the two rigidity compositions
\begin{align*}
X\xrightarrow{l_X} \vac\tens X \xrightarrow{i_X\tens\Id_X} & (X\tens X^*)\tens X\nonumber\\
&\xrightarrow{\cA_{X,X^*,X}^{-1}} X\tens(X^*\tens X)\xrightarrow{\Id_X\tens e_X} X\tens\vac\xrightarrow{r_X} X
\end{align*}
and
\begin{align*}
X^*\xrightarrow{r_{X^*}} X^*\tens\vac & \xrightarrow{\Id_{X^*}\tens i_X} X^*\tens(X\tens X^*)\nonumber\\
&\xrightarrow{\cA_{X^*,X,X^*}} (X^*\tens X)\tens X^*\xrightarrow{e_X\tens\Id_{X^*}} \vac\tens X^*\xrightarrow{l_{X^*}} X^*
\end{align*}
are both identities. There are analogous definitions of right duals and right rigidity, but we will only consider categories in which left and right duals may be taken to be the same, so we only consider left duals in this paper. A monoidal category $\cC$ is \textit{rigid} if every object has a dual. In a rigid monoidal category, duals induce a contravariant endofunctor with the dual $f^*$ of a morphism $f: W\rightarrow X$ given by the composition
\begin{align*}
X^*\xrightarrow{r_{X^*}} X^*\tens\vac  \xrightarrow{\Id_{X^*}\tens i_W} & X^*\tens(W\tens W^*) \xrightarrow{\cA_{X^*,W,W^*}} (X^*\tens W)\tens W^*  \nonumber\\
& \xrightarrow{(\Id_{X^*}\tens f)\tens \Id_{W^*})} (X^*\tens X)\tens W^* \xrightarrow{e_X\tens\Id_{W^*}} \vac\tens W^*\xrightarrow{l_{W^*}} W^*.
\end{align*}
A \textit{finite tensor category} is a finite abelian category which is also a rigid tensor category.

A \textit{braided ribbon category} is a rigid braided tensor category $\cC$ equipped with a natural \textit{twist} isomorphism $\theta: \Id_{\cC}\rightarrow\Id_{\cC}$ which satisfies $\theta_\vac=\Id_\vac$, $\theta_{X^*}=\theta_X^*$ for $X\in\mathrm{Ob}(\cC)$, and the \textit{balancing equation} 
\begin{equation}\label{eqn:balancing}
\theta_{X_1\tens X_2} =\cR^2_{X_1,X_2}\circ(\theta_{X_1}\tens\theta_{X_2})
\end{equation}
for $X_1, X_2\in\mathrm{Ob}(\cC)$, where $\cR^2_{X_1,X_2}=\cR_{X_2,X_1}\circ\cR_{X_1,X_2}$ is the double braiding. A \textit{modular tensor category} is a finite braided ribbon category whose braiding is non-degenerate, that is, if $\cR^2_{W,X}=\Id_{W\tens X}$
for all $X\in\mathrm{Ob}(\cC)$, then $W\cong\vac^{\oplus n}$ for some $n\in\NN$. Since modular tensor categories are often assumed to be semisimple as well, we will sometimes refer to ``not necessarily semisimple modular tensor categories'' or ``non-semisimple modular tensor categories'' to emphasize that we are not assuming semisimplicity here.

Now let $\cC$ and $\cD$ be two locally finite braided tensor categories. Then the Deligne tensor product $\cC\otimes\cD$ is also a locally finite braided tensor category (see \cite[Proposition 5.17]{De} or \cite[Section 4.6]{EGNO}). In more detail, since the functors $\tens_\cC: \cC\times\cC\rightarrow\cC$ and $\tens_\cD: \cD\times\cD\rightarrow\cD$ are right exact in both variables, they induce right exact functors
\begin{equation*}
T_\cC: \cC\otimes\cC\longrightarrow\cC,\qquad T_\cD:\cD\otimes\cD\longrightarrow\cD
\end{equation*}
such that $T_\cC\circ\otimes = \tens_\cC$ and $T_\cD\circ\otimes =\tens_\cD$. We then get a commutative diagram:
\begin{equation*}
\xymatrixcolsep{4.5pc}
\xymatrix{
\cC\times\cD\times\cC\times\cD  \ar[d]^{\otimes\times\otimes} \ar[r]^{\Id_\cC\times\sigma\times\Id_\cD} & \cC\times \cC\times\cD\times \cD \ar[d]^{\otimes\times\otimes} \ar[rd]^{\tens_\cC\times\tens_\cD} & \\
(\cC\otimes\cD)\times(\cC\otimes\cD) \ar[d]^\otimes & (\cC\otimes\cC)\times(\cD\otimes\cD) \ar[d]^\otimes \ar[r]^(.6){T_\cC\times T_\cD} & \cC\times\cD \ar[d]^\otimes\\
\cC\otimes\cD\otimes\cC\otimes\cD \ar[r]^{\Id_\cC\otimes\sigma\otimes\Id_\cD} & \cC\otimes\cC\otimes\cD\otimes\cD \ar[r]^(.6){T_\cC\otimes T_\cD}  & \cC\otimes\cD\\
}
\end{equation*} 
The fusion product $\tens$ on $\cC\otimes\cD$ is then the composition
\begin{equation*}
\tens =(T_\cC\otimes T_\cD)\circ(\Id_\cC\otimes\sigma\otimes\Id_\cD)\circ\otimes,
\end{equation*}
and the commutative diagram implies that for $M_1, M_2\in\mathrm{Ob}(\cC)$ and $W_1, W_2\in\mathrm{Ob}(\cD)$, 
\begin{equation}\label{eqn:fus_prod_in_Del_prod}
(M_1\otimes W_1)\tens(M_2\otimes W_2) = (M_1\tens_\cC M_2)\otimes(W_1\tens_\cD W_2).
\end{equation}
If $f_i: M_i\rightarrow\til{M}_i$ and $g_i: W_i\rightarrow\til{W}_i$ for $i=1,2$ are morphisms in $\cC$ and $\cD$, respectively, then
\begin{equation}\label{eqn:fus_prod_of_morph_in_Del_prod}
(f_1\otimes g_1)\tens(f_2\otimes g_2) =(f_1\tens_\cC f_2)\otimes(g_1\tens_\cD g_2).
\end{equation}
Since
\begin{equation}\label{eqn:homs_between_Del_prods}
\hom_{\cC\otimes\cD}(M_1\otimes W_1, M_2\otimes W_2)\cong\hom_{\cC}(M_1,M_2)\otimes_\CC \hom_\cD(W_1,W_2)
\end{equation}
(see \cite[Proposition 5.13]{De} or \cite[Proposition 1.11.2]{EGNO}), \eqref{eqn:fus_prod_of_morph_in_Del_prod} and $\CC$-bilinearity completely determine the fusion product of morphisms between objects in the image of $\otimes: \cC\times\cD\rightarrow\cC\otimes\cD$. The unit object of $\cC\otimes\cD$ is $\vac_\cC\otimes\vac_\cD$, and under the identification \eqref{eqn:fus_prod_in_Del_prod},
\begin{equation}\label{eqn:unit_in_Del_prod}
l_{M\otimes W} = l_M\otimes l_W,\qquad r_{M\otimes W} =r_M\otimes r_W
\end{equation}
for $M\in\mathrm{Ob}(\cC)$ and $W\in\mathrm{Ob}(\cD)$. Associativity and braiding isomorphisms are given by
\begin{equation}\label{eqn:assoc_in_Del_prod}
\cA_{M_1\otimes W_1,M_2\otimes W_2,M_3\otimes W_3}=\cA_{M_1,M_2,M_3}\otimes\cA_{W_1,W_2,W_3}
\end{equation}
and
\begin{equation}\label{eqn:braid_in_Del_prod}
\cR_{M_1\otimes W_1,M_2\otimes W_2} =\cR_{M_1,M_2}\otimes\cR_{W_1,W_2}
\end{equation}
for $M_1,M_2,M_3\in\mathrm{Ob}(\cC)$ and $W_1, W_2, W_3\in\mathrm{Ob}(\cD)$.

Let $\til{\cC\otimes\cD}\subseteq\cC\otimes\cD$ be the full subcategory whose objects are isomorphic to finite direct sums of Deligne tensor product objects $M\otimes W$. Because $\tens$ is $\CC$-bilinear, if $X_1=\bigoplus_{i_1\in I_1} M_{i_1}\otimes W_{i_1}$ and $X_2=\bigoplus_{i_2\in I_2} M_{i_2}\otimes W_{i_2}$, then we can identify
\begin{equation}\label{eqn:fus_prod_in_P}
X_1\tens X_2=\bigoplus_{(i_1,i_2)\in I_1\times I_2} (M_{i_1}\tens_\cC M_{i_2})\otimes(W_{i_1}\tens_\cD W_{i_2})
\end{equation}
via the natural isomorphism
\begin{equation*}
F_{X_1,X_2}=\sum_{(i_1,i_2)\in I_1\times I_2} q^{(i_1,i_2)}_{X_1\tens X_2}\circ(\pi_{X_1}^{i_1}\tens\pi_{X_2}^{i_2}),
\end{equation*}
where  $\pi^{i_k}_{X_k}: X_k\rightarrow M_{i_k}\otimes W_{i_k}$ for $k=1,2$ is the projection and $q_{X_1\tens X_2}^{(i_1,i_2)}$ is the inclusion into the direct sum on the right side of \eqref{eqn:fus_prod_in_P}. Thus $\til{\cC\otimes\cD}$ is a braided monoidal subcategory of $\cC\otimes\cD$. Moreover, the fusion product of morphisms in $\til{\cC\otimes\cD}$, as well as the unit, associativity, and braiding isomorphisms, are completely determined from \eqref{eqn:fus_prod_of_morph_in_Del_prod}, \eqref{eqn:unit_in_Del_prod}, \eqref{eqn:assoc_in_Del_prod}, and \eqref{eqn:braid_in_Del_prod} using the natural isomorphism $F$, the $\CC$-bilinearity of $\tens$, and the naturality of the unit, associativity, and braiding isomorphisms in $\cC\otimes\cD$.

If $M\in\mathrm{Ob}(\cC)$ and $W\in\mathrm{Ob}(\cD)$ are rigid with duals $M^*$ and $W^*$, respectively, then $M\otimes W$ in $\cC\otimes\cD$ is rigid with dual $M^*\otimes W^*$, evaluation $e_{M\otimes W}=e_M\otimes e_W$, and coevaluation $i_{M\otimes W}=i_M\otimes i_W$, because we can identify
\begin{align*}
(M^*\otimes W^*)\tens(M\otimes W) & =(M^*\tens_\cC M)\otimes(W^*\tens_\cC W),\nonumber\\
 (M\otimes W)\tens(M^*\otimes W^*) & =(M\tens_\cC M^*)\otimes(W\tens_\cC W^*).
\end{align*}
 Moreover, if $\cC$ and $\cD$ have ribbon twists $\theta_\cC$ and $\theta_{\cD}$, then $\cC\otimes\cD$ has a ribbon twist $\theta$ characterized by $\theta_{M\otimes W}=(\theta_\cC)_M\otimes(\theta_{\cD})_W$ for $M\in\mathrm{Ob}(\cC)$, $W\in\mathrm{Ob}(\cD)$.

\subsection{Vertex operator algebras and tensor categories}\label{subsec:prelim_VOA}

We use the definition of vertex operator algebra from \cite{FLM, LL}. In particular, a vertex operator algebra $V$ is a graded vector space $V=\bigoplus_{n\in\ZZ} V_{(n)}$ with vertex operator map
\begin{align*}
Y_V: V & \rightarrow (\mathrm{End}\,V)[[x,x^{-1}]]\nonumber\\
 v & \mapsto Y_V(v,x) =\sum_{n\in\ZZ} v_n\,x^{-n-1},
\end{align*}
vacuum vector $\vac\in V_{(0)}$ such that $Y_V(\vac,x)=\Id_V$, and conformal vector $\omega\in V_{(2)}$ whose vertex operator $Y_V(\omega,x)=\sum_{n\in\ZZ} L(n)\,x^{-n-2}$ generates a representation of the Virasoro Lie algebra on $V$. For $n\in\ZZ$, $V_{(n)}$ is the $L(0)$-eigenspace with eigenvalue $n$, and we say that any non-zero vector in $V_{(n)}$ has \textit{conformal weight} $n$.

Given a vertex operator algebra $V$, a \textit{weak $V$-module} is a module for $V$ considered as a vertex algebra, as in \cite[Definition 4.1.1]{LL}, with no conformal weight grading assumed. In particular, a weak $V$-module $W$ is a vector space equipped with a vertex operator
\begin{align*}
Y_W: V & \rightarrow (\mathrm{End}\,W)[[x,x^{-1}]]\nonumber\\
 v & \mapsto Y_W(v,x) =\sum_{n\in\ZZ} v_n\,x^{-n-1}.
\end{align*}
Among other axioms, the vertex operator satisfies $Y_W(\vac,x)=\Id_W$, and $Y_W(\omega,x)=\sum_{n\in\ZZ} L(n)\,x^{-n-2}$ generates an action of the Virasoro algebra on $W$. A \textit{generalized $V$-module}, as in \cite{HLZ1}, is a weak $V$-module $W$ which decomposes as the direct sum of generalized $L(0)$-eigenspaces, that is, $W=\bigoplus_{h\in\CC} W_{[h]}$ where $W_{[h]}$ is the generalized $L(0)$-eigenspace with generalized eigenvalue $h$. We call $h$ the conformal weight of non-zero vectors in $W_{[h]}$. A \textit{grading-restricted generalized $V$-module} is a generalized $V$-module $W$ such that for any $h\in\CC$, $\dim W_{[h]}<\infty$ and $W_{[h+n]}=0$ for all sufficiently negative $n\in\ZZ$. 

The conformal weights of any simple, or more generally indecomposable, generalized $V$-module $W$ are contained in a single coset of $\CC/\ZZ$; see \cite[Remark 2.20]{HLZ1}. Thus if $W$ is in addition grading restricted, its conformal weights are contained in $h+\NN$ for some $h\in\CC$. More generally, the conformal weights of any finite-length generalized $V$-module are contained in the union of finitely many cosets of $\CC/\ZZ$.

If $W=\bigoplus_{h\in\CC} W_{[h]}$ is a graded vector space, then $W'=\bigoplus_{h\in\CC} W_{[h]}^*$ is its \textit{graded dual}. If $W$ is a grading-restricted generalized $V$-module, then $W'$ also has a $V$-module structure, called the \textit{contragredient} \cite{FHL}, given by
\begin{equation}\label{eqn:contra}
\langle Y_{W'}(v,x)w', w\rangle = \langle w', Y_W(e^{xL(1)}(-x^{-2})^{L(0)} v,x^{-1})w\rangle
\end{equation}
for $v\in V$, $w\in W$, and $w'\in W'$.

To obtain braided tensor categories of modules for vertex operator algebras, we need the notion of intertwining operator from \cite{FHL}; see also \cite{HLZ2}.
\begin{defn}
Let $V$ be a vertex operator algebra, and let $W_1$, $W_2$, and $W_3$ be weak $V$-modules. An \textit{intertwining operator} of type $\binom{W_3}{W_1\,W_2}$ is a linear map
\begin{align*}
\cY: W_1\otimes W_2 & \rightarrow W_3[\log x]\lbrace x\rbrace\nonumber\\
w_1\otimes w_2 & \mapsto \cY(w_1,x)w_2=\sum_{h\in\CC}\sum_{k\in\NN} (w_1)^{\cY}_{h;k} w_2\,x^{-h-1}(\log x)^k
\end{align*}
which satisfies the following properties:
\begin{enumerate}
\item \textit{Lower truncation}: For all $w_1\in W_1$, $w_2\in W_2$, and $h\in\CC$, there exists $N\in\ZZ$ such that if $n\in\ZZ_{\geq N}$, then $(w_1)^\cY_{h+n;k} w_2=0$ for all $k\in\NN$.

\item The \textit{Jacobi identity}: For $v\in V$ and $w_1\in W_1$,
 \begin{align*}
  x_0^{-1}\delta\left(\frac{x_1-x_2}{x_0}\right)Y_{W_3}(v,x_1)\cY(w_1,x_2) & -x_0^{-1}\delta\left(\frac{-x_2+x_1}{x_0}\right)\cY(w_1,x_2)Y_{W_2}(v,x_1)\nonumber\\
  & = x_2^{-1}\delta\left(\frac{x_1-x_0}{x_2}\right)\cY(Y_{W_1}(v,x_0)w,x_2),
 \end{align*}
 where $\delta(x)=\sum_{n\in\ZZ} x^n$ is the formal delta function.
 
 \item The \textit{$L(-1)$-derivative property}: For $w_1\in W_1$,
 \begin{equation*}
  \dfrac{d}{dx}\cY(w_1,x)=\cY(L(-1)w_1,x).
 \end{equation*}
\end{enumerate}
\end{defn}

An intertwining operator $\cY$ of type $\binom{W_3}{W_1\,W_2}$ is \textit{surjective} if $W_3$ is spanned by the vectors $ (w_1)^{\cY}_{h;k} w_2$ for $w_1\in W_1$, $w_2\in W_2$, $h\in\CC$, and $k\in\NN$, though actually, if $\cY$ is surjective, then $W_3$ is already spanned by the vectors $(w_1)^\cY_{h;0} w_2$ for $w_1\in W_1$, $w_2\in W_2$, and $h\in\CC$ (see for example \cite[Lemma 2.2]{MY}). Equivalently, when $W_1$, $W_2$, and $W_3$ are generalized $V$-modules with conformal weight space decompositions, then $\cY$ is surjective if and only if $W_3$ is spanned by projections of $\cY(w_1,1)w_2$ to the conformal weight spaces of $W_3$, for $w_1\in W_1$, $w_2\in W_2$. Here, $\cY(\cdot,z)\cdot$ for $z\in\CC^\times$ denotes the \textit{$P(z)$-intertwining map} (in the terminology of \cite{HLZ3}) obtained by substituting the formal variable $x$ in $\cY$ with the complex number $z$, using some choice of branch of logarithm to evaluate powers of $\log z$ and non-integral powers of $z$. The image of a $P(z)$-intertwining map is contained in the \textit{algebraic completion}
\begin{equation*}
\overline{W}_3=\prod_{h\in\CC} (W_3)_{[h]}.
\end{equation*}
If $z$ is a positive real number, we always substitute $x\mapsto z$ in an intertwining operator using the real-valued branch of logarithm $\ln$.

Now let $\cC$ be a category of generalized modules for a vertex operator algebra $V$. A \textit{fusion product} of two objects $W_1$ and $W_2$ in $\cC$ is a pair $(W_1\tens W_2,\cY_{W_1,W_2})$, where $W_1\tens W_2\in\mathrm{Ob}(\cC)$ and $\cY_{W_1,W_2}$ is an intertwining operator of type $\binom{W_1\tens W_2}{W_1\,W_2}$, which satisfies the following universal property: For any $W_3\in\mathrm{Ob}(\cC)$ and intertwining operator $\cY$ of type $\binom{W_3}{W_1\,W_2}$, there is a unique $V$-module homomorphism $f: W_1\tens W_2\rightarrow W_3$ such that $f\circ\cY_{W_1,W_2}=\cY$. We use the term ``fusion product'' here rather than ``tensor product'' to avoid confusion with tensor products of vector spaces and Deligne tensor products of categories. If a fusion product of $W_1$ and $W_2$ in $\cC$ exists, then the fusion product intertwining operator $\cY_{W_1,W_2}$ is surjective \cite[Proposition 4.23]{HLZ3}.

If fusion products of all pairs of objects in $\cC$ exist, then fusion products define a $\CC$-bilinear functor $\tens: \cC\times\cC\rightarrow\cC$ which is right exact in both variables \cite[Proposition 4.26]{HLZ3}. The fusion product of two morphisms $f_1: W_1\rightarrow X_1$ and $f_2: W_2\rightarrow X_2$ in $\cC$ is the unique $V$-module homomorphism
\begin{equation*}
f_1\tens f_2: W_1\tens W_2\longrightarrow X_1\tens X_2,
\end{equation*}
induced by the universal property of $(W_1\tens W_2,\cY_{W_1,W_2})$, such that
\begin{equation}\label{eqn:fus_prod_of_morph_char}
(f_1\tens f_2)\circ\cY_{W_1,W_2} =\cY_{X_1,X_2}\circ(f_1\otimes f_2).
\end{equation}
We also get natural left and right unit isomorphisms
\begin{equation*}
l_W: V\tens W\longrightarrow W,\qquad r_W: W\tens V\longrightarrow W
\end{equation*}
for any object $W\in\mathrm{Ob}(\cC)$
such that 
\begin{align}\label{eqn:unit_char}
l_W(\cY_{V, W}(v, x)w)  = Y_W(v, x)w,\qquad r_W(\cY_{W,V}(w,x)v)  =e^{x L(-1)} Y_W(v, -x)w
\end{align}
for all $v\in V$ and $w\in W$. There is also a natural braiding isomorphism 
\begin{equation*}
\cR_{W_1,W_2}: W_1\tens W_2\longrightarrow W_2\tens W_1
\end{equation*}
for all $W_1, W_2\in\mathrm{Ob}(\cC)$ such that
\begin{equation}\label{eqn:braid_char}
\cR_{W_1,W_2}(\cY_{W_1,W_2}(w_1,x)w_2)=e^{xL(-1)}\cY_{W_2,W_1}(w_2,e^{\pi i} x)w_1
\end{equation}
for all $w_1\in W_1$ and $w_2\in W_2$. The category $\cC$ also has a ribbon twist given by $\theta_W=e^{2\pi i L(0)}$ for any $W\in\mathrm{Ob}(\cC)$.

Thus $\cC$ will be a braided monoidal category (and a braided tensor category if it is abelian) if it has a natural associativity isomorphism $\cA: \tens\circ(\Id_\cC\times\tens)\rightarrow\tens\circ(\tens\times\Id_{\cC})$ which satisfies the triangle, pentagon, and hexagon identities. Conditions for the existence and a construction of such an associativity isomorphism are given in \cite{HLZ6}-\cite{HLZ8}. When the conditions are satisfied, the associativity isomorphism
\begin{equation*}
\cA_{W_1,W_2,W_3}: W_1\tens(W_2\tens W_3)\longrightarrow(W_1\tens W_2)\tens W_3
\end{equation*}
for $W_1, W_2, W_3\in\mathrm{Ob}(\cC)$ is given by 
\begin{align}\label{eqn:assoc_char}
\overline{\cA_{W_1,W_2,W_3}} & \left(\cY_{W_1,W_2\tens W_3}(w_1, r_1)\cY_{W_2,W_3}(w_2,r_2)w_3\right)\nonumber\\
& = \cY_{W_1\tens W_2,W_3}(\cY_{W_1,W_2}(w_1,r_1-r_2)w_2, r_2)w_3
\end{align}
for any $w_1\in W_1$, $w_2\in W_2$, and $w_3\in W_3$, and for any choice of $r_1, r_2\in\RR_+$ such that $r_1>r_2>r_1-r_2$. We substitute positive real numbers for formal variables in intertwining operators using the branch of logarithm $\ln$, and it is assumed that the compositions of intertwining maps on both sides of \eqref{eqn:assoc_char} converge absolutely to elements of the algebraic completions $\overline{W_1\tens(W_2\tens W_3)}$ and $\overline{(W_1\tens W_2)\tens W_3}$, respectively. Then $\overline{\cA_{W_1,W_2,W_3}}$ denotes the natural extension of the graded linear map $\cA_{W_1,W_2,W_3}$ to algebraic completions.

Since the conditions in \cite{HLZ6}-\cite{HLZ8} for existence of the associativity isomorphisms in $\cC$ are rather technical, we mention some examples where they are satisfied. First, suppose that $V$ is $\NN$-graded and $C_2$-cofinite, that is, $\dim V/C_2(V)<\infty$ where 
\begin{equation*}
C_2(V)=\mathrm{span}\lbrace u_{-2} v\mid u,v\in V\rbrace.
\end{equation*}
It follows from \cite{Hu-C2} (see also the discussion in \cite[Section 3.1]{CM} and \cite[Lemma 2.10]{McR-rat}) that the category $\rep(V)$ of all grading-restricted generalized $V$-modules is a finite abelian category and a braided tensor category. In particular, if $V$ is $\NN$-graded and $C_2$-cofinite, then $\rep(V)$ is a finite braided ribbon category if and only if it is rigid. Huang showed in \cite{Hu-rig-mod} that if $V$ is in addition simple and self-contragredient, and if $\rep(V)$ is semisimple, then $\rep(V)$ is a modular tensor category. If $V$ is simple and self-contragredient but $\rep(V)$ is not semisimple, then it is not known in general whether $\rep(V)$ is rigid; however, if it is rigid, then it is a non-semisimple modular tensor category \cite[Main Theorem 1]{McR-rat}.

When $V$ is not necessarily $C_2$-cofinite, we consider the category $\cC_V^1$ of \textit{$C_1$-cofinite} $V$-modules, which satisfy $\dim W/C_1(W) <\infty$ where
\begin{equation*}
C_1(W)=\mathrm{span}\lbrace v_{-1} w\mid w\in W,\,v\in V,\,\mathrm{wt}\,v>0\rbrace.
\end{equation*}
Any $C_1$-cofinite $V$-module is finitely generated (see for example \cite[Proposition 2.1]{CMY1}). Specifically, if $W$ is $C_1$-cofinite, so that $W=T+C_1(W)$ for some finite-dimensional graded subspace $T$, then $W$ is spanned by vectors of the form
\begin{equation}\label{eqn:C1_span_set}
v^{(1)}_{-1}v^{(2)}_{-1}\cdots v^{(k)}_{-1} t
\end{equation}
for $t\in T$ and homogeneous $v^{(1)},v^{(2)},\ldots, v^{(k)}\in V$ of positive conformal weight; thus $W$ is generated by a finite spanning set for $T$. 

\begin{thm}\label{thm:C1_cofin_prop}
If $\cC_V^1$ is closed under contragredients, then:
\begin{enumerate}
\item Any submodule of a $C_1$-cofinite $V$-module is $C_1$-cofinite, and $\cC_V^1$ is a locally finite abelian category.

\item The category $\cC_V^1$ admits the vertex algebraic braided tensor category structure of \cite{HLZ1}-\cite{HLZ8}, with a ribbon twist.
\end{enumerate}
\end{thm}
\begin{proof}
For (1), it is easy to see that $\cC_V^1$ is closed under quotients and finite direct sums. For submodules, suppose $W\subseteq X$ where $X$ is $C_1$-cofinite. Then $W'$ is a quotient of $X'$, and $X'$ is $C_1$-cofinite because $\cC_V^1$ is closed under contragredients. Thus $W'$ is $C_1$-cofinite, and then so is $W\cong (W')'$. So $\cC_V^1$ is closed under submodules. Moreover, morphism spaces in $\cC_V^1$ are finite dimensional because any homomorphism $f: W\rightarrow X$ in $\cC_V^1$ is completely determined by its values on a finite generating set $S$ for $W$, and $f(S)$ is contained in the direct sum of finitely many of the finite-dimensional conformal weight spaces of $X$.

To complete the proof of (1), it remains to show that every $C_1$-cofinite $V$-module $W$ has finite length. We first prove that $W$ has a simple submodule; for this, we define a \textit{singular vector} to be a non-zero homogeneous $w\in W$ such that $v_n w= 0$ if $v_n$ is a vertex operator mode that lowers conformal weights. We claim $W$ has a singular vector of conformal weight $h$ for only finitely many $h\in\CC$. Indeed, since $W'$ is $C_1$-cofinite by hypothesis, there is a finite-dimensional graded subspace $T\subseteq W'$ such that $W'=T+C_1(W)$, and $W'$ is spanned by vectors as in \eqref{eqn:C1_span_set}. Then if $w\in W$ is a singular vector,
\begin{equation*}
\langle v^{(1)}_{-1}v^{(2)}_{-1}\cdots v^{(k)}_{-1} t, w\rangle = 0
\end{equation*}
unless $k=0$, since $v^{(1)}_{-1}$ raises conformal weight and hence its adjoint (from \eqref{eqn:contra}) is a linear combination of vertex operator modes that lower conformal weight. It follows that the conformal weight of $w$ is one of the finitely many conformal weights occurring in $T$.

Let $h$ have maximal real part such that $W$ has a singular vector of conformal weight $h$. The space of singular vectors of conformal weight $h$ is a finite-dimensional module for the Zhu algebra $A(V)$ defined in \cite{Zh}. Thus this space of singular vectors contains a simple $A(V)$-submodule, and this simple $A(V)$-submodule generates a simple $V$-submodule of $W$ (it is simple because any non-zero proper submodule would contain a singular vector whose weight would have greater real part than $h$). Thus any $C_1$-cofinite $V$-module contains a simple submodule when $\cC_V^1$ is closed under contragredients.

As a result, for any $C_1$-cofinite $W$, we can obtain an ascending chain of submodules
\begin{equation*}
W_1\subseteq W_2\subseteq W_3\subseteq\cdots\subseteq W
\end{equation*}
such that  $W_{i+1}/W_i$ is a simple submodule of the $C_1$-cofinite $V$-module $W/W_i$ if $W_i\neq W$, and $W_{i+1}=W$ otherwise. Let $X=\bigcup_{i\geq 1} W_i$; it is $C_1$-cofinite and thus finitely generated because $\cC_V^1$ is closed under submodules. Since the finitely many generators of $X$ are contained in some $W_i$, the ascending chain of submodules stabilizes, and thus $W_i=W$ for all sufficiently large $i$. So $W$ has finite length, completing the proof of part (1) of the theorem.

Part (2) was essentially proved in \cite[Section 4]{CJORY} and \cite[Section 3]{CY}; see also \cite[Theorem 2.3]{McR-cosets}.
\end{proof}

\begin{remark}\label{rem:RepV=C1_for_C2}
If $V$ is $\NN$-graded and $C_2$-cofinite, then any grading-restricted generalized $V$-module has finite length \cite{Hu-C2} and thus is finitely generated. Then the spanning set of \cite[Lemma 2.4]{Mi-mod-inv} for singly-generated weak $V$-modules implies that every $V$-module is $C_1$-cofinite. Thus $\rep(V)=\cC_V^1$ for $\NN$-graded $C_2$-cofinite $V$.
\end{remark}

\begin{remark}\label{rem:C1_BTC_examples}
Non-$C_2$-cofinite vertex operator algebras $V$ such that $\cC_V^1$ is closed under contragredients, and thus satisfies the conclusions of Theorem \ref{thm:C1_cofin_prop}, include all Virasoro vertex operator algebras \cite{CJORY}, many affine vertex operator (super)algebras \cite{CY}, and the singlet vertex operator algebras $\cM(p)$ for $p\in\mathbb{Z}_{\geq 2}$ \cite{CMY-typ-sing}.
\end{remark}

\subsection{Tensor product vertex operator algebras}\label{subsec:prelim_tens_prod_VOAs}

We now discuss modules for tensor product vertex operator algebras; for simplicity, we focus on tensor products of two vertex operator algebras, but everything can be iterated to yield results for tensor products of more than two vertex operator algebras. Thus let $U$ and $V$ be vertex operator algebras. Then $U\otimes V$ is also a vertex operator algebra with vertex operator given by
\begin{equation*}
Y_{U\otimes V}(u_1\otimes v_1, x)(u_2\otimes v_2) = Y_U(u_1,x)u_2\otimes Y_V(v_1,x)v_2,
\end{equation*}
vacuum vector $\vac=\vac_U\otimes\vac_V$, and conformal vector $\omega=\omega_U\otimes\vac_V+\vac_U\otimes\omega_V$. Similarly, if $M$ is a generalized $U$-module and $W$ is a generalized $V$-module, then $M\otimes W$ is a generalized $U\otimes V$-module. By \cite[Theorem 4.7.4]{FHL}, every simple grading-restricted $U\otimes V$-module is the vector space tensor product of a simple $U$-module and a simple $V$-module.

Any generalized $U\otimes V$-module $X$ restricts to both a weak $U$-module and a weak $V$-module, with commuting actions of $U$ and $V$. We use $L_U(n)$ and $L_V(n)$ to denote the commuting Virasoro operators on $X$ associated to the conformal vectors $\omega_U$ and $\omega_V$, respectively, while $L(n)$ denotes the Virasoro operator associated to $\omega$.

If $\cU$ and $\cV$ are locally finite abelian categories of $U$-modules and $V$-modules, respectively, then there is a right exact functor from the Deligne tensor product $\cU\otimes\cV$ to the category of all generalized $U\otimes V$-modules which sends Deligne tensor product objects in $\cU\otimes\cV$ to vector space tensor products of $U$-modules and $V$-modules. This explains our use of $\otimes$ to denote both Deligne and vector space tensor products. Indeed, a major goal of this paper is to show that the Deligne tensor product $\cU\otimes\cV$ can be identified with a suitable locally finite abelian category of $U\otimes V$-modules. To prove this, we will need in particular that homomorphisms between vector space tensor product $U\otimes V$-modules come from tensor products of $U$-module and $V$-module homomorphisms (recall \eqref{eqn:homs_between_Del_prods}):
\begin{prop}\label{prop:proj_cover_homs}
 Let $M_1$, $M_2$ be weak $U$-modules and let $W_1$, $W_2$ be weak $V$-modules. Then the natural linear map
 \begin{equation*}
  \hom_U(M_1,M_2)\otimes\hom_V(W_1,W_2)\longrightarrow\hom_{U\otimes V}(M_1\otimes W_1,M_2\otimes W_2)
 \end{equation*}
is injective. It is also surjective if $M_1$, $M_2$ are grading-restricted generalized $U$-modules and $\hom_V(W_1,W_2)$ is finite dimensional.
\end{prop}
\begin{proof}
To show that the natural linear map is injective, we show that if $\sum_i f_i\otimes g_i$ vanishes as a $U\otimes V$-homomorphism $M_1\otimes W_1\rightarrow M_2\otimes W_2$, then it also vanishes as a vector in $\hom_U(M_1,M_2)\otimes\hom_V(W_1,W_2)$. We may assume that the morphisms $g_i: W_1\rightarrow W_2$ are linearly independent, and then it suffices to show that each $f_i=0$. In fact, for any $m_1\in M_1$, $w_1\in W_1$, and $m_2^*\in M_2^*$, we get
 \begin{align*}
  \sum_i \langle m_2^*, f_i(m_1)\rangle g_i(w_1) &=(m_2^*\otimes\id_{W_2})\bigg(\sum_i f_i(m_1)\otimes g_i(w_1)\bigg)\nonumber\\
  &=(m_2^*\otimes\id_{W_2})\bigg(\bigg(\sum_i f_i\otimes g_i\bigg)(m_1\otimes w_1)\bigg) =0.
 \end{align*}
Thus $\sum_i \langle m_2^*,f_i(m_1)\rangle g_i=0$ for any $m_2^*\in M_2^*$ and $m_1\in M_1$; since the $g_i$ are linearly independent, this means $\langle m_2^*,f_i(m_1)\rangle=0$ for each $i$. Then each $f_i=0$ since all of its matrix coefficients vanish. This proves the desired injectivity.

To prove that the natural linear map is also surjective when $M_1$, $M_2$ are grading-restricted generalized $U$-modules and $\hom_V(W_1,W_2)$ is finite dimensional, let $F: M_1\otimes W_1\rightarrow M_2\otimes W_2$ be a $U\otimes V$-module homomorphism, and let $\lbrace g_i\rbrace$ be a (finite) basis of $\hom_{V}(W_1,W_2)$. Then for any $m_2'\in M_2'$ and $m_1\in M_1$, we can write
\begin{equation*}
 (m_2'\otimes\id_{W_2})\circ F(m_1\otimes\bullet)=\sum_i c_i(m_2'\otimes m_1)\cdot g_i
\end{equation*}
for certain $c_i\in(M_2'\otimes M_1)^*$. Since $F$ is in particular a $U$-homomorphism that commutes with $L_U(0)$, and since the $g_i$ are linearly independent, $c_i(m_2'\otimes m_1)=0$ if $m_2'$ and $m_1$ have different conformal weights. Thus for each $i$ and each $m_1\in M_1$, $c_i(\bullet\otimes m_1)$ defines an element of $(M_2')'\cong M_2$, and we get a linear map $f_i: M_1\rightarrow M_2$ characterized by
\begin{equation*}
 \langle m_2',f_i(m_1)\rangle=c_i(m_2'\otimes m_1)
\end{equation*}
for $m_2'\in M_2'$, $m_1\in M_1$. We show that each $f_i$ is a $U$-module homomorphism: for $u\in U$ and $n\in\ZZ$,
\begin{align*}
\sum_i \langle m_2', f_i(u_n m_1)\rangle g_i & =\sum_i c_i(m_2'\otimes u_n m_1)\cdot g_i =(m_2'\otimes\id_{W_2})\circ F(u_n m_1\otimes\bullet)\nonumber\\
& = (m_2'\otimes\id_{W_2})\circ(u_n\otimes\id_{W_2})\circ F(m_1\otimes\bullet)\nonumber\\
&= (u_n^o m_2'\otimes\id_{W_2})\circ F(m_1\otimes\bullet) = \sum_i\langle u_n^o m_2',f_i(m_1)\rangle g_i\nonumber\\
& =\sum_i\langle m_2',u_n f_i(m_1)\rangle g_i
\end{align*}
where $u_n^o$ is the operator on $M_2'$ adjoint to the operator $u_n$ on $M_2$.
Since the $g_i$ are linearly independent and $m_2'\in M_2'$ is arbitrary, it follows that $f_i(u_n m_1)=u_nf_i(m_1)$ for each $m_1\in M_1$, that is, each $f_i$ is a $U$-module homomorphism. Finally, we show that $F=\sum_i f_i\otimes g_i$: for all $m_1\in M_1$, $w_1\in W_1$, $m_2'\in M_2'$, and $w_2^*\in W_2^*$, we have
\begin{align*}
 \langle m_2'\otimes w_2^*, F(m_1 & \otimes w_1)\rangle  = \langle w_2^*,(m_2'\otimes\id_{W_2})\circ F(m_1\otimes w_1)\rangle\nonumber\\
 & =\sum_i \langle m_2', f_i(m_1)\rangle\langle w_2^*, g_i(w_1)\rangle =\sum_i\langle m_2'\otimes w_2^*,(f_i\otimes g_i)(m_1\otimes w_1)\rangle,
\end{align*}
as required.
\end{proof}

\section{Deligne tensor products}

Let $U$ and $V$ be vertex operator algebras. In this section, we show that the Deligne tensor product of locally finite abelian categories of $U$-modules and $V$-modules can be identified with a certain category of $U\otimes V$-modules.

\subsection{The \texorpdfstring{$U\otimes V$}{U x V}-module category \texorpdfstring{$\cD(\cU,\cV)$}{D(U,V)}}

If $\cU$ and $\cV$ are locally finite abelian categories of $U$- and $V$-modules, respectively, and at least one of $\cU$ and $\cV$ is semisimple, then \cite[Theorem 5.5]{CKM2} shows that their Deligne tensor product is equivalent to the category of $U\otimes V$-modules which are isomorphic to finite direct sums of $M\otimes W$ for $M\in\mathrm{Ob}(\cU)$, $W\in\mathrm{Ob}(\cV)$. However, if neither $\cU$ nor $\cV$ is semisimple, then this category of $U\otimes V$-modules is not necessarily abelian. Thus in general, we need a different category of $U\otimes V$-modules:
\begin{defn}\label{def:D(U,V)}
Let $\cU$ and $\cV$ be additive categories of grading-restricted generalized $U$-modules and $V$-modules, respectively. Then define $\cD(\cU,\cV)$ to be the category of finite-length grading-restricted generalized $U\otimes V$-modules $X$ such that the $U$-submodule generated by any vector in $X$ is an object of $\cU$, and the $V$-submodule generated by any vector in $X$ is an object of $\cV$.
\end{defn}

From the definition, it is clear that $\cD(\cU,\cV)$ is closed under submodules and finite direct sums. It is also closed under quotients if $\cU$ and $\cV$ are. Since every object of $\cD(\cU,\cV)$ has finite length, every object of $\cD(\cU,\cV)$ is finitely generated, as well as grading restricted; thus morphism spaces in $\cD(\cU,\cV)$ are finite dimensional. Thus if $\cU$ and $\cV$ are additive categories of grading-restricted generalized $U$- and $V$-modules which are closed under quotients, then $\cD(\cU,\cV)$ is a locally finite abelian category.

\begin{prop}\label{prop:D(U,V)_surjections}
Any object of $\cD(\cU,\cV)$ is the homomorphic image of a finite direct sum of modules $M\otimes W$ such that $M\in\mathrm{Ob}(\cU)$ and $W\in\mathrm{Ob}(\cV)$.
\end{prop}
\begin{proof}
Any $X\in\mathrm{Ob}(\cD(\cU,\cV))$ is finitely generated, so it is enough to consider the case that $X$ is generated by a single vector $b$. Let $M$ be the $U$-submodule generated by $b$, and let $W$ be the $V$-submodule generated by $b$. By the definition of $\cD(\cU,\cV)$, $M$ is an object of $\cU$ and $W$ is an object of $\cV$. Since $M$ is spanned by the vectors $u_n b$ for $u\in U$ and $n\in\ZZ$ (see \cite[Proposition 4.5.6]{LL}), we attempt to define a surjection
\begin{align*}
p: M\otimes W & \rightarrow X\\
\sum_i u^{(i)}_{n_i} b\otimes w & \mapsto \sum_i u^{(i)}_{n_i} w.
\end{align*}
For $p$ to be well defined, we need to show that if $\sum_i u^{(i)}_{n_i} b=0$, then also $\sum_i u^{(i)}_{n_i} w=0$ for all $w\in W$. Indeed, $w=\sum_j v^{(j)}_{m_j} b$ for certain $v^{(j)}\in V$, $m_j\in\ZZ$, so if $\sum_i u^{(i)}_{n_i} b=0$, then
\begin{equation}\label{eqn:p_well_defined}
\sum_i u^{(i)}_{n_i} w =\sum_{i,j} u^{(i)}_{n_i} v^{(j)}_{m_j} b = \sum_{i,j} v^{(j)}_{m_j} u^{(i)}_{n_i} b =0
\end{equation}
since the actions of $U$ and $V$ on $X$ commute. 

We still need to show that $p$ is a $U\otimes V$-module homomorphism; it is enough to show that $p$ is both a $U$-module and a $V$-module homomorphism. First, for $v\in V$, $n\in\ZZ$, $m=\sum_i u^{(i)}_{n_i} b\in M$, and $w\in W$, we have
\begin{equation*}
v_n p(m\otimes w) =\sum_i v_n u^{(i)}_{n_i} w =\sum_i u^{(i)}_{n_i} v_n w= p(m\otimes v_n w),
\end{equation*}
so $p$ is a $V$-module homomorphism. Then writing $w=\sum_j v^{(j)}_{m_j} b$, the calculation in \eqref{eqn:p_well_defined} shows that $p(m\otimes w)=\sum_j v^{(j)}_{m_j} m$. Thus for $u\in U$ and $n\in\ZZ$,
\begin{equation*}
u_n p(m\otimes w) =\sum_j u_n v^{(j)}_{m_j} m = \sum_j v^{(j)}_{m_j} u_n m = p(u_n m\otimes w),
\end{equation*}
so $p$ is also a $U$-module homomorphism.
\end{proof}

\begin{prop}\label{prop:tens_prod_in_D(U,V)}
Suppose that $\cU$ and $\cV$ are closed under subquotients and that all objects in $\cU$ and $\cV$ have finite length. If $M\in\mathrm{Ob}(\cU)$ and $W\in\mathrm{Ob}(\cV)$, then $M\otimes W\in\mathrm{Ob}(\cD(\cU,\cV))$. Conversely, if $M\otimes W$ is an object of $\cD(\cU,\cV)$ such that $M$ and $W$ are finitely generated, then $M\in\mathrm{Ob}(\cU)$ and $W\in\mathrm{Ob}(\cV)$.
\end{prop}
\begin{proof}
If $M\in\mathrm{Ob}(\cU)$ and $W\in\mathrm{Ob}(\cV)$, then $M\otimes W$ has finite length, equal to the product of the lengths of of $M$ and $W$. Now for any $b=\sum_i m_i\otimes w_i\in M\otimes W$, let $\langle b\rangle_U$ be the $U$-submodule of $M\otimes W$ generated by $b$. Then, assuming that the $w_i$ are linearly independent,
\begin{equation*}
\langle b\rangle_U \subseteq\sum_i \langle m_i\rangle_U\otimes w_i \cong\bigoplus_i \langle m_i\rangle_U.
\end{equation*} 
Since $\cU$ is closed under submodules and finite direct sums, $\langle b\rangle_U\in\mathrm{Ob}(\cU)$. Similarly, the $V$-submodule of $M\otimes W$ generated by $b$ is an object of $\cV$, showing that $M\otimes W\in\mathrm{Ob}(\cD(\cU,\cV))$.

Conversely, if $M\otimes W\in\mathrm{Ob}(\cD(\cU,\cV))$, then for any non-zero $m\in M$ and $w\in W$, $m\otimes w$ generates a $U$-submodule which is isomorphic to a submodule of $M$ and is also an object of $\cU$. Since $M$ is finitely generated, it is thus the homomorphic image of the direct sum of finitely many of these submodules. Then $M\in\mathrm{Ob}(\cU)$ because $\cU$ is closed under finite direct sums and quotients, and $W\in\mathrm{Ob}(\cV)$ similarly.
\end{proof}

Let $\til{\cD}(\cU,\cV)$ be the additive full subcategory of $\cD(\cU,\cV)$ consisting of modules which are isomorphic to finite direct sums of $M\otimes W$ for $M\in\mathrm{Ob}(\cU)$, $W\in\mathrm{Ob}(\cV)$. 
\begin{cor}\label{cor:cokernel_of_D_tilde}
If $\cU$ and $\cV$ are closed under subquotients and all objects of $\cU$ and $\cV$ have finite length, then every object in $\cD(\cU,\cV)$ is a cokernel of a morphism in $\til{\cD}(\cU,\cV)$.
\end{cor}
\begin{proof}
For any $X\in\mathrm{Ob}(\cD(\cU,\cV))$, there is a surjection $p: \bigoplus_i M_i\otimes W_i \rightarrow X$ by Proposition \ref{prop:D(U,V)_surjections}, where $\bigoplus_i M_i\otimes W_i\in\mathrm{Ob}(\til{\cD}(\cU,\cV)$ by Proposition \ref{prop:tens_prod_in_D(U,V)}. Then since $\cD(\cU,\cV)$ is closed under submodules, there is also a surjection $q: \bigoplus_j \til{M}_j\otimes\til{W}_j\rightarrow \ker p$, where again $\bigoplus_j \til{M}_j\otimes\til{W}_j\in\mathrm{Ob}(\til{\cD}(\cU,\cV))$. Thus $(X,p)$ is a cokernel of $q$.
\end{proof}

The conditions of the preceding corollary imply that $\cU$ and $\cV$ are locally finite abelian categories, in which case $\cD(\cU,\cV)$ is also locally finite abelian. In this setting, we can determine all simple objects in $\cD(\cU,\cV)$:
\begin{prop}\label{prop:D(U,V)_simple}
If $\cU$ and $\cV$ are closed under subquotients and all objects of $\cU$ and $\cV$ have finite length, then a $U\otimes V$-module is a simple object of $\cD(\cU,\cV)$ if and only if it is isomorphic to $M\otimes W$ for simple $M\in\mathrm{Ob}(\cU)$, $W\in\mathrm{Ob}(\cV)$.
\end{prop}
\begin{proof}
If $M\in\mathrm{Ob}(\cU)$ and $W\in\mathrm{Ob}(\cV)$ are simple, then $M\otimes W$ is a simple $U\otimes V$-module by \cite[Proposition 4.7.2]{FHL} and an object of $\cD(\cU,\cV)$ by Proposition \ref{prop:tens_prod_in_D(U,V)}. Conversely, by \cite[Theorem 4.7.4]{FHL}, any simple object of $\cD(\cU,\cV)$ is isomorphic to $M\otimes W$ where $M$ is a simple $U$-module and $W$ is a simple $V$-module. Then $M\in\mathrm{Ob}(\cU)$ and $W\in\mathrm{Ob}(\cV)$ by Proposition \ref{prop:tens_prod_in_D(U,V)}.
\end{proof}

Under mild conditions, $\cD(\cU,\cV)$ is closed under contragredients if $\cU$ and $\cV$ are:
\begin{prop}\label{prop:D(U,V)_contra}
If $\cU$ and $\cV$ are closed under subquotients and contragredients, then $\cD(\cU,\cV)$ is closed under contragredients.
\end{prop}
\begin{proof}
For $X\in\mathrm{Ob}(\cD(\cU,\cV))$, its contragredient $X'$ is a finite-length grading-restricted generalized $U\otimes V$-module whose composition factors are the contragredients of the composition factors of $X$. It remains to show that any $b'\in X'$ generates a $U$-submodule $\langle b'\rangle_U$ and $V$-submodule $\langle b'\rangle_V$ which are objects of $\cU$ and $\cV$, respectively. 
 
 Since the operators $L(0)$, $L_U(0)$, and $L_V(0)$ on $X'$ all commute, each finite-dimensional generalized $L(0)$-eigenspace of $X'$ is the direct sum of simultaneous generalized eigenspaces for $L_U(0)$ and $L_V(0)$. Thus $b'=\sum_i b_i'$ where each $b_i'$ is homogeneous for both $L_U(0)$ and $L_V(0)$. Then $\langle b'\rangle_U$ is a subquotient of $\bigoplus_i \langle b_i'\rangle_U$, so because $\cU$ is closed under subquotients and finite direct sums, $\langle b'\rangle_U\in\mathrm{Ob}(\cU)$ if each $\langle b_i'\rangle_U\in\mathrm{Ob}(\cU)$. Thus we are reduced to the case that $b'$ is a generalized eigenvector for both $L_U(0)$ and $L_V(0)$ with generalized eigenvalues $h_U$ and $h_V$, respectively. In this case, $\langle b'\rangle_U$ is a generalized $U$-module with $L_U(0)$-weights contained in $h_U+\ZZ$, where for $n\in\ZZ$, the $L_U(0)$-weight space of generalized $L_U(0)$-eigenvalue $h_U+n$ is contained in the $L(0)$-weight space $(X')_{[h_U+h_V+n]}=X_{[h_U+h_V+n]}^*$. Since $X$ is a grading-restricted generalized $U\otimes V$-module, $\langle b'\rangle_U$ is a grading-restricted generalized $U$-module. Thus $\langle b'\rangle_U$ has a $U$-module contragredient $\langle b'\rangle_U'$.
 
 Now, the $U$-module inclusion $i: \langle b'\rangle_U\rightarrow X'$ dualizes to a $U$-module map $i': X\rightarrow\langle b'\rangle_U'$ such that for $b\in X$ and $m\in\langle b'\rangle_U$,
 \begin{equation*}
 \langle i'(b), m\rangle =\langle i(m), b\rangle.
 \end{equation*}
Since there exists $b\in X$ such that $\langle i(m),b\rangle\neq 0$ for any non-zero $m\in\langle b'\rangle_U\subseteq X'$, the annihilator in $\langle b'\rangle_U$ of $\mathrm{Im}\,i'\subseteq \langle b'\rangle_U'$ is $0$. This implies $i'$ is surjective because $\langle b'\rangle_U$ is grading restricted. Then since $X$ has finite length, any finite $U\otimes V$-module filtration
\begin{equation*}
0=X_0\subseteq X_1\subseteq X_2\subseteq\cdots\subseteq X_n = X,
\end{equation*}
where each $X_j/X_{j-1}$ is simple, induces a $U$-module filtration
\begin{equation*}
0 = i'(X_0)\subseteq i'(X_1)\subseteq i'(X_2)\subseteq\cdots\subseteq i'(X_n)=\langle b'\rangle_U'
\end{equation*}
such that $i'(X_j)/i'(X_{j-1})$ is a homomorphic image of $X_j/X_{j-1}$. By \cite[Theorem 4.7.4]{FHL}, $X_j/X_{j-1}\cong M_j\otimes W_j$ for some simple $U$-module $M_j$ and simple $V$-module $W_j$; thus because $i'(X_j)/i'(X_{j-1})$ is grading restricted, it is isomorphic to a finite direct sum of copies of $M_j$. It follows that $\langle b'\rangle_U'$ is a finite-length $U$-module and thus is finitely generated.

For each of finitely many generators $\lbrace b_i\rbrace$ of $\langle b'\rangle_U'$, choose $c_i\in X$ such that $i'(c_i)=b_i$. Then $\langle b'\rangle_U'$ is a quotient of $\bigoplus_i \langle c_i\rangle_U$, and each $\langle c_i\rangle_U\in\mathrm{Ob}(\cU)$ because $X\in\mathrm{Ob}(\cD(\cU,\cV))$. Thus $\langle b'\rangle_U'\in\mathrm{Ob}(\cU)$, and then so is $\langle b'\rangle_U''\cong\langle b'\rangle_U$ since $\cU$ is closed under contragredients. Similarly, $\langle b'\rangle_V\in\mathrm{Ob}(\cV)$, completing the proof that $X'\in\mathrm{Ob}(\cD(\cU,\cV))$.
\end{proof}

For the next key result, $\cU$ and $\cV$ will be the categories of $C_1$-cofinite $U$- and $V$-modules:
\begin{thm}\label{thm:D(U,V)_for_C1_cofinite}
Assume that $U$ and $V$ are $\NN$-graded by conformal weights. If the categories $\cC_U^1$ and $\cC_V^1$ of $C_1$-cofinite grading-restricted generalized $U$- and $V$-modules are closed under submodules, and if all objects in $\cC_U^1$ and $\cC_V^1$ have finite length, then $\cD(\cC_U^1,\cC_V^1)=\cC_{U\otimes V}^1$.
\end{thm}
\begin{proof}
To show that $\cD(\cC_U^1,\cC_V^1)\subseteq\cC_{U\otimes V}^1$, note that $\cD(\cC_U^1,\cC_V^1)$ is closed under subquotients. So the composition factors of any object of $\cD(\cC_U^1,\cC_V^1)$ are objects of $\cD(\cC_U^1,\cC_V^1)$ and thus by Proposition \ref{prop:D(U,V)_simple} have the form $M\otimes W$ where $M\in\mathrm{Ob}(\cC_U^1)$ and $W\in\mathrm{Ob}(\cC_V^1)$. The tensor product of $C_1$-cofinite $U$- and $V$-modules is a $C_1$-cofinite $U\otimes V$-module, so the composition factors of any object of $\cD(\cU,\cV)$ are objects of $\cC_{U\otimes V}^1$. Now if $X_1$ is a $U\otimes V$-module with a $C_1$-cofinite submodule $X_2$ such that $X_1/X_2$ is $C_1$-cofinite, then $X_1$ is also $C_1$-cofinite (see \cite[Lemma 2.11]{Hu-C2}). Thus any object of $\cD(\cU,\cV)$ is $C_1$-cofinite.

Conversely, suppose $X$ is any $C_1$-cofinite $U\otimes V$-module, so that $X= T+C_1(X)$ for some finite-dimensional subspace $T$. We may assume that $T$ is doubly graded by generalized eigenvalues for $L_U(0)$ and $L_V(0)$. 
We claim that $X$ is spanned by the set of all vectors
\begin{equation}\label{eqn:tensor_C1_span}
u^{(1)}_{-1}u^{(2)}_{-1}\cdots u^{(k)}_{-1} v^{(1)}_{-1}v^{(2)}_{-1}\cdots v^{(l)}_{-1} t
\end{equation}
such that $t\in T$, $u^{(i)}\in U$ is homogeneous with $\mathrm{wt}\,u^{(i)}>0$ for $i=1,\ldots k$, and $v^{(j)}\in V$ is homogeneous with $\mathrm{wt}\,v^{(j)}>0$ for $j=1,\ldots, l$. To prove the claim, we use an $\NN$-grading $X=\bigoplus_{n=0}^\infty X(n)$ such that for $m\in\ZZ$ and homogeneous $u\in U$, $v \in V$,
\begin{equation*}
(u\otimes v)_m X(n)\subseteq X(n+\mathrm{wt}\,u+\mathrm{wt}\,v-m-1).
\end{equation*} 
Specifically, since $X$ is grading restricted, we may take $X(n)=\bigoplus_{\mu\in\CC/\ZZ} X_{[h_\mu+n]}$ where for any coset $\mu\in\CC/\ZZ$ such that $X$ has conformal weights in $\mu$, $h_\mu$ is the conformal weight of $X$ in $\mu$ that has minimal real part (and $h_\mu$ is arbitrary if no conformal weights of $X$ are contained in $\mu$). We will prove by induction on $n$ that $X(n)$ is spanned by the vectors \eqref{eqn:tensor_C1_span}.

Indeed, we may write any $b\in X(n)$ as $b=t+\sum_i (u^{(i)}\otimes v^{(i)})_{-1} b^{(i)}$ for $t\in T$ and homogeneous $u^{(i)}\in U$, $v^{(i)}\in V$, and $b^{(i)}\in X$ such that $\mathrm{wt}\,u^{(i)}+\mathrm{wt}\,v^{(i)}>0$. If $b\in X(0)$, then each $b^{(i)}=0$ since $(u^{(i)}\otimes v^{(i)})_{-1}$ is an operator of degree $\mathrm{wt}\,u^{(i)}+\mathrm{wt}\,v^{(i)}>0$, so $b\in T$ in this case. For general $n$, we have
\begin{align*}
b & = t+\sum_i (u^{(i)}\otimes v^{(i)})_{-1} b^{(i)} = t+\sum_i\sum_{j\in\ZZ} u^{(i)}_{-j-1} v^{(i)}_{j-1} b^{(i)}\nonumber\\
& = t+\sum_i u^{(i)}_{-1} v^{(i)}_{-1}b^{(i)} +\sum_i\sum_{j>0} \left(u^{(i)}_{-j-1} v^{(i)}_{j-1}+v^{(i)}_{-j-1} u^{(i)}_{j-1}\right)b^{(i)}\nonumber\\
& = t+\sum_i u^{(i)}_{-1} v^{(i)}_{-1}b^{(i)} +\sum_i\sum_{j>0} \frac{1}{j!}\left((L_U(-1)^j u^{(i)})_{-1} v^{(i)}_{j-1}+(L_V(-1)^j v^{(i)})_{-1} u^{(i)}_{j-1}\right)b^{(i)},
\end{align*}
using the $L(-1)$-derivative property for $U$- and $V$-modules in the last step. Since $U$ and $V$ are $\NN$-graded and $\mathrm{wt}\,u^{(i)}+\mathrm{wt}\,v^{(i)}>0$ for each $i$, one of $u^{(i)}$ and $v^{(i)}$ has strictly positive conformal weight and the other has non-negative conformal weight. Thus for $j>0$, $L_U(-1)^j u^{(i)}$ and $L_V(-1)^j v^{(i)}$ have strictly positive weight, while $v^{(i)}_{j-1} b^{(i)}$ and $u^{(i)}_{j-1} b^{(i)}$ have degree less than $n$. Similarly, if $\mathrm{wt}\,u^{(i)}>0$, then $\deg v^{(i)}_{-1} b^{(i)} <n$, while if $\mathrm{wt}\,v^{(i)}>0$, then $\deg u^{(i)}_{-1} b^{(i)}<n$. Thus by induction on $n$, any $b\in X(n)$ is in the span of the vectors \eqref{eqn:tensor_C1_span}.

Now that we know $X$ is spanned by the vectors \eqref{eqn:tensor_C1_span}, consider the subspace
\begin{equation*}
W=\mathrm{span}\lbrace v^{(1)}_{-1} v^{(2)}_{-1}\cdots v^{(l)}_{-1} t\mid t\in T,\,v^{(j)}\in V,\,\mathrm{wt}\,v^{(j)}>0,\,j=1,\ldots, l\rbrace.
\end{equation*}
This space is $L_V(0)$-stable and graded by generalized $L_V(0)$-eigenvalues because $T$ is. For any $h\in\CC$, the generalized $L_V(0)$-eigenspace $W_{[V,h]}$ of $W$ with generalized eigenvalue $h$ is finite dimensional because $T$ and each conformal weight space of $V$ are finite dimensional, and because $v_{-1}$ strictly raises conformal weights when $\mathrm{wt}\,v>0$. Now because vertex operator modes $u_n$ for $u\in U$, $n\in\ZZ$ preserve generalized $L_V(0)$-eigenvalues, the spanning set \eqref{eqn:tensor_C1_span} implies that
\begin{equation*}
X_{[V,h]} = \mathrm{span}\lbrace u^{(1)}_{-1}u^{(2)}_{-1}\cdots u^{(k)}_{-1} w\mid w\in W_{[V,h]},\,u^{(i)}\in U,\,\mathrm{wt}\,u^{(i)}>0,\,i=1,\ldots, k\rbrace
\end{equation*}
is the generalized $L_V(0)$-eigenspace of $X$ with generalized eigenvalue $h$. Since $W_{[V,h]}$ is finite dimensional, each $X_{[V,h]}$ is thus a $C_1$-cofinite $U$-submodule of $X$. Then because $\cC_U^1$ is closed under submodules by assumption, any vector in $X$ generates a $C_1$-cofinite $U$-submodule. Similarly, any vector in $X$ generates a $C_1$-cofinite $V$-submodule.

It remains to show that $X$ has finite length. Since $X$ is $C_1$-cofinite, it has a finite generating set $\lbrace b_i\rbrace$ (which could be a spanning set for the subspace $T$). Each $b_i$ generates a $C_1$-cofinite $U$-submodule $\langle b_i\rangle_U$ and a $C_1$-cofinite $V$-submodule $\langle b_i\rangle_V$, and the same proof as for Proposition \ref{prop:D(U,V)_surjections} shows that $X$ is a homomorphic image of $\bigoplus_i \langle b_i\rangle_U\otimes\langle b_i\rangle_V$. Since each $\langle b_i\rangle_U$ and $\langle b_i\rangle_V$ have finite length by assumption, so does $X$. This completes the proof that any $C_1$-cofinite $U\otimes V$-module is an object of $\cD(\cC_U^1,\cC_V^1)$.
\end{proof}

\begin{cor}\label{cor:C1_UxV}
If $U$ and $V$ are $\NN$-graded and $\cC_U^1$ and $\cC_V^1$ are closed under contragredients, then $\cC_{U\otimes V}^1$ is closed under submodules and contragredients, is a locally finite abelian category, and admits the braided tensor category structure of \cite{HLZ1}-\cite{HLZ8}.
\end{cor}
\begin{proof}
By Theorems \ref{thm:C1_cofin_prop} and \ref{thm:D(U,V)_for_C1_cofinite}, $\cC^1_{U\otimes V}=\cD(\cC_U^1,\cC_V^1)$, and then $\cC_{U\otimes V}^1$ is closed under submodules (by the definition of $\cD(\cC_U^1,\cC_V^1)$) and contragredients (by Proposition \ref{prop:D(U,V)_contra}). Then $\cC_{U\otimes V}^1$ is a locally finite braided tensor category by Theorem \ref{thm:C1_cofin_prop}.
\end{proof}

\subsection{Projective \texorpdfstring{$U\otimes V$}{U x V}-modules}

In this subsection, $\cU$ and $\cV$ will be locally finite abelian categories of grading-restricted generalized $U$- and $V$-modules, respectively, which are closed under subquotients, so that the category $\cD(\cU,\cV)$ of $U\otimes V$-modules from Definition \ref{def:D(U,V)} is also a locally finite abelian category. We will show that if simple $U$- and $V$-modules $M$ and $W$ have projective covers in $\cU$ and $\cV$, respectively, then so does $M\otimes W$ in $\cD(\cU,\cV)$. Thus if $\cU$ and $\cV$ have enough projectives, then so does $\cD(\cU,\cV)$. Recall from Proposition \ref{prop:tens_prod_in_D(U,V)} that if $M\in\mathrm{Ob}(\cU)$ and $W\in\mathrm{Ob}(\cV)$, then $M\otimes W\in\mathrm{Ob}(\cD(\cU,\cV))$.

Suppose a simple $U$-module $M\in\mathrm{Ob}(\cU)$ has a projective cover $p_M: P_M\twoheadrightarrow M$ in $\cU$. It is not difficult to show that $P_M$ is singly-generated by any $\overline{m}\in P_M\setminus\ker p_M$ (see for example \cite[Lemma 5.13]{McR-rat}). That is, $\ker p_M$ is the unique maximal proper submodule of $P_M$. This means that if $\tilM$ is any simple $U$-module, then any homomorphism $P_M\rightarrow\tilM$ is either $0$ or has the same kernel as $p_M$. Thus
\begin{equation}\label{eqn:proj_cover_homs}
 \hom_U(P_M,\tilM) =\left\lbrace\begin{array}{ccc}
                                 \CC \cdot p_M & \text{if} & \tilM=M\\
                                 0 & \text{if} & \tilM\ncong M
                                \end{array}
\right. .
\end{equation}
The same considerations hold for projective covers of irreducible $V$-modules in $\cV$.

\begin{lemma}\label{lem:PM_PW_max_prop}
 Suppose $p_M: P_M\twoheadrightarrow M$ and $p_W: P_W\twoheadrightarrow W$ are projective covers of simple modules in $\cU$ and $\cV$, respectively. Then $P_M\otimes P_W$ has a unique maximal proper submodule, given by the kernel of $p_M\otimes p_W: P_M\otimes P_W\twoheadrightarrow M\otimes W$.
\end{lemma}
\begin{proof}
 Since $P_M\otimes P_W$ is a finite-length $U\otimes V$-module, any maximal proper submodule is the kernel of a surjection $p: P_M\otimes P_W\rightarrow\tilM\otimes\tilW$, where $\tilM$ and $\tilW$ are irreducible $U$- and $V$-modules, respectively. We need to show that $\ker p =\ker(p_M\otimes p_W)$.
 
 Let $\overline{m}\in P_M\setminus\ker p_M$ and $\overline{w}\in P_W\setminus\ker p_W$ be generating vectors for $P_M$ and $P_W$. Then
 \begin{equation*}
  p_{\overline{w}}: P_M\xrightarrow{\cong} P_M\otimes\overline{w}\hookrightarrow P_M\otimes P_W\xrightarrow{p} \tilM\otimes\tilW
 \end{equation*}
is a homomorphism of weak $U$-modules whose image is a direct sum of finitely many copies of $\tilM$. Since $p$ is non-zero and $P_M\otimes\overline{w}$ generates $P_M\otimes P_W$, we have $p_{\overline{w}}\neq 0$, so that $\mathrm{Im}\,p_{\overline{w}}\cong M$ by \eqref{eqn:proj_cover_homs}. So we may assume $\tilM=M$. It is then easy to show from the complete reducibility of $M\otimes\tilW$ as a weak $U$-module and the Jacobson Density Theorem that every $U$-module homomorphism $M\rightarrow M\otimes\tilW$ has the form $m\mapsto m\otimes\widetilde{w}$ for some $\widetilde{w}\in\tilW$. Consequently, we have a factorization
\begin{equation*}
 p_{\overline{w}}: P_M \xrightarrow{c\cdot p_M} M \xrightarrow{m\mapsto m\otimes\widetilde{w}} M\otimes\tilW
\end{equation*}
for some $c\in\CC$ and $\tilw\in\tilW$, so that $p(\overline{m}\otimes\overline{w})=p_M(\overline{m})\otimes(c\cdot\tilw)$.

By the same argument, we may assume $\tilW=W$, and then $p(\overline{m}\otimes\overline{w})=\tilm\otimes p_W(\overline{w})$ for some $\tilm\in M$. Expanding $\tilm$ in a basis for $M$ that includes $p_M(\overline{m})$ and using
\begin{equation*}
 p_M(\overline{m})\otimes(c\cdot\tilw)=\tilm\otimes p_W(\overline{w}),
\end{equation*}
we see that $\tilm$ is a multiple of $p_M(\overline{m})$, so that $p(\overline{m}\otimes\overline{w})\in\CC\cdot(p_M(\overline{m})\otimes p_W(\overline{w}))$. Thus $p\in\CC^\times\cdot(p_M\otimes p_M)$ since $\overline{m}\otimes\overline{w}$ generates $P_M\otimes P_W$, and then $\ker p=\ker(p_M\otimes p_W)$ as required.
\end{proof}

The preceding lemma shows that the submodule generated by any vector in $(P_M\otimes P_W)\setminus\ker(p_M\otimes p_W)$ cannot be proper, since it is not contained in the unique maximal proper submodule $\ker(p_M\otimes p_W)$. Thus we have:
\begin{cor}\label{cor:PM_PW_sing_gen}
 In the setting of Lemma \ref{lem:PM_PW_max_prop}, $P_M\otimes P_W$ is singly-generated by any vector in $(P_M\otimes P_W)\setminus\ker(p_M\otimes p_W)$.
\end{cor}

Using this corollary, we now prove:
\begin{thm}\label{thm:proj_covers_in_C}
Let $\cU$ and $\cV$ be locally finite abelian categories of grading-restricted generalized $U$- and $V$-modules which are closed under subquotients, and suppose $p_M: P_M\twoheadrightarrow M$ and $p_W: P_W\twoheadrightarrow W$ are projective covers of simple objects in $\cU$ and $\cV$, respectively. Then $p_M\otimes p_W: P_M\otimes P_W\twoheadrightarrow M\otimes W$ is a projective cover of $M\otimes W$ in $\cD(\cU,\cV)$. In particular, if $\cU$ and $\cV$ have enough projectives, then so does $\cD(\cU,\cV)$, and every projective object in $\cD(\cU,\cV)$ is a finite direct sum of tensor products of projective objects in $\cU$ and $\cV$.
\end{thm}
\begin{proof}
The second conclusion follows from the first because $\cD(\cU,\cV)$ is a locally finite abelian category.
Now consider a diagram
\begin{equation*}
 \xymatrix{
 & P_M\otimes P_W \ar[d]^q\\
 X \ar[r]^p & Y\\
 }
\end{equation*}
in $\cD(\cU,\cV)$ with $p$ surjective. We fix generating vectors $\overline{m}\in P_M\setminus\ker p_M$ and $\overline{w}\in P_W\setminus\ker p_W$. Then because $p$ is surjective, there exists $b\in X$ such that
\begin{equation*}
 p(b)=q(\overline{m}\otimes\overline{w}).
\end{equation*}
Let $\langle b\rangle_U\subseteq X$ and $\langle q(\overline{m}\otimes\overline{w})\rangle_U\subseteq Y$ be the $U$-submodules generated by $b$ and $q(\overline{m}\otimes\overline{w})$, respectively; they are objects of $\cU$ by the definition of $\cD(\cU,\cV)$. Thus because $P_M$ is projective in $\cU$, there is a $U$-module homomorphism $f: P_M\rightarrow\langle b\rangle_U$ such that the diagram
\begin{equation*}
\xymatrixcolsep{3pc}
 \xymatrix{
 & P_M \ar[ld]_f \ar[d]^{m\mapsto q(m\otimes\overline{w})}\\
 \langle b\rangle_U \ar[r]_(.4){p\vert_{\langle b\rangle_U}} & \langle q(\overline{m}\otimes\overline{w})\rangle_U\\
 }
\end{equation*}
commutes. Next, let $\langle f(\overline{m})\rangle_V\subseteq X$ and $\langle q(\overline{m}\otimes\overline{w})\rangle_V\subseteq Y$ be the $V$-submodules generated by the indicated vectors; they are objects of $\cV$ since $X$ and $Y$ are objects of $\cD(\cU,\cV)$. Thus there is a $V$-module homomorphism $g: P_W\rightarrow\langle f(\overline{m})\rangle_V$ such that the diagram
\begin{equation*}
\xymatrixcolsep{3pc}
 \xymatrix{
 & P_W \ar[ld]_g \ar[d]^{w\mapsto q(\overline{m}\otimes w)}\\
 \langle f(\overline{m})\rangle_V \ar[r]_(.45){p\vert_{\langle f(\overline{m})\rangle_V}} & \langle q(\overline{m}\otimes\overline{w})\rangle_V\\
 }
\end{equation*}
commutes.

We can now construct a map $F: P_M\otimes P_W\rightarrow X$ as follows. Since $\overline{m}$ generates $P_M$,
\begin{equation*}
 P_M=\mathrm{span}\lbrace u_n\overline{m}\,\vert\,u\in U, n\in\ZZ\rbrace
\end{equation*}
by \cite[Proposition 4.5.6]{LL}. Thus for $m=\sum_i u^{(i)}_{n_i}\overline{m}\in P_M$ and $w\in P_W$, we attempt to set
\begin{equation*}
 F(m\otimes w) = \sum_i u^{(i)}_{n_i}g(w).
\end{equation*}
As in the proof of Proposition \ref{prop:D(U,V)_surjections}, we need to show that  if $\sum_i u^{(i)}_{n_i}\overline{m}\in P_M$ vanishes, then so does $\sum_i u^{(i)}_{n_i}g(w)$ for any $w\in P_W$. In fact, since $g(w)\in\langle f(\overline{m})\rangle_V$ for any such $w$, we can write $g(w)=\sum_j v^{(j)}_{k_j}f(\overline{m})$ for certain $v^{(j)}\in V$ and $k_j\in\ZZ$. Then
\begin{equation}\label{eqn:q_well_def}
 \sum_i u^{(i)}_{n_i} g(w) =\sum_{i,j} u^{(i)}_{n_i} v^{(j)}_{k_j}f(\overline{m}) =\sum_j v^{(j)}_{k_j}\sum_i f(u^{(i)}_{n_i}\overline{m})= 0
\end{equation}
since the vertex operators $Y_{X}(u^{(i)},x)$ and $Y_{X}(v^{(j)},x)$ commute, and since $f$ is a $U$-module homomorphism. This shows that $F$ is well defined.

We now show that $F$ is a $U\otimes V$-module homomorphism. First, for $v\in V$, $n\in\ZZ$, $m=\sum_i u^{(i)}_{n_i}\overline{m}\in P_M$, and $w\in P_W$, we have
\begin{equation*}
 v_n F(m\otimes w)=\sum_i v_n u^{(i)}_{n_i}g(w)=\sum_i u^{(i)}_{n_i}g(v_n w)=F(m\otimes v_n w)
\end{equation*}
since $g$ is a $V$-module homomorphism. This shows that $F$ is a $V$-module homomorphism. To show that $F$ is also a $U$-module homomorphism, fix any $w\in P_W$ and write $g(w)=\sum_j v^{(j)}_{k_j} f(\overline{m})$. Then the calculation in \eqref{eqn:q_well_def} shows that for any $m\in P_M$, 
\begin{equation*}
 F(m\otimes w)=\sum_j v^{(j)}_{k_j}f(m).
\end{equation*}
Consequently, for any $u\in U$, $n\in\ZZ$, $m\in P_M$, and for our arbitrary fixed $w\in P_W$, 
\begin{equation*}
 u_n F(m\otimes w)=\sum_j u_n v^{(j)}_{k_j}f(m) =\sum_j v^{(j)}_{k_j}f(u_n m)=F(u_n m\otimes w),
\end{equation*}
showing that $F$ is also a $U$-module homomorphism.

We now compute $p\circ F$: for any $m=\sum_{i} u^{(i)}_{n_i}\overline{m}\in P_M$ and $w\in P_W$, the definitions imply
\begin{equation*}
 (p\circ F)(m\otimes w) =\sum_i p(u^{(i)}_{n_i}g(w)) = \sum_i u^{(i)}_{n_i}(p\circ g)(w) =\sum_i u^{(i)}_{n_i}q(\overline{m}\otimes w)=q(m\otimes w).
\end{equation*}
This completes the proof that $P_M\otimes P_W$ is projective in $\cD(\cU,\cV)$. To show that it is a projective cover of $M\otimes W$, note that $p_M\otimes p_W: P_M\otimes P_W\rightarrow M\otimes W$ is surjective; then by Proposition \ref{prop:proj_cover_indecomposable}, it is enough to show that $P_M\otimes P_W$ is indecomposable. If indeed
\begin{equation*}
 P_M\otimes P_W= X_1\oplus X_2
\end{equation*}
for non-zero summands $X_1$ and $X_2$, then there would be surjections
\begin{equation*}
 p_1: X_1\rightarrow M_1\otimes W_1,\qquad p_2: X_2\rightarrow M_2\otimes W_2
\end{equation*}
for simple $U$-modules $M_1$, $M_2$ and simple $V$-modules $W_1$, $W_2$. But this is impossible because then $\ker p_1\oplus X_2$ and $X_1\oplus\ker p_2$ would be two distinct maximal proper submodules of $P_M\otimes P_W$,  contradicting Lemma \ref{lem:PM_PW_max_prop}.
\end{proof}

\begin{remark}\label{rem:C2_D(U,V)}
 If $U$ and $V$ are $\NN$-graded and $C_2$-cofinite and we take $\cU=\rep(U)$ and $\cV=\rep(V)$ to be the full categories of grading-restricted generalized $U$- and $V$-modules, then \cite[Theorem 3.24]{Hu-C2} implies that $\cU$ and $\cV$ have enough projectives. In this case, recalling Remark \ref{rem:RepV=C1_for_C2} and Theorem \ref{thm:D(U,V)_for_C1_cofinite}, $\cD(\cU,\cV)=\rep(U\otimes V)$, and thus Theorem \ref{thm:proj_covers_in_C} shows that projective covers of simple $U\otimes V$-modules are tensor products of projective covers of simple $U$- and $V$-modules.
\end{remark}

\subsection{\texorpdfstring{$\cD(\cU,\cV)$}{D(U,V)} as a Deligne tensor product}

Let $\cU$ and $\cV$ be locally finite abelian categories of $U$- and $V$-modules which are closed under subquotients. Then the Deligne tensor product $\cU\otimes\cV$ of $\cU$ and $\cV$ exists; in this subsection, we will show that $\cU\otimes\cV$ is equivalent to the category $\cD(\cU,\cV)$ of Definition \ref{def:D(U,V)}, equipped with the obvious $\CC$-bilinear functor
\begin{equation*}
 \otimes: \cU\times\cV\rightarrow\cD(\cU,\cV)
\end{equation*}
(recall Proposition \ref{prop:tens_prod_in_D(U,V)}) which is exact in both variables.

Recall from Section \ref{subsec:prelim_loc_fin} that the abelian full subcategory $\langle M\rangle$ generated by any $M\in\mathrm{Ob}(\cU)$ is a finite abelian subcategory of $\cU$, and similarly for $\cV$. We will use these subcategories of $\cU$ and $\cV$ to realize $\cD(\cU,\cV)$ as the union of finite abelian full subcategories; since we will need to consider right exact sequences 
$X_1\rightarrow X_2\rightarrow X_3\rightarrow 0$ in $\cD(\cU,\cV)$ to show that $\cD(\cU,\cV)$ is a Deligne tensor product, we will associate such a subcategory to any unordered triple $X_1, X_2, X_3\in\mathrm{Ob}(\cD(\cU,\cV))$. Thus for $i=1,2,3$, we use Proposition \ref{prop:D(U,V)_surjections} to fix a surjection $\bigoplus_{n_i} M_{n_i}\otimes W_{n_i} \rightarrow X_i$ where $M_{n_i}\in\mathrm{Ob}(\cU)$, $W_{n_i}\in\mathrm{Ob}(\cV)$. We then set
\begin{equation}\label{eqn:DX1X2X3_def}
\cD_{X_1,X_2,X_3} = \cD(\langle\oplus_{i=1,2,3}\oplus_{n_i} M_{n_i}\rangle,\langle\oplus_{i=1,2,3} \oplus_{n_i} W_{n_i}\rangle ).
\end{equation}
This is an abelian full subcategory of $\cD(\cU,\cV)$, and it contains $X_1$, $X_2$, and $X_3$ by Proposition \ref{prop:tens_prod_in_D(U,V)}. Also, because $\langle\oplus_{i=1,2,3}\oplus_{n_i} M_{n_i}\rangle$ and $\langle\oplus_{i=1,2,3}\oplus_{n_i} W_{n_i}\rangle$ are finite abelian categories, Proposition \ref{prop:D(U,V)_simple} implies that $\cD_{X_1,X_2,X_3}$ has finitely many simple objects up to isomorphism, and Theorem \ref{thm:proj_covers_in_C} implies that any simple module in $\cD_{X_1,X_2,X_3}$ has a projective cover. Thus $\cD_{X_1,X_2,X_3}$ is a finite abelian full subcategory of $\cD(\cU,\cV)$ that contains $X_1$, $X_2$, and $X_3$. For brevity, we use the notation $\cD_{X_1,X_2} = \cD_{X_1,X_1,X_2}$ (which is equal to $\cD_{X_1,X_2,X_2}$) for $X_1,X_2\in\mathrm{Ob}(\cD(\cU,\cV))$, and $\cD_X=\cD_{X,X,X}$ for $X\in\mathrm{Ob}(\cD(\cU,\cV))$. It is easy to see that
\begin{equation*}
\cD_{X_i}\subseteq\cD_{X_i,X_j}\subseteq\cD_{X_1,X_2,X_3}
\end{equation*}
for $i,j=1,2,3$.

For $\cD(\cU,\cV)$ to be a Deligne tensor product of $\cU$ and $\cV$, we need to show that bilinear functors $\cU\times\cV\rightarrow\cC$ which are right exact in both variables uniquely induce suitable right exact functors $\cD(\cU,\cV)\rightarrow\cC$. The idea is to first use properties of projective objects to show that each $\cD_X\subseteq\cD(\cU,\cV)$ is a Deligne tensor product of the subcategories $\langle\oplus_n M_n\rangle$ and $\langle\oplus_n W_n\rangle$, and thus a functor $\cU\times\cV\rightarrow\cC$ which is right exact in both variables induces a right exact functor $\cD_X\rightarrow\cC$. We then want to ``glue'' the functors defined on these subcategories into a functor $\cD(\cU,\cV)\rightarrow\cC$; we will use the subcategories $\cD_{X_i,X_j}$ and $\cD_{X_1,X_2,X_3}$ to show that these functors can be glued together coherently. Most of the proof has nothing particularly to do with vertex operator algebras, so we put most of the work into the proof of the following general theorem on abelian categories; probably the result is known or straightforward to experts, but for completeness we include a proof in Appendix \ref{app:extend_to_right_exact}:
\begin{thm}\label{thm:gen_extend_to_right_exact}
Let $\cD$ be a $\CC$-linear abelian category, let $\til{\cD}\subseteq\cD$ be an additive full subcategory, and assume that for any unordered triple $X_1,X_2,X_3\in\mathrm{Ob}(\cD)$, there is an abelian full subcategory $\cD_{X_1,X_2,X_3}\subseteq\cD$  that contains $X_1$, $X_2$, and $X_3$ and has enough projectives, such that every projective object in $\cD_{X_1,X_2,X_3}$ is in $\til{\cD}$, and such that for all $i,j=1,2,3$,
\begin{equation*}
\cD_{X_i,X_i,X_i}\subseteq\cD_{X_i,X_i,X_j}=\cD_{X_i,X_j,X_j}\subseteq\cD_{X_1,X_2,X_3}.
\end{equation*}
Let $\cG:\til{\cD}\rightarrow\cC$ be a $\CC$-linear functor, where $\cC$ is an abelian category, such that for any $X_1,X_2,X_3\in\mathrm{Ob}(\cD)$ and $X\in\mathrm{Ob}(\til{\cD}\cap\cD_{X_1,X_2,X_3})$, there is a right exact sequence
\begin{equation*}
Q_{X}\xrightarrow{q_{X}} P_{X}\xrightarrow{p_{X}} X\longrightarrow 0,
\end{equation*}
where $P_{X}$ and $Q_{X}$ are projective in $\cD_{X_1,X_2,X_3}$, such that $(\cG(X),\cG(p_{X}))$ is a cokernel of $\cG(q_{X})$ in $\cC$. Then there is a unique (up to natural isomorphism) right exact functor $\cF:\cD\rightarrow\cC$ such that $\cF\vert_{\til{\cD}}\cong\cG$.
\end{thm}

\begin{remark}\label{rem:equal_vs_nat_iso}
In Theorem \ref{thm:gen_extend_to_right_exact}, we may replace the conclusion $\cF\vert_{\til{\cD}}\cong\cG$ with the conclusion that $\cF\vert_{\til{\cD}}$ is precisely equal to $\cG$. Indeed, given $\cF:\cD\rightarrow\cC$ and a natural isomorphism $\eta: \cF\vert_{\til{\cD}}\rightarrow\cG$, we can define a new functor $\til{\cF}: \cD\rightarrow\cC$ on objects by
\begin{equation*}
\til{\cF}(X) =\left\lbrace\begin{array}{ccc}
\cG(X) & \text{if} & X\in\mathrm{Ob}\,\til{\cD}\\
\cF(X) & \text{if} & X\notin\mathrm{Ob}\,\til{\cD}\\
\end{array}\right. ,
\end{equation*}
and on morphisms by
\begin{equation*}
\til{\cF}(f: X_1\rightarrow X_2) =\left\lbrace\begin{array}{lll}
\eta_{X_2}\circ\cF(f)\circ\eta_{X_1}^{-1} & \text{if} & X_1,X_2\in\mathrm{Ob}\,\til{\cD}\\
\cF(f)\circ\eta_{X_1}^{-1} & \text{if} & X_1\in\mathrm{Ob}\,\til{\cD},\, X_2\notin\mathrm{Ob}\,\til{\cD}\\
\eta_{X_2}\circ\cF(f) & \text{if} & X_1\notin\mathrm{Ob}\,\til{\cD},\, X_2\in\mathrm{Ob}\,\til{\cD}\\
\cF(f) & \text{if} & X_1,X_2\notin\mathrm{Ob}\,\til{\cD}
\end{array}\right. .
\end{equation*}
Then $\til{\cF}\vert_{\til{\cD}}=\cG$, and there is a natural isomorphism $\til{\eta}:\cF\rightarrow\til{\cF}$ given by
\begin{equation*}
\til{\eta}_{X} =\left\lbrace\begin{array}{lll}
\eta_X & \text{if} & X\in\mathrm{Ob}\,\til{\cD}\\
\Id_{\cF(X)} & \text{if} & X\notin\mathrm{Ob}\,\til{\cD}\\
\end{array}\right. ,
\end{equation*}
so $\til{\cF}$ is right exact if $\cF$ is.
\end{remark}

We will now use Theorem \ref{thm:gen_extend_to_right_exact} to prove the main theorem of this subsection:
\begin{thm}\label{thm:D(U,V)_is_Del_prod}
If $\cU$ and $\cV$ are locally finite abelian categories of grading-restricted generalized $U$- and $V$-modules, respectively, which are closed under subquotients, then $(\cD(\cU,\cV),\otimes)$ is a Deligne tensor product of $\cU$ and $\cV$.
\end{thm}
\begin{proof}
We need to show that for any abelian category $\cC$ and any $\CC$-bilinear functor $\cB: \cU\times\cV\rightarrow\cC$ which is right exact in both variables, there is a unique (up to natural isomorphism) right exact functor $\cF:\cD(\cU,\cV)\rightarrow\cC$ such that the diagram
\begin{equation*}
 \xymatrix{
 \cU\times\cV \ar[d]_{\otimes} \ar[r]^(.57){\cB} & \cC\\
 \cD(\cU,\cV) \ar[ru]_{\cF} &\\
 }
\end{equation*}
commutes. We will apply Theorem \ref{thm:gen_extend_to_right_exact} in the case that $\cD=\cD(\cU,\cV)$ and $\til{\cD}$ is the additive full subcategory consisting of all $U\otimes V$-modules in $\cD(\cU,\cV)$ which are isomorphic to finite direct sums of modules $M\otimes W$ for $M\in\mathrm{Ob}(\cU)$, $W\in\mathrm{Ob}(\cV)$. By Theorem \ref{thm:proj_covers_in_C}, the finite abelian subcategories $\cD_{X_1,X_2,X_3}$ defined in \eqref{eqn:DX1X2X3_def} satisfy the assumptions in Theorem \ref{thm:gen_extend_to_right_exact}.

 We now construct a functor $\cG:\til{\cD}\rightarrow\cC$ such that $\cG\circ\otimes=\cB$. For any $X\in\mathrm{Ob}(\til{\cD})$, we fix an isomorphism $X\cong\bigoplus_i M_i\otimes W_i$, where each $M_i\in\mathrm{Ob}(\cU)$ and each $W_i\in\mathrm{Ob}(\cV)$, via a system of inclusions and projections
\begin{equation*}
 q^{(i)}_X: M_i\otimes W_i\longrightarrow X,\qquad\pi^{(i)}_X: X\longrightarrow M_i\otimes W_i
\end{equation*}
such that
\begin{equation*}
 \pi^{(i)}_X\circ q^{(j)}_X =\delta_{i,j}\Id_{M_i\otimes W_i},\qquad\sum_i q^{(i)}_X\circ\pi^{(i)}_X=\Id_X.
\end{equation*}
Then on objects of $\til{\cD}$, we define
\begin{equation*}
 \cG(X)=\bigoplus_i \cB(M_i,W_i).
\end{equation*}
For a morphism $f: X\rightarrow\til{X}\cong\bigoplus_j\til{M}_j\otimes\til{W}_j$ in $\til{\cD}$, we have a morphism
\begin{equation*}
 \pi_{\til{X}}^{(j)}\circ f\circ q^{(i)}_X: M_i\otimes W_i\longrightarrow\til{M}_j\otimes\til{W}_j
\end{equation*}
for all $i,j$. We would like to apply $\cB$ to these morphisms and then define
\begin{equation*}
 \cG(f)=\sum_{i,j} q^{(j)}_{\cG(\til{X})}\circ\cB(\pi^{(j)}_{\til{X}}\circ f\circ q_X^{(i)})\circ\pi^{(i)}_{\cG(X)}.
\end{equation*}
However, $\pi_{\til{X}}^{(j)}\circ f\circ q^{(i)}_X$ is a morphism in $\til{\cD}$, not in $\cU\times\cV$, so we need to determine how to interpret $\cB$ as a map on morphisms in $\til{\cD}$.

In fact, because $\cB$ is by assumption bilinear on morphisms, we have, for any $M_1, M_2\in\mathrm{Ob}(\cU)$ and $W_1, W_2\in\mathrm{Ob}(\cV)$, a linear map
\begin{align*}
 \hom_\cU(M_1,M_2)\otimes\hom_\cV(W_1,W_2) &\rightarrow\hom_\cC(\cB(M_1,W_1),\cB(M_2,W_2))\nonumber\\
 f\otimes g & \mapsto \cB(f,g).
\end{align*}
Combining this with Proposition \ref{prop:proj_cover_homs} yields a linear map
\begin{equation*}
\cB: \hom_{\cD(\cU,\cV)}(M_1\otimes W_1,M_2\otimes W_2)\longrightarrow\hom_\cC(\cB(M_1,W_1),\cB(M_2,W_2)).
\end{equation*}
 We use this map $\cB$ to define $\cG: \til{\cD}\rightarrow\cC$ on morphisms, and then it is straightforward to show that $\cG$ is a $\CC$-linear functor. Moreover, if $X=M\otimes W$ is in the image of $\otimes: \cU\times\cV\rightarrow\til{\cD}$, then we can choose $q_X^{(i)}=\pi_X^{(i)}=\Id_X$, so we may assume $\cG\circ\otimes=\cB$.
 
Now fix $X_1,X_2,X_3\in\mathrm{Ob}(\cD)$ and $X\cong\bigoplus_i M_i\otimes W_i\in\mathrm{Ob}(\til{\cD}\cap\cD_{X_1,X_2,X_3})$. Recall from \eqref{eqn:DX1X2X3_def} that $\cD_{X_1,X_2,X_3}=\cD(\til{\cU},\til{\cV})$ for certain finite abelian subcategories $\til{\cU}\subseteq\cU$ and $\til{\cV}\subseteq\cV$. By Proposition \ref{prop:tens_prod_in_D(U,V)}, each $M_i\in\mathrm{Ob}(\til{\cU})$ and each $W_i\in\mathrm{Ob}(\til{\cV})$. Thus for any $i$,
there are right exact sequences
\begin{equation*}
Q_M\xrightarrow{q_M} P_M\xrightarrow{p_M} M_i\longrightarrow 0,\qquad Q_W\xrightarrow{q_W} P_W\xrightarrow{p_W} W_i\longrightarrow 0
\end{equation*}
such that $P_M$, $Q_M$ are projective in $\til{\cU}$ and $P_W$, $Q_W$ are projective in $\til{\cV}$. Then
\begin{equation*}
(Q_M\otimes P_W)\oplus(P_M\otimes Q_W) \xrightarrow{F} P_M\otimes P_W\xrightarrow{p_M\otimes p_W} M_i\otimes W_i\longrightarrow 0
\end{equation*}
is right exact in $\cD_{X_1,X_2,X_3}$, where $F=(q_M\otimes\Id_{P_W})\circ\pi_1+(\Id_{P_M}\otimes q_W)\circ\pi_2$, and the domains of $F$ and $p_M\otimes p_W$ are projective in $\cD_{X_1,X_2,X_3}$ by Theorem \ref{thm:proj_covers_in_C}.

Now because $\cB$ is right exact in both variables, Lemma \ref{lem:cokernel_of_tens_prod} implies that
\begin{equation*}
\cB(Q_M, P_W)\oplus\cB(P_M, Q_W) \xrightarrow{\til{F}} \cB(P_M, P_W)\xrightarrow{\cB(p_M, p_W)} \cB(M_i, W_i)\longrightarrow 0
\end{equation*}
is right exact in $\cC$, where $\til{F}=\cB(q_M,\Id_{P_W})\circ\pi_1+\cB(\Id_{P_M},q_W)\circ\pi_2$. Because $\cG\circ\otimes=\cB$, $\cG(P_M\otimes P_W)=\cB(P_M,P_W)$, $\cG(M_i\otimes W_i)=\cB(M_i,W_i)$, and $\cG(p_M\otimes p_W)=\cB(p_M,p_W)$, Because $\cG$ is also $\CC$-linear and thus preserves direct sums up to natural isomorphism, there is an isomorphism
\begin{equation*}
G: \cG((Q_M\otimes P_W)\oplus(P_M\otimes Q_W))\longrightarrow \cB(Q_M,P_W)\oplus\cB(P_M,Q_W)
\end{equation*}
such that $\cG(F)=\til{F}\circ G$. Thus $(\cG(M_i\otimes W_i), \cG(p_M\otimes p_W))$ is a cokernel of $\cG(F)$ for each $i$. Since $\cG$ is $\CC$-linear and thus preserves direct sums, we can then take the direct sum over $i$ of the above right exact sequences to get a right exact sequence
\begin{equation*}
Q_X \xrightarrow{q_X} P_X \xrightarrow{p_X} X \longrightarrow 0
\end{equation*}
such that $Q_X$ and $P_X$ are projective objects of $\cD_{X_1,X_2,X_3}$ and $(\cG(X),\cG(p_X))$ is a cokernel of $\cG(q_X)$ in $\cC$. It now follows from Theorem \ref{thm:gen_extend_to_right_exact} that there is a unique right exact $\CC$-linear functor $\cF:\cD(\cU,\cV)\rightarrow \cC$ such that $\cF\vert_{\til{\cD}}\cong\cG$.

As in Remark \ref{rem:equal_vs_nat_iso}, we can modify $\cF$ if necessary so that $\cF\vert_{\til{\cD}}$ is precisely equal to $\cG$, so
\begin{equation*}
\cF\circ\otimes=\cF\vert_{\til{\cD}}\circ\otimes = \cG\circ\otimes =\cB.
\end{equation*}
It remains to show that if $\til{\cF}: \cD(\cU,\cV)\rightarrow\cC$ is any right exact functor such that $\til{\cF}\circ\otimes =\cB$, then $\til{\cF}\cong\cF$. By the uniqueness of $\cF$, it is enough to show that $\til{\cF}\circ\otimes=\cB$ implies $\til{\cF}\vert_{\til{\cD}}\cong \cG$. Indeed, since all objects in $\til{\cD}$ are finite direct sums of objects in the image of $\otimes$, and since both $\til{\cF}\vert_{\til{\cD}}$ and $\cG$ are $\CC$-linear and thus preserve finite direct sums up to natural isomorphism, both $\til{\cF}\vert_{\til{\cD}}$ and $\cG$ are uniquely determined by their restrictions to the image of $\otimes$. In particular, it is straightforward to construct a natural isomorphism $\til{\cF}\vert_{\til{\cD}}\rightarrow\cG$. This completes the proof that $(\cD(\cU,\cV),\otimes)$ is a Deligne tensor product of $\cU$ and $\cV$.
\end{proof}

\section{Tensor categories for tensor product vertex operator algebras}

Let $U$ and $V$ be vertex operator algebras. In this section, we construct tensor categories of $U\otimes V$-modules from tensor categories of $U$- and $V$-modules. The goal is to show that if $\cU$ and $\cV$ are locally finite braided tensor categories of $U$- and $V$-modules, then the category $\cD(\cU,\cV)$ of $U\otimes V$-modules admits vertex algebraic braided tensor category structure, as in \cite{HLZ1}-\cite{HLZ8}, which is equivalent to the braided tensor structure on the Deligne tensor product $\cU\otimes\cV$ inherited from the braided tensor categories $\cU$ and $\cV$.

\subsection{Fusion product modules}

In this subsection, we prove that under fairly general conditions, the fusion product of two vector space tensor product $U\otimes V$-modules is a tensor product of two fusion products. We will prove this in the following setting:
\begin{assum}\label{assum:fus_1}
Let $\cU$, $\cV$, and $\cD$ be categories of generalized $U$-modules, $V$-modules, and $U\otimes V$-modules, respectively, which satisfy the following conditions:
\begin{itemize}

\item All objects of $\cD$ decompose into generalized eigenspaces for $L_U(0)$ and $L_V(0)$, that is, objects of $\cD$ are generalized $U$-modules and generalized $V$-modules.

\item If $M\in\mathrm{Ob}(\cU)$ and $W\in\mathrm{Ob}(\cV)$, then $M\otimes W\in\mathrm{Ob}(\cD)$.

\item The conformal weights of any object in $\cU$ or $\cV$ are contained in finitely many cosets of $\CC/\ZZ$.

\item If $\mathcal{Y}$ is a $U$-module intertwining operator of type $\binom{X}{M_1\,M_2}$ where $M_1, M_2\in\mathrm{Ob}(\cU)$ and $X\in\mathrm{Ob}(\mathcal{D})$, then $\im\,\cY\in\mathrm{Ob}(\mathcal{U})$.
Similarly, if $\mathcal{Y}$ is a $V$-module intertwining operator of type $\binom{X}{W_1\,W_2}$ where $W_1, W_2\in\mathrm{Ob}(\cV)$ and $X\in\mathrm{Ob}(\mathcal{D})$, then $\im\,\cY\in\mathrm{Ob}(\mathcal{V})$.

\end{itemize}
\end{assum}

The third assumption in Assumption \ref{assum:fus_1} guarantees that for any $U$-module intertwining operator $\cY_U$ of type $\binom{M_3}{M_1\,M_2}$ and any $V$-module intertwining operator $\cY_V$ of type $\binom{W_3}{W_1\,W_2}$, where $M_i\in\mathrm{Ob}(\cU)$ and $W_i\in\mathrm{Ob}(\cV)$ for $i=1,2,3$, then
\begin{align*}
\cY_U\otimes\cY_V: (M_1\otimes W_1)\otimes(M_2\otimes W_2) & \rightarrow (M_3\otimes W_3)[\log x]\lbrace x\rbrace\nonumber\\
(m_1\otimes w_1)\otimes(m_2\otimes w_2) & \mapsto\cY_U(m_1,x)m_2\otimes\cY_V(w_1,x)w_2
\end{align*}
is a well-defined linear map which satisfies lower truncation, and is moreover a $U\otimes V$-module intertwining operator of type $\binom{M_3\otimes W_3}{M_1\otimes W_1\,M_2\otimes W_2}$; see \cite[Section 2.2]{ADL}. In particular, if $(M_1\tens_U M_2,\cY_{M_1,M_2})$ is a fusion product of $M_1$ and $M_2$ in $\cU$, and if $(W_1\tens_V W_2,\cY_{W_1,W_2})$ is a fusion product of $W_1$ and $W_2$ in $\cV$, then 
\begin{equation*}
\cY_{M_1,M_2}\otimes\cY_{W_1,W_2}: (M_1\otimes W_1)\otimes(M_2\otimes W_2)\longrightarrow(M_1\tens_U M_2)\otimes(W_1\tens_V W_2)[\log x]\lbrace x\rbrace
\end{equation*}
is a $U\otimes V$-module intertwining operator, where all three $U\otimes V$-modules are objects of $\cD$.

We will show that $\cY_{M_1,M_2}\otimes\cY_{W_1,W_2}$ is a fusion product intertwining operator. To do so, we need a lemma which is a version of \cite[Proposition 13.18]{DL} and \cite[Proposition 2.11]{ADL}; it is proved the same way as \cite[Lemma B.1]{McR-cosets}. In the statement of the lemma, $X_{[h]}$ denotes the generalized $L_V(0)$-eigenspace of eigenvalue $h$ inside a weak $U\otimes V$-module $X$ that decomposes into generalized $L_V(0)$-eigenspaces, and $P_h$ denotes the projection onto $X_{[h]}$ with respect to the generalized $L_V(0)$-eigenspace decomposition of $X$. Since the actions of $U$ and $V$ commute on weak $U\otimes V$-modules, each $X_{[h]}$ is a weak $U$-submodule of $X$.
\begin{lemma}\label{lem:V-mod_intw_op}
Let $M_1$ and $M_2$ be weak $U$-modules, $W_1$ and $W_2$ generalized $V$-modules, and $X$ a weak $U\otimes V$-module that decomposes into generalized $L_V(0)$-eigenspaces. If $\cY$ is a $U\otimes V$-module intertwining operator of type $\binom{X}{M_1\otimes W_1\,M_2\otimes W_2}$, then for $w_1\in W_1$, $w_2\in W_2$, $h\in\CC$,
\begin{equation*}
\cY_{w_1,w_2}^{(h)}=x^{-L_V(0)} P_h\cY(\cdot\otimes x^{L_V(0)}w_1,x)(\cdot\otimes x^{L_V(0)}w_2)
\end{equation*}
is a weak $U$-module intertwining operator of type $\binom{X_{[h]}}{M_1\,M_2}$.
\end{lemma}

Using this lemma, we now prove the main result of this subsection, which generalizes \cite[Lemma 2.16]{Lin}, \cite[Proposition 3.3]{CKLR}, and \cite[Theorem 5.2(1)]{CKM2}:
\begin{thm}\label{thm:dist_fus_over_tens}
Under Assumption \ref{assum:fus_1}, let $(M_1\tens_U M_2,\cY_{M_1,M_2})$ be a fusion product of $M_1$ and $M_2$ in $\cU$ and let $(W_1\tens_V W_2,\cY_{W_1,W_2})$ be a fusion product of $W_1$ and $W_2$ in $\cV$. Then
\begin{equation*}
((M_1\tens_U M_2)\otimes(W_1\tens_V W_2),\cY_{M_1,M_2}\otimes\cY_{W_1,W_2})
\end{equation*}
is a fusion product of $M_1\otimes W_1$ and $M_2\otimes W_2$ in $\cD$.
\end{thm}
\begin{proof}
We need to show that for any $U\otimes V$-module intertwining operator $\cY$ of type $\binom{X}{M_1\otimes W_1\,M_2\otimes W_2}$ where $X\in\mathrm{Ob}(\cD)$, there is a unique $U\otimes V$-module map
\begin{equation*}
F: (M_1\tens_U M_2)\otimes(W_1\tens_V W_2)\longrightarrow X
\end{equation*}
such that $F\circ(\cY_{M_1,M_2}\otimes\cY_{W_1,W_2})=\cY$. For any such $\cY$, and for any $w_1\in W_1$, $w_2\in W_2$, $h\in\CC$, consider the intertwining operator $\cY^{(h)}_{w_1,w_2}$ of type $\binom{X_{[h]}}{M_1\,M_2}$ from the lemma. Since $X\in\mathrm{Ob}(\cD)$, Assumption \ref{assum:fus_1} implies that $\im\,\cY^{(h)}_{w_1,w_2}\in\mathrm{Ob}(\cU)$. Thus by the universal property of fusion products in $\cU$, there is a unique $U$-module map
\begin{equation*}
f^{(h)}_{w_1,w_2}: M_1\tens_U M_2\longrightarrow X_{[h]}
\end{equation*}
such that $f^{(h)}_{w_1,w_2}\circ\cY_{M_1,M_2}=\cY^{(h)}_{w_1,w_2}$. The map $(w_1,w_2)\mapsto f^{(h)}_{w_1,w_2}$ is bilinear, so for any $m\in M_1\tens_U M_2$, there is a linear map
\begin{align*}
\cY_m: W_1\otimes W_2 & \rightarrow X[\log x]\lbrace x\rbrace\nonumber\\
 w_1\otimes w_2 & \mapsto \sum_{h\in\CC} x^{L_V(0)} f^{(h)}_{x^{-L_V(0)} w_1,x^{-L_V(0)} w_2}(m),
\end{align*}
where for homogeneous $w_1\in W_1$, $w_2\in W_2$,
\begin{equation*}
f^{(h)}_{x^{-L_V(0)} w_1,x^{-L_V(0)} w_2} =x^{-\mathrm{wt}_V\,w_1-\mathrm{wt}_V\,w_2}\sum_{j,k\in\NN} \frac{(-1)^{j+k}}{j!k!} (\log x)^{j+k} f^{(h)}_{L_V^{nil}(0)^j w_1, L_V^{nil}(0)^k w_2},
\end{equation*}
 and $L_V^{nil}(0)$ denotes the locally nilpotent part of $L_V(0)$ acting on a generalized $V$-module.

We claim that $\cY_m$ is a $V$-module intertwining operator of type $\binom{X}{W_1\,W_2}$ for any $m\in M_1\tens_U M_2$. To prove this, we may take $m$ to be the coefficient of $y^{h}(\log y)^0$ in
$\cY_{M_1,M_2}(m_1,y)m_2$ for some $h\in\CC$ and homogeneous $m_1\in M_1$, $m_2\in M_2$; such coefficients span $M_1\tens_U M_2$ by \cite[Lemma 2.2]{MY} since $\cY_{M_1,M_2}$ is surjective. Thus $\cY_m(w_1,x)w_2$ is a coefficient of
\begin{align*}
\sum_{h\in\CC} x^{L_V(0)} & f^{(h)}_{x^{-L_V(0)} w_1,x^{-L_V(0)} w_2}(\cY_{M_1,M_2}(m_1,y)m_2)\nonumber\\
& =\sum_{h\in\CC} x^{L_V(0)}\cY^{(h)}_{x^{-L_V(0)} w_1, x^{-L_V(0)} w_2}(m_1,y)m_2\nonumber\\
& =\sum_{h\in\CC} \left(\frac{x}{y}\right)^{L_V(0)} P_h\cY\bigg(m_1\otimes\left(\frac{x}{y}\right)^{-L_V(0)}w_1,y\bigg)\bigg(m_2\otimes\left(\frac{x}{y}\right)^{-L_V(0)}w_2\bigg)\nonumber\\
& =\sum_{h\in\CC} x^{L_V(0)} P_h y^{L_U(0)}\cY\left( y^{-L_U(0)} m_1\otimes x^{-L_V(0)} w_1,1\right)(y^{-L_U(0)} m_2\otimes x^{-L_V(0)} w_2)\nonumber\\
& = x^{L_V(0)} y^{L_U(0)}\cY\left( y^{-L_U(0)} m_1\otimes x^{-L_V(0)} w_1,1\right)(y^{-L_U(0)} m_2\otimes x^{-L_V(0)} w_2)\nonumber\\
& =\left(\frac{y}{x}\right)^{L_U(0)} x^{L(0)}\cY\left( y^{-L_U(0)} m_1\otimes x^{-L_V(0)} w_1,1\right)(y^{-L_U(0)} m_2\otimes x^{-L_V(0)} w_2)\nonumber\\
& =\left(\frac{x}{y}\right)^{-L_U(0)} \cY\bigg( \left(\frac{x}{y}\right)^{L_U(0)} m_1\otimes w_1,x\bigg)\bigg(\left(\frac{x}{y}\right)^{L_U(0)} m_2\otimes w_2\bigg),
\end{align*}
using the $L(0)$-conjugation property \cite[Proposition 3.36(b)]{HLZ2} for the intertwining operator $\cY$. We now extract the coefficient of $y^{\eta-\mathrm{wt}\,m_1-\mathrm{wt}\,m_2}(\log y)^0$, that is, we are taking
\begin{equation*}
m=(m_1)^{\cY_{M_1,M_2}}_{\mathrm{wt}_U\,m_1+\mathrm{wt}_U\,m_2-\eta-1;0} m_2
\end{equation*}
for arbitrary $\eta\in\CC$. Since $\mathrm{wt}_U\,m=\eta$ and $f^{(h)}_{w_1,w_2}$ is a $U$-module homomorphism for any $w_1\in W_1$, $w_2\in W_2$, and $h\in\CC$, the definition of $\cY_m$ implies that $\mathrm{Im}\,\cY_m\subseteq X^{[\eta]}$, where $X^{[\eta]}$ is the generalized $L_U(0)$-eigenspace of $X$ with eigenvalue $\eta$. Let $P^\eta: X\rightarrow X^{[\eta]}$ denote the projection with respect to the generalized $L_U(0)$-eigenspace decomposition of $X$. Then it follows that $\cY_m(w_1,x)w_2$ is the constant term in $\log y$ in
\begin{equation*}
e^{L_U^{nil}(0) \log y} x^{-L_U(0)} P^\eta\cY(x^{L_U(0)} e^{-L_U^{nil}(0) \log y} m_1\otimes w_1,x)(x^{L_U(0)} e^{-L_U^{nil}(0) \log y} m_2\otimes w_2).
\end{equation*}
That is,
\begin{equation*}
\cY_m(w_1,x)w_2= x^{-L_U(0)} P^\eta\cY(x^{L_U(0)}  m_1\otimes w_1,x)(x^{L_U(0)} m_2\otimes w_2),
\end{equation*}
which is a $V$-module intertwining operator of type  $\binom{X^{[\eta]}}{W_1\,W_2}$ by Lemma \ref{lem:V-mod_intw_op} with the roles of $U$ and $V$ reversed.

Since each $\cY_m$ is an intertwining operator, its image in $X$ is an object of $\cV$ by assumption. Thus for each $m\in M_1\tens_U M_2$, there is a unique $V$-module homomorphism
\begin{equation*}
f_m: W_1\boxtimes_V W_2\longrightarrow X
\end{equation*}
such that $f_m\circ\cY_{W_1,W_2}=\cY_m$. Since $\cY_m$ is linear in $m$, so is $f_m$, and thus we can now define
\begin{equation*}
F: (M_1\tens_U M_2)\otimes(W_1\boxtimes_V W_2)\longrightarrow X
\end{equation*}
by
\begin{equation*}
F(m\otimes w) = f_m(w)
\end{equation*}
for $m\in M_1\tens_U M_2$ and $w\in W_1\tens_V W_2$. This map is a $V$-module homomorphism by definition, but we need to check that it is also a $U$-module homomorphism. Since $\cY_{W_1,W_2}$ is a surjective intertwining operator of type $\binom{W_1\tens_V W_2}{W_1\,W_2}$, it is enough to show that
\begin{equation*}
u_n F(m\otimes\cY_{W_1,W_2}(w_1,y)w_2)=F(u_n m\otimes\cY_{W_1,W_2}(w_1,y)w_2)
\end{equation*}
for $u\in U$, $n\in\ZZ$, $m\in M_1\tens_U M_2$, $w_1\in W_1$, and $w_2\in W_2$. Indeed,
\begin{align*}
u_n F(m & \otimes\cY_{W_1,W_2}(w_1,y)w_2)  = u_nf_m(\cY_{W_1,W_2}(w_1,y)w_2) =u_n\cY_m(w_1,y)w_2\nonumber\\
& =\sum_{h\in\CC} u_n y^{L_V(0)} f^{(h)}_{y^{-L_V(0)} w_1, y^{-L_V(0)} w_2}(m) =\sum_{h\in\CC} y^{L_V(0)} f^{(h)}_{y^{-L_V(0)} w_1, y^{-L_V(0)} w_2}(u_n m)\nonumber\\
& =\cY_{u_n m}(w_1,y)w_2=f_{u_n m}(\cY_{W_1,W_2}(w_1,y)w_2) = F(u_n m\otimes\cY_{W_1,W_2}(w_1,y)w_2),
\end{align*}
as required.

Finally, we need to show that $F_\cY\circ(\cY_{M_1,M_2}\otimes\cY_{W_1,W_2})=\cY$. For $m_1\in M_1$, $w_1\in W_1$, $m_2\in M_2$, and $w_2\in W_2$, we calculate
\begin{align*}
 F & \left((\cY_{M_1,M_2}\otimes\cY_{W_1,W_2})(m_1 \otimes w_1,x)(m_2\otimes w_2)\right) \nonumber\\ 
 &=F(\cY_{M_1,M_2}(m_1,x)m_2\otimes\cY_{W_1,W_2}(w_1,x)w_2) = f_{\cY_{M_1,M_2}(m_1,x)m_2}(\cY_{W_1,W_2}(w_1,x)w_2)\nonumber\\
& =\sum_{h\in\CC} x^{L_V(0)}f^{(h)}_{x^{-L_V(0)} w_1,x^{-L_V(0)} w_2}(\cY_{M_1,M_2}(m_1,x)m_2)\nonumber\\
& =\sum_{h\in\CC} x^{L_V(0)}\cY^{(h)}_{x^{-L_V(0)} w_1, x^{-L_V(0)} w_2}(m_1,x)m_2 = \cY(m_1\otimes w_1,x)(m_2\otimes w_2),
\end{align*}
using the definitions. This shows the existence of the required $F$, and the uniqueness follows because $\cY_{M_1,M_2}\otimes\cY_{W_1,W_2}$ is a surjective intertwining operator, which is proved exactly as in \cite[Theorem 5.2]{CKM2}.
\end{proof}

The above theorem implies that if fusion products exist in $\cU$ and $\cV$, then so do fusion products of vector space tensor product modules in $\cD$. We would like to extend existence of fusion products to more general objects of $\cD$. First, assume that $\cD$ is closed under finite direct sums, and define $\til{\cD}$ to be the full subcategory of $\cD$ consisting of finite direct sums of generalized $U\otimes V$-modules $M\otimes W$, where $M\in\mathrm{Ob}(\cU)$ and $W\in\mathrm{Ob}(\cV)$. By Theorem \ref{thm:dist_fus_over_tens} and \cite[Proposition 4.24]{HLZ3}, every pair of objects of $\til{\cD}$ has a fusion product in $\til{\cD}$ if $\cU$ and $\cV$ are both closed under fusion products. Indeed, there is a natural isomorphism
\begin{equation*}
\bigg(\bigoplus_{i_1\in I_1} M_{i_1}\otimes W_{i_1}\bigg)\tens\bigg(\bigoplus_{i_2\in I_2} M_{i_2}\otimes W_{i_2}\bigg)\cong\bigoplus_{(i_1,i_2)\in I_1\times I_2} (M_{i_1}\tens_U M_{i_2})\otimes(W_{i_1}\tens_V W_{i_2})
\end{equation*}
under these assumptions. To get fusion products for all pairs of objects in $\cD$, we now impose some additional conditions:
\begin{assum}\label{assum:fus_2}
The categories $\cU$, $\cV$, and $\cD$ satisfy:
\begin{itemize}

\item Every pair of objects in $\cU$ (or $\cV$) has a fusion product in $\cU$ (or $\cV$).

\item The category $\cD$ is closed under finite direct sums and cokernels, and every object of $\cD$ is isomorphic to a cokernel of a morphism in $\til{\cD}$.

\end{itemize}
\end{assum}

Under this assumption, fusion products of objects of $\til{\cD}$ exist, and we can also take fusion products of morphisms in $\til{\cD}$. That is, for morphisms $f_1: P_1\rightarrow Q_1$ and $f_2: P_2\rightarrow Q_2$ in $\til{\cD}$, 
$f_1\tens f_2: P_1\tens P_2 \rightarrow Q_1\tens Q_2$
is the unique morphism such that
\begin{equation}\label{eqn:fus_prod_of_morph}
(f_1\tens f_2)\circ\cY_{P_1,P_2} =\cY_{Q_1,Q_2}\circ(f_1\otimes f_2),
\end{equation}
induced by the universal property of the fusion product $(P_1\tens P_2,\cY_{P_1,P_2})$, as in \eqref{eqn:fus_prod_of_morph_char}.

\begin{thm}\label{thm:fus_prods_in_C}
Under Assumptions \ref{assum:fus_1} and \ref{assum:fus_2}, every pair of objects in $\cD$ has a fusion product in $\cD$.
\end{thm}
\begin{proof}
Let $X_1, X_2\in\mathrm{Ob}(\cD)$. We already know that $(X_1\tens X_2,\cY_{X_1,X_2})$ exists if $X_1,X_2\in\mathrm{Ob}(\til{\cD})$. In general, by Assumption \ref{assum:fus_2},  there are right exact sequences
\begin{equation*}
Q_i \xrightarrow{q_i} P_i \xrightarrow{p_i} X_i\longrightarrow 0
\end{equation*}
for $i=1,2$ such that $P_i, Q_i\in\mathrm{Ob}(\til{\cD})$. We then define $X_1\tens X_2$ to be the cokernel of
\begin{equation}\label{eqn:F_tens}
(q_1\tens\Id_{P_2})\circ\pi_1 +(\Id_{P_1}\tens q_2)\circ\pi_2: (Q_1\tens P_2)\oplus(P_1\tens Q_2) \longrightarrow P_1\tens P_2,
\end{equation}
where $\pi_1$ and $\pi_2$ are the projections to the first and second summands. Let $p: P_1\tens P_2\rightarrow X_1\tens X_2$ be the surjective cokernel morphism. Now consider the diagram:
\begin{equation*}
\xymatrixcolsep{1.5pc}
\xymatrix{
(Q_1\otimes P_2)\oplus(P_1\otimes Q_2) \ar[r]^(.6){F_\otimes} \ar[d]^{\cY_{Q_1,P_2}\oplus\cY_{P_1,Q_2}} & P_1\otimes P_2 \ar[r]^(.475){p_1\otimes p_2} \ar[d]^{\cY_{P_1,P_2}}  & X_1\otimes X_2 \ar[r] \ar@{-->}[d]^{\exists\,!\,\cY_{X_1,X_2}} & 0\\
((Q_1\tens P_2)\oplus(P_1\tens Q_2))[\log x]\lbrace x\rbrace \ar[r]^(.6){F_\tens}  & (P_1\tens P_2)[\log x]\lbrace x\rbrace \ar[r]^(.475){p} & (X_1\tens X_2)[\log x]\lbrace x\rbrace \ar[r] & 0\\
}
\end{equation*}
where $F_\tens$ is the morphism \eqref{eqn:F_tens} and $F_\otimes$ is the linear map obtained by replacing all instances of $\tens$ in \eqref{eqn:F_tens} with $\otimes$. Both rows of the diagram are right exact, and the left square commutes by \eqref{eqn:fus_prod_of_morph}. So the unique linear map
\begin{equation*}
\cY_{X_1,X_2}: X_1\otimes X_2\longrightarrow(X_1\tens X_2)[\log x]\lbrace x\rbrace
\end{equation*}
satisfying
\begin{equation*}
p\circ\cY_{P_1,P_2} = \cY_{X_1,X_2}\circ(p_1\otimes p_2)
\end{equation*}
is induced by the universal property of the cokernel $(X_1\otimes X_2, p_1\otimes p_2)$. Because $\cY_{P_1,P_2}$ satisfies lower truncation, the Jacobi identity, and the $L(-1)$ derivative property, and because $p_1$, $p_2$, and $p$ are $U\otimes V$-module homomorphisms, it is straightforward to show that $\cY_{X_1,X_2}$ is an intertwining operator of type $\binom{X_1\tens X_2}{X_1\,X_2}$.

We show that $(X_1\tens X_2,\cY_{X_1,X_2})$ satisfies the universal property of a fusion product in $\cD$. First, $X_1\tens X_2$ is an object of $\cD$ because $\cD$ is closed under cokernels. Now if $\cY$ is an intertwining operator of type $\binom{X_3}{X_1\,X_2}$ where $X_3\in\mathrm{Ob}(\cD)$, then $\cY\circ(p_1\otimes p_2)$ is an intertwining operator of type $\binom{X_3}{P_1\,P_2}$, so there is a unique morphism $G: P_1\tens P_2\rightarrow X_3$ such that 
\begin{equation*}
G\circ\cY_{P_1,P_2}=\cY\circ(p_1\otimes p_2).
\end{equation*}
Then using the above commutative diagram,
\begin{align*}
G\circ F_\tens\circ(\cY_{Q_1,P_2}\oplus\cY_{P_1,Q_2}) & = G\circ\cY_{P_1,P_2}\circ F_\otimes =\cY\circ(p_1\otimes p_2)\circ F_\otimes =0.
\end{align*}
Since $\cY_{Q_1,P_2}\oplus\cY_{P_1,Q_2}$ is surjective, it follows that $G\circ F_\tens=0$ as well, and then the universal property of the cokernel $(X_1\tens X_2, p)$ implies that there is a unique map $F: X_1\tens X_2\rightarrow X_3$ such that $F\circ p = G$. Now we get
\begin{align*}
F\circ\cY_{X_1,X_2}\circ(p_1\otimes p_2) = F\circ p\circ\cY_{P_1,P_2} =G\circ\cY_{P_1,P_2}=\cY\circ(p_1\otimes p_2).
\end{align*}
Since $p_1\otimes p_2$ is surjective, it follows that $F\circ\cY_{X_1,X_2}=\cY$. Moreover, $F$ is the unique map with this property because $\cY_{X_1,X_2}$ is surjective, which follows from surjectivity of $p$ and $\cY_{P_1,P_2}$. This proves that $X_1$ and $X_2$ have a fusion product in $\cD$.
\end{proof}

We now discuss categories $\cU$, $\cV$, and $\cD$ which satisfy Assumptions \ref{assum:fus_1} and \ref{assum:fus_2}. Let $\cC^1_U$ and $\cC^1_V$ be the categories of $C_1$-cofinite grading-restricted generalized $U$- and $V$-modules, respectively. By \cite[Main Theorem]{Mi-C1-cofin}, every pair of objects in $\cC^1_U$ or $\cC^1_V$ has a $C_1$-cofinite fusion product. In general, $\cC^1_U$ or $\cC^1_V$ might not be locally finite abelian categories since it is not clear whether these categories are closed under submodules, or whether all objects in these categories have finite length (though by Theorem \ref{thm:C1_cofin_prop}, these properties hold if $\cC_U^1$ and $\cC_V^1$ are closed under contragredients). Thus we restrict to locally finite abelian subcategories $\cU\subseteq\cC^1_U$ and $\cV\subseteq\cC^1_V$, and then we take $\cD=\cD(\cU,\cV)$ as in Definition \ref{def:D(U,V)}.

\begin{cor}\label{cor:UV_closed_under_fusion}
If $\cU\subseteq\cC^1_U$ and $\cV\subseteq\cC^1_V$ are locally finite abelian subcategories which are closed under subquotients and fusion products, then every pair of objects in $\cD(\cU,\cV)$ has a fusion product in $\cD(\cU,\cV)$. In particular, for any $M_1, M_2\in\mathrm{Ob}(\cU)$ and $W_1, W_2\in\mathrm{Ob}(\cV)$,
\begin{equation*}
((M_1\tens_U M_2)\otimes(W_1\tens_V W_2),\cY_{M_1,M_2}\otimes\cY_{W_1,W_2})
\end{equation*}
is a fusion product of $M_1\otimes W_1$ and $M_2\otimes W_2$ in $\cD(\cU,\cV)$.
\end{cor}
\begin{proof}
For the first condition in Assumption \ref{assum:fus_1}, any $X\in\mathrm{Ob}(\cD(\cU,\cV))$ decomposes into generalized $L_U(0)$-eigenspaces because $X$ is the union of $U$-submodules which are objects of $\cU$, and all objects of $\cU$ are generalized $U$-modules. Similarly, $X$ decomposes into generalized $L_V(0)$-eigenspaces. The second condition in Assumption \ref{assum:fus_1} holds by Proposition \ref{prop:tens_prod_in_D(U,V)}, and the third holds because all objects of $\cU$ and $\cV$ have finite length. For the fourth condition, suppose $\cY$ is a $U$-module intertwining operator of type $\binom{X}{M_1\,M_2}$ where $M_1, M_2\in\mathrm{Ob}(\cU)$ and $X\in\mathrm{Ob}(\cD(\cU,\cV))$. Since $\cU\subseteq\cC_U^1$, $\im\,\cY$ is $C_1$-cofinite by \cite[Corollary 2.12]{CMY1} (which is a slight strengthening of \cite[Key Theorem]{Mi-C1-cofin}). Thus there is a surjective morphism $M_1\tens_U M_2\rightarrow\im\cY$ where $M_1\tens_U M_2$ is the fusion product of $M_1$ and $M_2$ in $\cC^1_U$. Since $\cU$ is closed under fusion products and quotients, $\im\cY$ is an object of $\cU$. Thus $\cU$ and $\cD(\cU,\cV)$ satisfy the fourth condition of Assumption \ref{assum:fus_1}, and similarly so do $\cV$ and $\cD(\cU,\cV)$. So the second conclusion of the corollary follows from Theorem \ref{thm:dist_fus_over_tens}.

Now we are assuming the first condition in Assumption \ref{assum:fus_2}, and the second holds by Corollary \ref{cor:cokernel_of_D_tilde}. Thus the first conclusion of the corollary follows from Theorem \ref{thm:fus_prods_in_C}.
\end{proof}

\subsection{Braided tensor category structure}

We continue under Assumption \ref{assum:fus_1}, and we also assume that the category $\cD$ of generalized $U\otimes V$-modules is closed under finite direct sums, and that $\cU$ and $\cV$ both admit the vertex algebraic braided monoidal structure of \cite{HLZ1}-\cite{HLZ8}. In particular, $U\in\mathrm{Ob}(\cU)$ and $V\in\mathrm{Ob}(\cV)$, so that $U\otimes V\in\mathrm{Ob}(\cD)$. We will first show that the $\CC$-linear additive subcategory $\til{\cD}\subseteq\cD$ consisting of finite direct sums of modules $M\otimes W$ for $M\in\mathrm{Ob}(\cU)$, $W\in\mathrm{Ob}(\cV)$ admits the braided monoidal structure of \cite{HLZ1}-\cite{HLZ8}. Then we will show that under further conditions, $\cD$ is also a braided monoidal (or tensor) category.

First, by Theorem \ref{thm:dist_fus_over_tens} and \cite[Proposition 4.24]{HLZ3}, $\til{\cD}$ is closed under fusion products. Thus we automatically have fusion products of morphisms in $\til{\cD}$, natural unit isomorphisms, and natural braiding isomorphisms given by  \eqref{eqn:fus_prod_of_morph_char},  \eqref{eqn:unit_char}, and \eqref{eqn:braid_char}, respectively.  In particular, for objects $X_1=\bigoplus_{i_1\in I_1} M_{i_1}\otimes W_{i_1}$ and $X_2=\bigoplus_{i_2\in I_2} M_{i_2}\otimes W_{i_2}$ of $\til{\cD}$, we take
\begin{equation}\label{eqn:fus_prod_in_Ctilde}
X_1\tens X_2 =\bigoplus_{(i_1, i_2)\in I_1\times I_2} (M_{i_1}\tens_U M_{i_2})\otimes(W_{i_1}\tens_V W_{i_2})
\end{equation}
with fusion product intertwining operator $\cY_{X_1,X_2}=\sum_{(i_1,i_2)\in I_1\times I_2} \cY_{X_1,X_2}^{(i_1,i_2)}$, where $\cY_{X_1,X_2}^{(i_1,i_2)}$ is the composition
\begin{align*}
X_1\otimes X_2 & \xrightarrow{\pi_{X_1}^{(i_1)}\otimes\pi_{X_2}^{(i_2)}}  (M_{i_1}\otimes W_{i_1})\otimes(M_{i_2}\otimes W_{i_2})\nonumber\\ & \xrightarrow{\cY_{M_{i_1},M_{i_2}}\otimes\cY_{W_{i_1},W_{i_2}}} \left( (M_{i_1}\tens_U M_{i_2})\otimes (W_{i_1}\tens_V W_{i_2})\right)[\log x]\lbrace x\rbrace\nonumber\\
& \xrightarrow{q_{X_1\tens X_2}^{(i_1,i_2)}} (X_1\tens X_2)[\log x]\lbrace x\rbrace;
\end{align*}
here the $\pi$'s and $q$'s are the obvious projections and inclusions.

Under this identification of fusion products in $\til{\cD}$, it is straightforward from \eqref{eqn:fus_prod_of_morph_char}, \eqref{eqn:unit_char}, and \eqref{eqn:braid_char}, together with $e^{xL(-1)}=e^{xL_U(-1)}\otimes e^{xL_V(-1)}$ on a vector space tensor product module, that the fusion product of morphisms, unit isomorphisms, and braiding isomorphisms in $\til{\cD}$ are given as follows:
\begin{itemize}
\item For morphisms $\lbrace f_{i_1}\rbrace_{i_1\in I_1}$, $\lbrace f_{i_2}\rbrace_{i_2\in I_2}$ in $\cU$ and $\lbrace g_{i_1}\rbrace_{i_1\in I_2}$, $\lbrace g_{i_2}\rbrace_{i_2\in I_2}$ in $\cV$,
\begin{equation}\label{eqn:fus_prod_of_morph_in_Ctilde}
\bigg(\bigoplus_{i_1\in I_1} f_{i_1}\otimes g_{i_1}\bigg)\tens\bigg(\bigoplus_{i_2\in I_2} f_{i_2}\otimes g_{i_2}\bigg) = \bigoplus_{(i_1,i_2)\in I_1\times I_2} (f_{i_1}\tens_U f_{i_2})\otimes(g_{i_1}\tens_V g_{i_2}).
\end{equation}

\item For any object $X=\bigoplus_{i\in I} M_i\otimes W_i$ in $\til{\cD}$,
\begin{equation}\label{eqn:unit_in_Ctilde}
l_X=\bigoplus_{i\in I} l^U_{M_i}\otimes l^V_{W_i},\qquad r_X=\bigoplus_{i\in I} r^U_{M_i}\otimes r^V_{W_i},
\end{equation}
where $l_{M_i}^U$, $r_{M_i}^U$ and $l_{W_i}^V$, $r_{W_i}^V$ are unit isomorphisms in $\cU$ and $\cV$, respectively.

\item For any objects $X_1=\bigoplus_{i_1\in I_1} M_{i_1}\otimes W_{i_1}$ and $X_2=\bigoplus_{i_2\in I_2} M_{i_2}\otimes W_{i_2}$ in $\til{\cD}$,
\begin{equation}\label{eqn:braiding_in_Ctilde}
\cR_{X_1,X_2} =\bigoplus_{(i_1,i_2)\in I_1\times I_2}\cR^U_{M_{i_1},M_{i_2}}\otimes\cR^V_{W_{i_1},W_{i_2}},
\end{equation}
where $\cR^U_{M_{i_1},M_{i_2}}$ and $\cR^V_{W_{i_1},W_{i_2}}$ are braiding isomorphisms in $\cU$ and $\cV$, respectively.
\end{itemize}
We can define associativity isomorphisms in $\til{\cD}$ similarly. If $X_1=\bigoplus_{i_1\in I_1} M_{i_1}\otimes W_{i_1}$, $X_2=\bigoplus_{i_2\in I_2} M_{i_2}\otimes W_{i_2}$, and $X_3=\bigoplus_{i_3\in I_3} M_{i_3}\otimes W_{i_3}$ are objects of $\til{\cD}$, then
\begin{align*}
X_1\tens(X_2\tens X_3) & =\bigoplus_{(i_1,i_2,i_3)\in I_1\times I_2\times I_3} \left(M_{i_1}\tens_U(M_{i_2}\tens_U M_{i_3})\right)\otimes\left(W_{i_1}\tens_V(W_{i_2}\tens_V W_{i_3})\right),\nonumber\\
(X_1\tens X_2)\tens X_3 & =\bigoplus_{(i_1,i_2,i_3)\in I_1\times I_2\times I_3} \left((M_{i_1}\tens_U M_{i_2})\tens_U M_{i_3}\right)\otimes\left((W_{i_1}\tens_V W_{i_2})\tens_V W_{i_3}\right),
\end{align*}
so we can define
\begin{equation}\label{eqn:assoc_in_Ctilde}
\cA_{X_1,X_2,X_3} =\bigoplus_{(i_1,i_2,i_3)\in I_1\times I_2\times I_3} \cA_{M_{i_1},M_{i_2},M_{i_3}}^U\otimes\cA_{W_{i_1},W_{i_2},W_{i_3}}^V,
\end{equation}
where $\cA_{M_{i_1},M_{i_2},M_{i_3}}^U$ and $\cA_{W_{i_1},W_{i_2},W_{i_3}}^V$ are the associativity isomorphisms in $\cU$ and $\cV$.

\begin{thm}\label{thm:Ctilde_vertex_monoidal}
Under Assumption \ref{assum:fus_1}, assume also that $\cD$ is closed under finite direct sums and that $\cU$ and $\cV$ admit the braided monoidal category structure of \cite{HLZ8}. Then the $\CC$-linear additive full subcategory $\til{\cD}\subseteq\cD$, whose objects are finite direct sums of vector space tensor products of modules in $\cU$ and $\cV$, admits the braided monoidal category structure of \cite{HLZ8}, with fusion products as in \eqref{eqn:fus_prod_in_Ctilde} and \eqref{eqn:fus_prod_of_morph_in_Ctilde}, unit isomorphisms \eqref{eqn:unit_in_Ctilde}, braiding isomorphisms \eqref{eqn:braiding_in_Ctilde}, and associativity isomorphisms \eqref{eqn:assoc_in_Ctilde}.
\end{thm}
\begin{proof}
We first check that the associativity isomorphisms in $\til{\cD}$ are those specified in \cite{HLZ6, HLZ8}, that is, they satisfy \eqref{eqn:assoc_char}. 
To prove that the associativity isomorphism \eqref{eqn:assoc_in_Ctilde} satisfies \eqref{eqn:assoc_char}, take any $(i_1,i_2,i_3)\in I_1\times I_2\times I_3$ and any $m_{i_1}\otimes w_{i_1}\in M_{i_1}\otimes W_{i_1}\subseteq X_1$, $m_{i_2}\otimes w_{i_2}\in M_{i_2}\otimes W_{i_2}\subseteq X_2$, and $m_{i_3}\otimes w_{i_3}\in M_{i_3}\otimes W_{i_3}\subseteq X_3$. Then
\begin{align*}
&\overline{\cA_{X_1,X_2,X_3}}\left(\cY_{X_1,X_2\tens X_3}(m_{i_1}\otimes w_{i_1}, r_1)\cY_{X_2,X_3}(m_{i_2}\otimes w_{i_2}, r_2)(m_{i_3}\otimes w_{i_3})\right)\nonumber\\
& = \overline{\cA_{X_1,X_2,X_3}\circ q_{X_1\tens(X_2\tens X_3)}^{(i_1,i_2,i_3)}}\left(\cY_{M_{i_1},M_{i_2}\tens_U M_{i_3}}(m_{i_1},r_1)\cY_{M_{i_2},M_{i_3}}(m_{i_2},r_2)m_{i_3}\otimes \right.\nonumber\\
&\hspace{10em}  \left. \otimes\cY_{W_{i_1},W_{i_2}\tens_V W_{i_3}}(w_{i_1},r_1)\cY_{W_{i_2},W_{i_3}}(w_{i_2},r_2)w_{i_3}\right)\nonumber\\
& = \overline{q_{(X_1\tens X_2)\tens X_3}^{(i_1,i_2,i_3)}}\left(\cY_{M_{i_1}\tens_U M_{i_2}, M_{i_3}}(\cY_{M_{i_1},M_{i_2}}(m_{i_1},r_1-r_2)m_{i_2},r_2)m_{i_3}\otimes \right.\nonumber\\
&\hspace{10em} \left. \otimes\cY_{W_{i_1}\tens_V W_{i_2}, W_{i_3}}(\cY_{W_{i_2},W_{i_3}}(w_{i_1},r_1-r_2)w_{i_2},r_2)w_{i_3}\right)\nonumber\\
& = \cY_{X_1\tens X_2,X_3}(\cY_{X_1,X_2}(m_{i_1}\otimes w_{i_1}, r_1-r_2)(m_{i_2}\otimes w_{i_2}), r_2)(m_{i_3}\otimes w_{i_3})
\end{align*}
for $r_1,r_2\in\RR_+$ such that $r_1>r_2>r_1-r_2$. In particular, all the above compositions of intertwining operators are absolutely convergent since they are tensor products of convergent compositions of intertwining operators in $\cU$ and $\cV$. Thus $\til{\cD}$ admits the vertex algebraic associativity isomorphisms of \cite{HLZ6, HLZ8}, and it is straightforward to show from \eqref{eqn:assoc_char} that these associativity isomorphisms are natural.

We also need to check that $\til{\cD}$ equipped with the fusion products \eqref{eqn:fus_prod_in_Ctilde} and \eqref{eqn:fus_prod_of_morph_in_Ctilde}, unit isomorphisms \eqref{eqn:unit_in_Ctilde}, braiding isomorphisms \eqref{eqn:braiding_in_Ctilde}, and associativity isomorphisms \eqref{eqn:assoc_in_Ctilde} is a braided monoidal category (and in fact admits the braided monoidal category structure of \cite{HLZ8}). It remains to check the triangle, pentagon, and hexagon identities. But these identities are straightforward from the corresponding identities in $\cU$ and $\cV$ together with the characterizations of the structure isomorphisms in $\til{\cD}$. For example, for the triangle identity, given $X_1=\bigoplus_{i_1\in I_1} M_{i_1}\otimes W_{i_1}$ and $X_2=\bigoplus_{i_2\in I_2} M_{i_2}\otimes W_{i_2}$ in $\til{\cD}$, we have
\begin{align*}
(r_{X_1} & \tens\Id_{X_2})\circ\cA_{X_1,U\otimes V,X_2} \nonumber\\ & =\bigg(\bigoplus_{i_1\in I_1} (r_{M_{i_1}}^U\otimes r_{W_{i_1}}^V)\tens\bigoplus_{i_2\in I_2} (\Id_{M_{i_2}}\otimes\Id_{W_{i_2}})\bigg)\circ\bigoplus_{(i_1',i_2')\in I_1\times I_2} \cA^U_{M_{i_1'}, U, M_{i_2'}}\otimes\cA^V_{W_{i_1'},V,W_{i_2'}}\nonumber\\
& =\bigoplus_{(i_1,i_2)\in I_1\times I_2} (r^U_{M_{i_1}}\tens_U\Id_{M_{i_2}})\circ\cA^U_{M_{i_1}, U, M_{i_2}}\otimes(r^V_{W_{i_1}}\tens_V \Id_{W_{i_2}})\circ\cA^V_{W_{i_1},V,W_{i_2}}\nonumber\\
& =\bigoplus_{(i_1,i_2)\in I_1\times I_2} (\Id_{M_{i_1}}\tens_U l_{M_{i_2}}^U)\otimes (\Id_{W_{i_1}}\tens_V l_{W_{i_2}}^V)\nonumber\\
& =\bigoplus_{i_1\in I_1}(\Id_{M_{i_1}}\otimes\Id_{W_{i_1}})\tens\bigoplus_{i_2\in I_2}(l^U_{M_{i_2}}\otimes l^V_{W_{i_2}}) = \Id_{X_1}\tens l_{X_2}
\end{align*}
using first \eqref{eqn:unit_in_Ctilde} and \eqref{eqn:assoc_in_Ctilde}, then \eqref{eqn:fus_prod_of_morph_in_Ctilde}, then the triangle identities in $\cU$ and $\cV$, then \eqref{eqn:fus_prod_of_morph_in_Ctilde} again, and finally \eqref{eqn:unit_in_Ctilde} again. The pentagon and hexagon identities can be proved similarly.
\end{proof}

To extend the braided monoidal structure from $\til{\cD}$ to $\cD$, we reimpose Assumption \ref{assum:fus_2}:
\begin{thm}\label{thm:C_vertex_monoidal}
Under Assumptions \ref{assum:fus_1} and \ref{assum:fus_2}, if $\cU$ and $\cV$ admit the braided monoidal category structure of \cite{HLZ8}, then so does $\cD$.
\end{thm}
\begin{proof}
By Theorem \ref{thm:fus_prods_in_C}, $\cD$ is closed under fusion products, and thus admits fusion products of morphisms as in \eqref{eqn:fus_prod_of_morph_char}, natural unit isomorphisms as in \eqref{eqn:unit_char}, and natural braiding isomorphisms as in \eqref{eqn:braid_char}. We still need to construct natural associativity isomorphisms in $\cD$ which satisfy the triangle, pentagon, and hexagon identities.

For $X_1, X_2, X_3\in\mathrm{Ob}(\cD)$, there are (as in the proof of Theorem \ref{thm:fus_prods_in_C})  right exact sequences
\begin{equation}\label{eqn:right_exact_Xi}
Q_i \xrightarrow{q_i} P_i \xrightarrow{p_i} X_i\longrightarrow 0
\end{equation}
for $i=1,2,3$, where all $P_i,Q_i\in\mathrm{Ob}(\til{\cD})$. The triple tensor product modules then fit into the following diagram:
\begin{equation*}
\xymatrixcolsep{4pc}
\xymatrix{
Q_{1(23)} \ar[r]^(.4){F} \ar[d]^{\cA\oplus\cA\oplus\cA} & P_1\tens (P_2\tens P_3) \ar[r]^(.48){p_1\tens (p_2\tens p_3)} \ar[d]^{\cA_{P_1,P_2,P_2}} & X_1\tens (X_2\tens X_3) \ar[r] \ar@{-->}[d]^{\exists\,!\,\cA_{X_1,X_2,X_3}} & 0\\
Q_{(12)3} \ar[r]^(.4){G} & (P_1\tens P_2)\tens P_3 \ar[r]^(.48){(p_1\tens p_2)\tens p_3} & (X_1\tens X_2)\tens X_3 \ar[r] & 0\\
}
\end{equation*}
where
\begin{align*}
Q_{1(23)} & = (Q_1\tens(P_2\tens P_3))\oplus (P_1\tens(Q_2\tens P_3))\oplus (P_1\tens(P_2\tens Q_3))\nonumber\\
Q_{(12)3} & = ((Q_1\tens P_2)\tens P_3)\oplus ((P_1\tens Q_2)\tens P_3)\oplus ((P_1\tens P_2)\tens Q_3)
\end{align*}
and
\begin{align*}
F & = (q_1\tens(\Id_{P_2}\tens\Id_{P_3}))\circ\pi_1+(\Id_{P_1}\tens(q_2\tens\Id_{P_3}))\circ\pi_2+(\Id_{P_1}\tens(\Id_{P_2}\tens q_3))\circ\pi_3, \nonumber\\
G & = ((q_1\tens\Id_{P_2})\tens\Id_{P_3})\circ\pi_1+((\Id_{P_1}\tens q_2)\tens\Id_{P_3})\circ\pi_2+((\Id_{P_1}\tens\Id_{P_2})\tens q_3)\circ\pi_3
\end{align*}
(the $\pi$'s refer to the obvious projections). The rows of the diagram are right exact by iterating Lemma \ref{lem:cokernel_of_tens_prod}, and the left square of the diagram commutes because the associativity isomorphisms in $\til{\cD}$ (from Theorem \ref{thm:Ctilde_vertex_monoidal}) are natural. Thus the existence and uniqueness of $\cA_{X_1,X_2,X_3}$ follows from the universal property of the cokernel $(X_1\tens(X_2\tens X_3), p_1\tens(p_2\tens p_3))$. Moreover, $\cA_{X_1,X_2,X_3}$ satisfies \eqref{eqn:assoc_char} because the maps $p_1$, $p_2$, and $p_3$ are surjective, and for all $p_1(b_1)\in X_1$, $p_2(b_2)\in X_2$, and $p_3(b_3)\in X_3$,
\begin{align*}
\overline{\cA_{X_1,X_2,X_3}} & \left(\cY_{X_1,X_2\tens X_3}(p_1(b_1),r_1)\cY_{X_2,X_3}(p_2(b_2),r_2)p_3(b_3)\right)\nonumber\\
& =\overline{\cA_{X_1,X_2,X_3}\circ(p_1\tens(p_2\tens p_3))}\left(\cY_{P_1,P_2\tens P_3}(b_1,r_1)\cY_{P_2,P_3}(b_2,r_2)b_3\right)\nonumber\\
& =\overline{((p_1\tens p_2)\tens p_3)\circ\cA_{P_1,P_2,P_3}}\left(\cY_{P_1,P_2\tens P_3}(b_1,r_1)\cY_{P_2,P_3}(b_2,r_2)b_3\right)\nonumber\\
& =\overline{((p_1\tens p_2)\tens p_3)}\left(\cY_{P_1\tens P_2, P_3}(\cY_{P_1,P_2}(b_1,r_1-r_2)b_2,r_2)b_3\right)\nonumber\\
& =\cY_{X_1\tens X_2,X_3}(\cY_{X_1,X_2}(p_1(b_1),r_1-r_2)p_2(b_2),r_2)p_3(b_3)
\end{align*}
for $r_1,r_2\in\RR_+$ such that $r_1>r_2>r_1-r_2$. It is then straightforward from \eqref{eqn:assoc_char} that these associativity isomorphisms in $\cD$ are natural.

Finally, the unit, associativity, and braiding isomorphisms in $\cD$ should satisfy the triangle, pentagon, and hexagon identities. But this is straightforward since by Theorem \ref{thm:Ctilde_vertex_monoidal} the triangle, pentagon, and hexagon identities are satisfied in $\til{\cD}$. For example, to prove the triangle identity for $X_1,X_2\in\mathrm{Ob}(\cD)$, fix right exact sequences as in \eqref{eqn:right_exact_Xi} and consider:
\begin{equation*}
\xymatrixcolsep{3.2pc}
\xymatrixrowsep{2.5pc}
\xymatrix{
P_1\tens P_2 \ar[d]^{p_1\tens p_2} \ar[r]^(.36){\Id_{P_1}\tens l_{P_2}^{-1}} & P_1\tens((U\otimes V)\tens P_2) \ar[d]^{p_1\tens(\Id_{U\otimes V}\tens p_2)} \ar[r]^{\cA_{P_1,U\otimes V,P_2}} & (P_1\tens(U\otimes V))\tens P_2 \ar[d]^{(p_1\tens\Id_{U\otimes V})\tens p_2} \ar[r]^(.63){r_{P_1}\tens\Id_{P_2}} & P_1\tens P_2 \ar[d]^{p_1\tens p_2} \\
X_1\tens X_2 \ar[r]^(.36){\Id_{X_1}\tens l_{X_2}^{-1}} & X_1\tens((U\otimes V)\tens X_2) \ar[r]^{\cA_{X_1,U\otimes V,X_2}} & (X_1\tens(U\otimes V))\tens X_2 \ar[r]^(.63){r_{X_1}\tens\Id_{X_2}} & X_1\tens X_2 \\
}
\end{equation*}
The diagram commutes by naturality of the unit and associativity isomorphisms in $\cD$, so the triangle identity for $X_1$, $X_2$ follows from the triangle identity for $P_1$, $P_2$ in $\til{\cD}$ and the surjectivity of $p_1\tens p_2$. The pentagon and hexagon identities in $\cD$ can be proved similarly.
\end{proof}

\begin{remark}
Under the assumptions of Theorem \ref{thm:C_vertex_monoidal}, $\cD$ is a braided tensor category if it is an abelian category. Moreover, $\cD$ has a ribbon twist $\theta$ given by $e^{2\pi i L(0)}$. For $M\in\mathrm{Ob}(\cU)$ and $W\in\mathrm{Ob}(\cV)$, we have $\theta_{M\otimes W}=e^{2\pi iL_U(0)}\otimes e^{2\pi iL_V(0)} =\theta_M\otimes\theta_W$.
\end{remark}

\subsection{\texorpdfstring{$\cD(\cU,\cV)$}{D(U,V)} as a braided tensor category}

Suppose $\cU$ and $\cV$ are locally finite braided tensor categories of $U$- and $V$-modules, respectively. Then by Theorem \ref{thm:D(U,V)_is_Del_prod}, the category $\cD(\cU,\cV)$ of $U\otimes V$-modules is a Deligne tensor product of $\cU$ and $\cV$ and thus inherits braided tensor category structure from $\cU$ and $\cV$ as discussed in Section \ref{subsec:prelim_tens_cat}. We will now use Theorem \ref{thm:C_vertex_monoidal} to show that under mild additional conditions, this braided tensor structure on $\cD(\cU,\cV)$ agrees with that of \cite{HLZ1}-\cite{HLZ8} specified by intertwining operators. In contrast with \cite[Theorem 5.5]{CKM2}, we will not require semisimplicity for either $\cU$ or $\cV$ here.

We will need the following general theorem which shows that a functor admits the structure of a tensor functor if it is a tensor functor on a suitable monoidal subcategory. Although this theorem is probably known, for completeness we include a proof in Appendix \ref{app:ext_is_(br)_tens}.
\begin{thm}\label{thm:ext_is_(br)_tens}
Suppose $\cD$ and $\cC$ are (not necessarily rigid) tensor categories with right exact fusion products, $\til{\cD}$ is a monoidal subcategory of $\cD$, and $\cF:\cD\rightarrow\cC$ is a right exact functor such that $\cG=\cF\vert_{\til{\cD}}$ is a tensor functor. Assume also that for any $X_1,X_2\in\mathrm{Ob}(\cD)$, there is an abelian full subcategory $\cD_{X_1,X_2}\subseteq\cD$ that contains $X_1$, $X_2$ and has enough projectives, such that every projective object in $\cD_{X_1,X_2}$ is an object of $\til{\cD}$, and such that $\cD_{X_i,X_i}\subseteq\cD_{X_1,X_2}$ for $i=1,2$. Then $\cF$ is also a tensor functor. If $\cD$ and $\cC$ are also braided and $\cG$ is a braided tensor functor, then $\cF$ is braided, and if $\cD$ and $\cC$ have ribbon twists $\theta$ such that $\cG(\theta_P)=\theta_{\cG(P)}$ for all $P\in\mathrm{Ob}(\til{\cD})$, then $\cF(\theta_X)=\theta_{\cF(X)}$ for all $X\in\mathrm{Ob}(\cD)$.
\end{thm}

We can now prove one of the main results of this paper:
\begin{thm}\label{thm:D(U,V)_braid_equiv}
Let $U$ and $V$ be vertex operator algebras, and let $\cU$ and $\cV$ be categories of grading-restricted generalized $U$-modules and $V$-modules, respectively.
\begin{enumerate}

\item Suppose the following conditions hold:
\begin{itemize}

\item The categories $\cU$ and $\cV$ are closed under subquotients and finite direct sums, and all modules in $\cU$ and $\cV$ have finite length.

\item The categories $\cU$ and $\cV$ admit the vertex algebraic braided tensor category structure of \cite{HLZ1}-\cite{HLZ8}; in particular $U\in\mathrm{Ob}(\cU)$ and $V\in\mathrm{Ob}(\cV)$.

\item If $\mathcal{Y}$ is a $U$-module intertwining operator of type $\binom{X}{M_1\,M_2}$ where $M_1, M_2\in\mathrm{Ob}(\cU)$ and $X\in\mathrm{Ob}(\mathcal{D(\cU,\cV)})$, then $\im\,\cY\in\mathrm{Ob}(\mathcal{U})$.
Similarly, if $\mathcal{Y}$ is a $V$-module intertwining operator of type $\binom{X}{W_1\,W_2}$ where $W_1, W_2\in\mathrm{Ob}(\cV)$ and $X\in\mathrm{Ob}(\mathcal{D}(\cU,\cV))$, then $\im\,\cY\in\mathrm{Ob}(\mathcal{V})$.

\end{itemize}
Then $\cD(\cU,\cV)$ admits the braided tensor category structure with ribbon twist of \cite{HLZ1}-\cite{HLZ8}, and $\cD(\cU,\cV)$ is braided tensor equivalent to the Deligne tensor product $\cU\otimes\cV$ equipped with its braided tensor category structure with ribbon twist inherited from $\cU$ and $\cV$.

\item Suppose the following additional conditions hold:
\begin{itemize}

\item The vertex operator algebras $U$ and $V$ are self-contragredient.

\item The categories $\cU$ and $\cV$ are closed under contragredient modules.

\item All simple modules in $\cU$ and $\cV$ are rigid.

\end{itemize}
Then $\cD(\cU,\cV)$ is a braided ribbon category with duals given by contragredients.

\end{enumerate}

\end{thm}

\begin{proof}
For part (1), the first three conditions of Assumption \ref{assum:fus_1} hold exactly as in the proof of Corollary \ref{cor:UV_closed_under_fusion}, as do the two conditions of Assumption \ref{assum:fus_2}. We are assuming the fourth condition in Assumption \ref{assum:fus_1}, so $\cD(\cU,\cV)$ admits the braided tensor category structure of \cite{HLZ8} by Theorem \ref{thm:C_vertex_monoidal}. Moreover, $\cD(\cU,\cV)\cong\cU\otimes\cV$ as categories by Theorem \ref{thm:D(U,V)_is_Del_prod}, with the Deligne tensor product functor $\otimes:\cU\times\cV\rightarrow\cD(\cU,\cV)$ identified with the tensor product of vector spaces. Thus $\cD(\cU,\cV)$ admits two braided tensor category structures: the vertex algebraic structure of \cite{HLZ8} characterized by intertwining operators (we denote its fusion product by $\tens_{\cD(\cU,\cV)}$), and the braided tensor structure on the Deligne tensor product inherited from the braided tensor structures on $\cU$ and $\cV$ (we denote its fusion product by $\tens_{\cU\otimes\cV}$). We need to show that these two braided tensor category structures are equivalent.

First, comparing \eqref{eqn:fus_prod_in_Del_prod}, \eqref{eqn:fus_prod_of_morph_in_Del_prod}, and \eqref{eqn:fus_prod_in_P} with \eqref{eqn:fus_prod_in_Ctilde} and \eqref{eqn:fus_prod_of_morph_in_Ctilde} (and also using \eqref{eqn:homs_between_Del_prods} and Proposition \ref{prop:proj_cover_homs} together with the $\CC$-bilinearity of $\tens_{\cD(\cU,\cV)}$ and $\tens_{\cU\otimes\cV}$), we see that there is a natural isomorphism 
\begin{equation*}
X_1\tens_{\cD(\cU,\cV)} X_2 \xrightarrow{\cong} X_1\tens_{\cU\otimes\cV} X_2
\end{equation*}
when $X_1, X_2\in\mathrm{Ob}(\til{\cD})$, where $\til{\cD}$ is the subcategory defined in the proof of Theorem \ref{thm:D(U,V)_is_Del_prod}. 
 Moreover, comparing \eqref{eqn:unit_in_Del_prod}, \eqref{eqn:assoc_in_Del_prod}, and \eqref{eqn:braid_in_Del_prod} with \eqref{eqn:unit_in_Ctilde}, \eqref{eqn:assoc_in_Ctilde}, and \eqref{eqn:braiding_in_Ctilde}, and using the $\CC$-bilinearity of $\tens_{\cD(\cU,\cV)}$ and $\tens_{\cU\otimes\cV}$ as well as naturality, this natural isomorphism identifies the unit, associativity, and braiding isomorphisms for $\tens_{\cD(\cU,\cV)}$ with those for $\tens_{\cU\otimes\cV}$. Thus $\Id_{\cD(\cU,\cV)}$ restricted to the monoidal subcategory $\til{\cD}$ has the structure of a braided tensor equivalence between the two braided monoidal structures on $\til{\cD}$. To verify the remaining assumption of Theorem \ref{thm:ext_is_(br)_tens} in this setting, define $\cD_{X_1,X_2}$ for $X_1,X_2\in\mathrm{Ob}(\cD(\cU,\cV))$ by taking $X_2=X_3$ in \eqref{eqn:DX1X2X3_def}. This category has enough projectives, and all of its projective objects are contained in $\til{\cD}$, by Theorem \ref{thm:proj_covers_in_C}. Theorem \ref{thm:ext_is_(br)_tens} now shows that $\cF=\Id_{\cD(\cU,\cV)}$ is a braided tensor equivalence between $\cD(\cU,\cV)$ and $\cU\otimes\cV$; also, the ribbon twists on $\cD(\cU,\cV)$ and $\cU\otimes\cV$ are the same since
\begin{equation*}
\theta_{M\otimes W} = e^{2\pi i L(0)} = e^{2\pi i L_U(0)}\otimes e^{2\pi i L_V(0)} = \theta_M\otimes\theta_W
\end{equation*}
 for $M\in\mathrm{Ob}(\cU)$ and $W\in\mathrm{Ob}(\cV)$. This proves part (1) of the theorem.
 
For Part (2), Proposition \ref{prop:D(U,V)_contra} shows that $\cD(\cU,\cV)$ is closed under contragredients. Then $\cD(\cU,\cV)$ will be rigid (and thus also ribbon) by \cite[Theorem 4.4.1]{CMY2} once we show that every simple module in $\cD(\cU,\cV)$ is rigid. In fact, every simple module in $\cD(\cU,\cV)$ has the form $M\otimes W$ where $M\in\mathrm{Ob}(\cU)$ and $W\in\mathrm{Ob}(\cV)$ are simple. By assumption, $M$ and $W$ are rigid, and their duals are necessarily the contragredients $M'$ and $W'$ (see for example the discussion in \cite[Section 4.3]{CMY2}). Since $\cU$ and $\cV$ are closed under contragredients, $M'\otimes W'$ is an object of $\cD(\cU,\cV)$, and it follows from the characterization of fusion products, unit isomorphisms, and associativity isomorphisms in $\cD(\cU,\cV)$ that $M'\otimes W'$ is the dual of $M\otimes W$, with evaluation and coevaluation given by
 \begin{align*}
 e_{M\otimes W} & = e_M\otimes e_W: (M'\otimes W')\tens (M\otimes W) =(M'\tens_U M)\otimes(W'\tens_V W)\longrightarrow U\otimes V, \nonumber\\
 i_{M\otimes W} & = i_M\otimes i_W: U\otimes V\longrightarrow (M\otimes W)\tens (M'\otimes W') =(M\tens_U M')\otimes(W\tens_V W').
 \end{align*}
 Thus $M\otimes W$ is rigid in $\cD(\cU,\cV)$, completing the proof of the theorem.
\end{proof}

For examples of $\cU$ and $\cV$ that satisfy the conditions of Theorem \ref{thm:D(U,V)_braid_equiv}, we take $\cU$ and $\cV$ to be subcategories of the $C_1$-cofinite categories $\cC^1_U$ and $\cC^1_V$, as in Corollary \ref{cor:UV_closed_under_fusion}. Recall from Section \ref{subsec:prelim_VOA} that $\cC_U^1$ and $\cC_V^1$ admit the braided tensor category structure of \cite{HLZ8} if they are closed under contragredient modules, and thus in this setting, so do subcategories $\cU$ and $\cV$ which are closed under fusion products.
\begin{cor}\label{cor:C1_subcat_braid_tens}
Let $U$ and $V$ be vertex operator algebras such that the $C_1$-cofinite categories $\cC_U^1$ and $\cC_V^1$ are closed under contragredients.
\begin{enumerate}
\item Let $\cU\subseteq\cC_U^1$ and $\cV\subseteq\cC_V^1$ be locally finite abelian categories of grading-restricted generalized $U$-modules and $V$-modules which contain $U$ and $V$, respectively, and are closed under subquotients and fusion products. Then $\cD(\cU,\cV)$ admits the braided tensor category structure with ribbon twist of \cite{HLZ1}-\cite{HLZ8} and is braided tensor equivalent to $\cU\otimes\cV$.

\item If in addition $U$ and $V$ are self-contragredient, $\cU$ and $\cV$ are closed under contragredients, and all simple objects of $\cU$ and $\cV$ are rigid, then $\cD(\cU,\cV)$ is a braided ribbon category with duals given by contragredients.
\end{enumerate}
\end{cor}
\begin{proof}
The first two conditions in part (1) of Theorem \ref{thm:D(U,V)_braid_equiv} hold by the discussion preceding the corollary and because $\cU$ and $\cV$ are locally finite abelian categories. The third condition holds also, by the proof of Corollary \ref{cor:UV_closed_under_fusion}, so the first conclusion follows from part (1) of Theorem \ref{thm:D(U,V)_braid_equiv}. The second conclusion follows from part (2) of Theorem \ref{thm:D(U,V)_braid_equiv}.
\end{proof}

We use Theorem \ref{thm:D(U,V)_for_C1_cofinite} (recall also Corollary \ref{cor:C1_UxV}) to specialize $\cU=\cC_U^1$ and $\cV=\cC_V^1$:
\begin{thm}\label{thm:C1_UxV}
If $U$ and $V$ are $\NN$-graded and $\cC_U^1$ and $\cC_V^1$ are closed under contragredients, then $\cC_{U\otimes V}^1$, equipped with the braided tensor category structure of \cite{HLZ1}-\cite{HLZ8}, is braided tensor equivalent to $\cC_U^1\otimes\cC_V^1$. If in addition $U$ and $V$ are self-contragredient and all simple $C_1$-cofinite $U$- and $V$-modules are rigid, then $\cC_{U\otimes V}^1$ is a braided ribbon category.
\end{thm}

\section{\texorpdfstring{$C_2$}{C2}-cofinite examples}\label{sec:C2_cofinite_examples}

If $V$ is an $\NN$-graded $C_2$-cofinite vertex operator algebra, then the category $\rep(V)$ of grading-restricted generalized $V$-modules is a finite abelian category and a braided tensor category \cite{Hu-C2}, and it is easy to see from the spanning set of \cite[Lemma 2.4]{Mi-mod-inv} that $\rep(V)=\cC_V^1$. If $V$ is also simple and self-contragredient, and $\rep(V)$ is a rigid tensor category, then $\rep(V)$ is a (generally non-semisimple) modular tensor category \cite[Main Theorem 1]{McR-rat}. Thus Theorem \ref{thm:C1_UxV} specializes to $C_2$-cofinite vertex operator algebras as follows, where we use induction on $n$ to generalize from the tensor product of two vertex operator algebras to the tensor product of finitely many:
 \begin{thm}\label{thm:Del_prod_of_C2}
 If $V_1, V_2,\ldots, V_n$ are $\NN$-graded $C_2$-cofinite vertex operator algebras, then
 \begin{equation*}
 \rep(V_1\otimes V_2\otimes\cdots\otimes V_n)\cong\rep(V_1)\otimes\rep (V_2)\otimes\cdots\otimes\rep(V_n)
 \end{equation*}
 as braided tensor categories with ribbon twist. If $V_i$ is also simple and self-contragredient for $i=1,2,\ldots, n$, and all simple objects of $\rep(V_i)$ are rigid, then $\rep(V_1\otimes V_2\otimes\cdots \otimes V_n)$ is a not necessarily semisimple modular tensor category.
 \end{thm}

 Probably the best understood $\NN$-graded $C_2$-cofinite vertex operator algebras with non-semisimple modules are the triplet $W$-algebras $\cW_p$ for $p\in\ZZ_{\geq 2}$ \cite{Ka, AM, TW}. We can now use Theorem \ref{thm:Del_prod_of_C2} to study the representation theory of the simple current extensions of tensor products of $\cW_p$ discussed in \cite[Section 4.1.2]{CKM1}. To do so, we recall some notation from \cite{TW}: $X_1^+$ denotes $\cW_p$ considered as a module for itself, and $X_1^-$ is the non-trivial self-dual simple current $\cW_p$-module such that $X_1^-\tens X_1^-\cong X_1^+$ (see \cite[Theorem 35]{TW}). Thus choosing $p_1,p_2,\ldots p_d\in\ZZ_{\geq 2}$, Theorem \ref{thm:dist_fus_over_tens} implies that   $X_1^{\varepsilon_1}\otimes\cdots\otimes X_1^{\varepsilon_d}$ is a self-dual simple current $\cW_{p_1}\otimes\cdots\otimes\cW_{p_d}$-module for any choice of $\varepsilon_i\in\lbrace\pm\rbrace$. We write
 \begin{equation*}
 X_1^S=X_1^{\varepsilon_1}\otimes X_1^{\varepsilon_2}\otimes\cdots\otimes X_1^{\varepsilon_d}
 \end{equation*}
where $S=\lbrace i\mid \varepsilon_i=-\rbrace$.

The power set $P(d)$ of $\lbrace 1,2,\ldots, d\rbrace$ is an $\FF_2$-vector space with addition given by symmetric difference, and an $\FF_2$-subspace $C\subseteq P(d)$ is called a \textit{binary linear code}. Given a binary linear code $C$, we set
\begin{equation*}
\cW_{p_1,\ldots,p_d}^C = \bigoplus_{S\in C} X_1^S.
\end{equation*}
From \cite[Section 4.1.2]{CKM1}, one observes that $\cW_{p_1,\ldots,p_d}^C$ is a simple subalgebra of a certain lattice abelian intertwining algebra (see \cite[Chapter 12]{DL}). The lattice is even, and thus the abelian intertwining algebra is actually a vertex operator algebra, if and only if $\sum_{i\in S} p_i\in 4\ZZ$ for all $S\in C$. Moreover, $\cW^C_{p_1,\ldots,p_d}$ is $\NN$-graded if and only if the lowest conformal weight of $X_1^S$ is a non-negative integer for all $S\in C$. Since the lowest conformal weight of $X_1^{\varepsilon_i}$ is $\frac{3p_i-2}{4}$, the lowest weight of $X_1^S$ is
\begin{equation*}
\sum_{i\in S} \frac{3p_i-2}{4} =\frac{3}{4}\sum_{i\in S} p_i -\frac{\vert S\vert}{2}.
\end{equation*}
Thus assuming $\sum_{i\in S} p_i\in 4\ZZ$, $X_1^S$ is $\NN$-graded if and only if $\vert S\vert\in 2\ZZ$. We conclude that $\cW_{p_1,\ldots, p_d}^C$ is a simple $\NN$-graded self-contragredient $C_2$-cofinite vertex operator algebra if and only if $C$ is an even binary linear code and $\sum_{i\in S} p_i\in 4\ZZ$ for all $S\in C$.

\begin{thm}\label{thm:ext_of_Wp_tens}
Let $C\subseteq P(d)$ be an even binary linear code and suppose $p_1,\ldots,p_d\in\ZZ_{\geq 2}$ satisfy $\sum_{i\in S} p_i\in 4\ZZ$ for all $S\in C$. Then $\rep(\cW_{p_1,\ldots, p_d}^C)$ is a non-semisimple modular tensor category.
\end{thm}
\begin{proof}
$\rep(\cW_{p_1,\ldots, p_d}^C)$ is non-semisimple because each $\rep(W_{p_i})$ is non-semisimple; for example, the lattice vertex operator algebra mentioned above is not semisimple as a $\cW_{p_1,\ldots, p_d}^C$-module. Now in view of \cite[Main Theorem 1]{McR-rat}, it remains to show that the braided tensor category $\rep(\cW_{p_1,\ldots,p_d}^C)$ is rigid. 

First, $\rep(\cW_{p_1}\otimes\cdots\otimes\cW_{p_d})$ is rigid by \cite[Theorem 40]{TW} (see also \cite[Theorem 7.6]{MY}) and Theorem \ref{thm:Del_prod_of_C2}. Then because $\cW_{p_1,\ldots,p_d}^C$ is a simple current extension, $\cW_{p_1}\otimes\cdots\otimes\cW_{p_d}\subseteq\cW_{p_1,\ldots,p_d}^C$ is the fixed-point subalgebra of an automorphism group isomorphic to $\hom(C,\CC^\times)$. By \cite[Proposition 4.15]{McR-orb}, this means that $\cW_{p_1,\ldots,p_d}^C$ has categorical dimension $\vert C\vert\neq 0$ in the ribbon category $\rep(\cW_{p_1}\otimes\cdots\otimes\cW_{p_d})$, and then by \cite[Theorem 3.2]{HKL} and \cite[Lemma 1.20]{KO}, $\cW_{p_1,\ldots,p_d}^C$ is a ``rigid $\rep(\cW_{p_1}\otimes\cdots\otimes\cW_{p_d})$-algebra'' in the sense of \cite[Definition 1.11]{KO}. It then follows from \cite[Theorem 1.15]{KO}, \cite[Theorem 3.4]{HKL}, and \cite[Theorem 3.65]{CKM1} that $\rep(\cW_{p_1,\ldots,p_d}^C)$ is rigid, as required.
\end{proof}

We can now use the detailed tensor structure of each $\rep(\cW_{p_i})$ from \cite{TW} (see also \cite[Section 7]{MY}), Theorem \ref{thm:Del_prod_of_C2}, and the vertex operator algebra extension theory of \cite{CKM1} to classify simple and projective objects in $\rep(\cW_{p_1,\ldots,p_d}^C)$, and to calculate their fusion products. For simplicity, we carry out these details only for one family of examples of $\cW_{p_1,\ldots,p_d}^C$, namely the even symplectic fermion vertex operator algebras $SF_d^+$.

For $d\in\ZZ_+$, the symplectic fermion superalgebra $SF_d$ was first introduced in the physics literature \cite{Ka2, Ka3, GK} and is the affine vertex operator superalgebra associated to a $2d$-dimensional purely odd Lie superalgebra $\mathfrak{h}$ equipped with a symplectic form. The even vertex operator subalgebra $SF_d^+\subseteq SF_d$ is $C_2$-cofinite \cite{Ab}, and $SF_1^+\cong\cW_2$ \cite{Ka3,GK}. In general, $SF_d\cong SF_1^{\otimes d}$ as vertex operator superalgebras. In \cite{Ru}, Runkel constructed a braided tensor category $\cS\cF_d$ which he conjectured to be braided tensor equivalent to $\rep(SF_d^+)$. The category $\cS\cF_d$ is braided tensor equivalent to the finite-dimensional representation category of a factorizable ribbon quasi-Hopf algebra \cite{GR, FGR}, and thus is a non-semisimple modular tensor category.

Theorem \ref{thm:ext_of_Wp_tens} now implies that $\rep(SF_d^+)$ 
is also a non-semisimple modular tensor category, consistent with Runkel's conjecture. Indeed, $SF_1\cong X_1^+\oplus X_1^-$ as an $SF_1^+\cong\cW_2$-module. Thus as a $\cW_2^{\otimes d}$-module,
\begin{equation*}
SF_d\cong SF_1^{\otimes d}\cong\bigoplus_{S\in P(d)} X_1^S,
\end{equation*}
where $X_1^S\subseteq SF_d^+$ if and only if $\vert S\vert\in 2\ZZ$. That is, $SF_d^+ \cong \cW_{2,\ldots, 2}^{E(d)}$
where $E(d)$ is the binary linear code consisting of all even-cardinality subsets of $\lbrace 1,2,\ldots, d\rbrace$. So we obtain:
\begin{cor}\label{cor:symplectic_fermions}
For all $d\in\ZZ_+$, $\rep(SF_d^+)$ equipped with the braided tensor category structure of \cite{HLZ8} is a non-semisimple modular tensor category.
\end{cor}

We now use the tensor structure of $\rep(\cW_2)$ from \cite{TW} (see also \cite[Section 7]{MY}) and the extension theory of \cite{CKM1} to classify simple and projective objects of $\rep(SF_d^+)$ and to compute their fusion products. These results are not entirely new, since simple $SF_d^+$-modules were classified in \cite{Ab}. Also, the fusion rules, that is, the dimensions of spaces of intertwining operators, for all triples of simple $SF_d^+$-modules were determined in \cite{AA}, although these do not completely determine the fusion products of simple $SF_d^+$-modules. In any case, we can derive all such results rather easily using vertex operator algebra extensions.

%
%

To begin, we recall some results on $\rep(\cW_2)$. There are four simple modules in $\rep(\cW_2)$ denoted $X_i^\varepsilon$ for $i=1,2$ and $\varepsilon =\pm$. The two simple modules $X_2^\pm$ are projective, while $X_1^\pm$ have projective covers $P_1^\pm$ of length four, both of which have two composition factors isomorphic to $X_1^+$ and two isomorphic to $X_1^-$. Fusion products of simple and projective $\cW_2$-modules are given by
\begin{equation*}
X_1^{\varepsilon_1}\tens X_i^{\varepsilon_2} \cong X_i^{\varepsilon_1\varepsilon_2},\qquad X_2^{\varepsilon_1}\tens X_2^{\varepsilon_2} \cong P_1^{\varepsilon_1\varepsilon_2}
\end{equation*}
for $i=1,2$, $\varepsilon_1,\varepsilon_2=\pm$, and
\begin{equation*}
X_1^{\varepsilon_1}\tens P_1^{\varepsilon_2} \cong P_1^{\varepsilon_1\varepsilon_2},\qquad X_2^{\varepsilon_1}\tens P_1^{\varepsilon_2} \cong 2\cdot X_2^+\oplus 2\cdot X_2^-
\end{equation*}
for $\varepsilon_1,\varepsilon_2=\pm$. We can then use associativity to calculate
\begin{align*}
P_1^{\varepsilon_1}\tens P_1^{\varepsilon_2} & \cong (X_2^+\tens X_2^{\varepsilon_1})\tens P_1^{\varepsilon_2} \cong X_2^+\tens(X_2^{\varepsilon_1}\tens P_1^{\varepsilon_2})\nonumber\\
& \cong X_2^+\tens(2\cdot X_2^+\oplus 2\cdot X_2^-)\cong 2\cdot P_1^+\oplus 2\cdot P_1^-
\end{align*}
for $\varepsilon_1,\varepsilon_2=\pm$.

Next, we recall from \cite{KO, CKM1} that there is an braided tensor functor of induction
\begin{equation*}
\cF: \rep(\cW_2^{\otimes d})^0 \longrightarrow \rep(SF_d^+),
\end{equation*}
given by $\cF(W)= SF_d^+\tens W$ on objects, where $\rep(\cW_2^{\otimes d})^0\subseteq\rep(\cW_2^{\otimes d})$ is the full subcategory of objects such that $\cF(W)$ is actually an $SF_d^+$-module (and not a ``non-local'' or ``twisted'' $SF_d^+$-module). More specifically (see for example \cite[Proposition 2.65]{CKM1}), $\rep(\cW_2^{\otimes d})^0$ is the M\"{u}ger centralizer of $SF_d^+$: it consists of all $\cW_2^{\otimes d}$-modules $W$ such that the double braiding $\cR^2_{X_1^S, W}$ is the identity on $X_1^S\tens W$ for all $S\in E(d)$. Because $SF_d^+$ is a rigid $\cW_2^{\otimes d}$-module, the induction functor is exact. Another important property of induction is Frobenius reciprocity: there is a natural isomorphism
\begin{equation*}
\hom_{\cW_2^{\otimes d}}(W, X) \cong \hom_{SF_d^+}(\cF(W), X)
\end{equation*}
for any $\cW_2^{\otimes d}$-module $W$ and $SF_d^+$-module $X$.

\begin{thm}\label{thm:sym_ferm_properties}
For $d\in\ZZ_+$, let $\rep(SF_d^+)$ be the non-semisimple modular tensor category of grading-restricted generalized $SF_d^+$-modules.
\begin{enumerate}
\item There are four simple objects in $\rep(SF_d^+)$ up to isomorphism, given by
\begin{equation*}
\cX_i^\varepsilon = \cF(X_i^\varepsilon\otimes X_i^+\otimes\cdots\otimes X_i^+)
\end{equation*}
for $i=1,2$ and $\varepsilon=\pm$.

\item For $\varepsilon=\pm$, the simple module $\cX_2^\varepsilon$ is projective in $\rep(SF_d^+)$, while $\cX_1^\varepsilon$ has a projective cover
\begin{equation*}
\cP_1^\varepsilon = \cF(P_1^\varepsilon\otimes P_1^+\otimes\cdots\otimes P_1^+)
\end{equation*}
of length $2^{2d}$, with $2^{2d-1}$ composition factors isomorphic to $\cX_1^\pm$ for each sign choice.

\item Fusion products of simple and projective modules in $\rep(SF_d^+)$ as as follows:
\begin{equation*}
\cX_1^{\varepsilon_1}\tens\cX_i^{\varepsilon_2} \cong \cX_i^{\varepsilon_1\varepsilon_2},\qquad \cX_1^{\varepsilon_1}\tens\cP_1^{\varepsilon_2} \cong \cP_1^{\varepsilon_1\varepsilon_2}
\end{equation*}
for $i=1,2$ and $\varepsilon_1,\varepsilon_2=\pm$,
\begin{equation*}
\cX_2^{\varepsilon_1}\tens \cX_2^{\varepsilon_2} \cong \cP_1^{\varepsilon_1\varepsilon_2},\qquad \cX_2^{\varepsilon_1}\tens \cP_1^{\varepsilon_2}\cong 2^{2d-1}\cdot \cX_2^+\oplus 2^{2d-1}\cdot \cX_2^-
\end{equation*}
for $\varepsilon_1,\varepsilon_2=\pm$, and
\begin{equation*}
\cP_1^{\varepsilon_1}\tens\cP_1^{\varepsilon_2}\cong 2^{2d-1}\cdot\cP_1^+\oplus 2^{2d-1}\cdot\cP_1^-
\end{equation*}
for $\varepsilon_1,\varepsilon_2=\pm$.
\end{enumerate}
\end{thm}
\begin{proof}
For part (1), \cite[Proposition 4.5]{CKM1} shows that an $SF_d^+$-module is simple if and only if it is the induction of a simple $\cW_2^{\otimes d}$-module in $\rep(\cW_2^{\otimes d})^0$. Thus we first need to determine which simple objects $W$ of $\rep(\cW_2^{\otimes d})$ double-braid trivially with $X_1^S$ for all $S\in E(d)$. Setting $X_1^S= X_1^{\delta_1}\otimes\cdots\otimes X_1^{\delta_d}$ (so $\delta_i=-$ precisely for $i\in S$) and $W=X_{i_1}^{\varepsilon_1}\otimes\cdots\otimes X_{i_d}^{\varepsilon_d}$, we calculate the double braiding using the balancing equation \eqref{eqn:balancing} and Theorem \ref{thm:Del_prod_of_C2}:
\begin{align*}
\cR^2_{X_1^S,W} & = \theta_{X_1^S\tens W}\circ(\theta_{X_1^S}^{-1}\tens\theta_W^{-1}) =\bigotimes_{k=1}^d \theta_{X_1^{\delta_k}\tens X_{i_k}^{\varepsilon_k}}\circ(\theta_{X_1^{\delta_k}}^{-1}\tens\theta_{X_{i_k}^{\varepsilon_k}}^{-1}) = \prod_{k=1}^d e^{2\pi i(h_{i_k}^{\delta_k\varepsilon_k}-h_1^{\delta_k}-h_{i_k}^{\varepsilon_k})},
\end{align*}
where $h_i^\pm=\frac{1}{8}(i-3)(i-3\pm 2)$ is the lowest conformal weight of $X_i^\pm$. Since $\delta_k\varepsilon_k =\varepsilon_k$ if $k\notin S$ and $h_1^\pm\in\ZZ$, it follows that
\begin{equation*}
\cR_{X_1^S,W}^2 =\prod_{k\in S} e^{2\pi i(h_{i_k}^{-\varepsilon_k} -h_{i_k}^{\varepsilon_k})} =\prod_{k\in S} e^{\pi i(-\varepsilon_k(i_k-3))} =\exp\bigg(-\pi i\sum_{k\in S} \varepsilon_k i_k\bigg),
\end{equation*}
where the last equality uses $\vert S\vert\in 2\ZZ$. Thus $\cF(W)$ is a module in $\rep(SF_d^+)$ if and only if $\sum_{k\in S} \varepsilon_k i_k\in 2\ZZ$ for all $S\in E(d)$. This occurs if and only if $i_j=i_k$ for all $1\leq j,k\leq d$, since $\lbrace j, k\rbrace\in E(d)$. Thus any simple $SF_d^+$ is isomorphic to some $\cF(X_i^{\varepsilon_1}\otimes\cdots\otimes X_i^{\varepsilon_d})$ for $i=1,2$ and $\varepsilon_k=\pm$.

Moreover, by Frobenius reciprocity, $\cF(X_i^{\varepsilon_1}\otimes\cdots\otimes X_i^{\varepsilon_d})\cong\cF(X_j^{\delta_1}\otimes\cdots\otimes X_j^{\delta_d})$ if and only if $X_i^{\varepsilon_1}\otimes\cdots\otimes X_i^{\varepsilon_d}$ occurs as a $\cW_2^{\otimes d}$-submodule of
\begin{equation*}
\cF(X_j^{\delta_1}\otimes\cdots\otimes X_j^{\delta_d}) \cong SF_d^+\tens(X_j^{\delta_1}\otimes\cdots\otimes X_j^{\delta_d})\cong\bigoplus_{S\in E(d)} X_1^S\tens(X_j^{\delta_1}\otimes\cdots\otimes X_j^{\delta_d}).
\end{equation*}
Because $\vert S\vert\in 2\ZZ$ for each $S\in E(d)$, $X_i^{\varepsilon_1}\otimes\cdots\otimes X_i^{\varepsilon_d}$ is then a submodule of $\cF(X_j^{\delta_1}\otimes\cdots\otimes X_j^{\delta_d})$ if and only if $i=j$ and the numbers of $\varepsilon_k$ which equal $+$ or $-$ agree modulo $2$ with the numbers of $\delta_k$ which equal $+$ or $-$. Thus each simple $SF_d^+$-module is isomorphic to exactly one of the four simple modules listed in part (1).

For part (2), Frobenius reciprocity implies that $\cF$ takes projective $\cW_2^{\otimes d}$-modules to projective $SF_d^+$-modules (see for example \cite[Lemma 17]{ACKR}). Thus since $P_1^\varepsilon\otimes P_1^+\otimes\cdots\otimes\cP_1^+$ and $X_2^\varepsilon\otimes X_2^+\otimes\cdots\otimes X_2^+$ are projective $\cW_2^{\otimes d}$-modules by Theorem \ref{thm:proj_covers_in_C}, and since $P_1^\varepsilon\otimes P_1^+\otimes\cdots\otimes P_1^+$ is an object of $\rep(\cW_2^{\otimes d})^0$ by the same calculation as for $X_1^\varepsilon\otimes X_1^+\otimes\cdots\otimes X_1^+$, $\cP_1^\varepsilon$ and $\cX_2^\varepsilon$ are projective in $\rep(SF_d^+)$. Also, $\cP_1^\varepsilon$ has length $2^{2d}$ because each $P_1^\pm$ has length $4$ and because $\cF$ is exact and takes simple $\cW_2^{\otimes d}$-modules to simple $SF_d^+$-modules. Moreover, since $P_1^{\pm}$ has two composition factors isomorphic to $X_1^\pm$ for both sign choices, half the composition factors of $P_1^\varepsilon\otimes P_1^+\otimes\cdots\otimes P_1^+$ induce to $\cX_1^+$ and half induce to $\cX_1^-$.

 Now by Proposition \ref{prop:proj_cover_indecomposable}, $\cP_1^\varepsilon$ will be a projective cover of $\cX_1^\varepsilon$ if it is indecomposable and surjects onto $\cX_1^\varepsilon$. To prove this,  we will show that if $\cX$ is a simple $SF_d^+$-module, then there is a non-zero (and one-dimensional) space of homomorphisms $\cP_1^\varepsilon\rightarrow\cX$ if and only $\cX\cong\cX_1^\varepsilon$. Indeed, taking $\cX=\cX_i^\delta$, Frobenius reciprocity and Proposition \ref{prop:proj_cover_homs} imply
\begin{align*}
\hom_{SF_d^+}(\cP_1^\varepsilon, \cX) & \cong\hom_{SF_d^+}(\cF(P_1^\varepsilon\otimes P_1^+\otimes\cdots\otimes P_1^+), \cF(X_i^\delta\otimes X_i^+\otimes\cdots\otimes X_i^+))\nonumber\\ 
& \cong\bigoplus_{S\in E(d)}\hom_{\cW_2^{\otimes d}}(P_1^\varepsilon\otimes P_1^+\otimes\cdots\otimes P_1^+, X_1^S\tens(X_i^\delta\otimes X_i^+\otimes\cdots\otimes X_i^+))\nonumber\\
& \cong\hom_{\cW_2}(P_1^\varepsilon, X_i^\delta)\otimes\bigotimes_{k=2}^d \hom_{\cW_2}(P_1^+, X_i^+).
\end{align*}
This space is indeed non-zero (and spanned by $p_{X_1^\varepsilon}\otimes p_{X_1^+}\otimes\cdots\otimes p_{X_1^+}$) if and only if $i=1$ and $\delta=\varepsilon$. Thus $\cP_1^\varepsilon$ is a projective cover of $\cX_1^\varepsilon$.

For part (3), we use Theorem \ref{thm:dist_fus_over_tens} together with the fact that induction preserves fusion products. Since the calculations are straightforward, we only illustrate them in the most interesting case: for any $\varepsilon_1,\varepsilon_2$,
\begin{align*}
\cX_2^{\varepsilon_1}\tens\cP_1^{\varepsilon_2} & = \cF(X_2^{\varepsilon_1}\otimes X_2^+\otimes\cdots\otimes X_2^+)\tens\cF(P_1^{\varepsilon_2}\otimes P_1^+\otimes\cdots\otimes P_1^+)\nonumber\\
& \cong\cF((X_2^{\varepsilon_1}\tens P_1^{\varepsilon_2})\otimes(X_2^+\tens P_1^+)\otimes\cdots\otimes(X_2^+\tens P_1^+))\nonumber\\
& \cong\cF((2\cdot X_2^+\oplus 2\cdot X_2^-)^{\otimes d}).
\end{align*}
Introducing the notation $X_2^S$ for $S\in P(d)$ analogous to $X_1^S$, we thus have
\begin{align*}
\cX_2^{\varepsilon_1}\tens\cP_1^{\varepsilon_2}\cong 2^d\cdot\bigoplus_{S\in P(d)} \cF(X_2^S).
\end{align*}
Since $\cF(X_2^S)\cong\cX_2^+$ if $\vert S\vert\in 2\ZZ$ and $\cF(X_2^S)\cong \cX_2^-$ if $\vert S\vert\in 2\ZZ+1$, it follows that $\cX_2^{\varepsilon_1}\tens\cP_1^{\varepsilon_2}$ contains $2^{2d-1}$ copies each of $\cX_2^+$ and  $\cX_2^-$.
\end{proof}

\begin{remark}
 Recalling the conjectural equivalence $\rep(SF_d^+)\cong\cS\cF_d$ as braided tensor categories, Theorem \ref{thm:sym_ferm_properties}(3) compares well with the fusion products in $\cS\cF_d$ given in \cite[Theorem 3.13]{Ru}. The dictionary between $\rep(SF_d^+)$ and $\cS\cF_d$ is as follows. As a category,
\begin{equation*}
\cS\cF_d = \rep(\mathfrak{h})\oplus s\cV ec
\end{equation*}
where $\mathfrak{h}$ is a $2d$-dimensional purely odd Lie superalgebra, and all morphisms in $\cS\cF_d$ preserve $\ZZ/2\ZZ$-gradings. Then $\cX_1^+$ and $\cP_1^+$ correspond to the trivial even $\mathfrak{h}$-module $\CC^{1\vert 0}$ and the $\mathfrak{h}$-module $U(\mathfrak{h})$, respectively, while $\cX_2^+$ corresponds to the vector superspace $\CC^{1\vert 0}$. The modules $\cX_1^-$, $\cP_1^-$, and $\cX_2^-$ correspond to the parity reversals of these objects in $\cS\cF_d$.
\end{remark}

\begin{remark}
The conjectural braided tensor equivalence $\rep(SF_d^+)\cong\cS\cF_d$ has recently been proved in the case $d=1$ \cite{CLR,GN}, in which case $SF_1^+\cong\cW_2$, and $\cS\cF_1$ is equivalent to the representation category of a quasi-Hopf modification of the small quantum group $u_i(\mathfrak{sl}_2)$ at a fourth root of unity \cite{GR}. Theorem \ref{thm:Del_prod_of_C2} now suggests an approach to proving the conjectured equivalence for general $d$. Namely, Theorem \ref{thm:Del_prod_of_C2}, together with \cite{HKL, CLR, GN}, shows that $SF_d^+$ is a commutative algebra in the braided tensor category
\begin{equation*}
\rep(\cW_2)^{\otimes d} \cong \rep(u_i(\mathfrak{sl}_2))^{\otimes d}.
\end{equation*}
This means $\rep(SF_d^+)$ is equivalent to the braided tensor category $\rep^{\text{loc}}(A)$ of local modules for a certain commutative algebra $A$ in $\rep(u_i(\mathfrak{sl}_2))^{\otimes d}$. On the other hand, $\cS\cF_d$ is braided tensor equivalent to the representation category of a certain factorizable ribbon quasi-Hopf algebra $Q(d)$ \cite{FGR}. So if it could be shown that $\rep(Q(d))\cong\rep^{\text{loc}}(A)$ as braided tensor categories, then the equivalence $\rep(SF_d^+)\cong\cS\cF_d$ would be proved.
\end{remark}

\appendix

\section{Proof of Theorem \ref{thm:gen_extend_to_right_exact}}\label{app:extend_to_right_exact}

Before proving Theorem \ref{thm:gen_extend_to_right_exact} in full generality, we first prove the special case that $\cD$ has enough projectives, every projective object is contained in $\til{\cD}$, and $\cD_{X_1,X_2,X_3}=\cD$ for all $X_1,X_2,X_3\in\mathrm{Ob}(\cD)$:
\begin{thm}\label{thm:extend_to_right_exact}
Let $\cD$ and $\cC$ be $\CC$-linear abelian categories such that $\cD$ has enough projectives, let $\til{\cD}\subseteq\cD$ be an additive full subcategory that contains all projective objects in $\cD$, and let $\cG: \til{\cD}\rightarrow\cC$ be a $\CC$-linear functor such that for any $X\in\mathrm{Ob}(\til{\cD})$, there is a right exact sequence
\begin{equation*}
Q_X\xrightarrow{q_X} P_X\xrightarrow{p_X} X\longrightarrow 0,
\end{equation*}
where $Q_X$ and $P_X$ are projective in $\cD$, such that $(\cG(X),\cG(p_X))$ is a cokernel of $\cG(q_X)$ in $\cC$. Then there is a unique (up to natural isomorphism) right exact $\CC$-linear functor $\cF: \cD\rightarrow\cC$ such that $\cF\vert_{\til{\cD}} \cong \cG$.
\end{thm}
\begin{proof}
Let $\cP\subseteq\cD$ be the additive full subcategory of projective objects. Since $\cD$ has enough projectives, for any $X\in\mathrm{Ob}(\cD)$, we can fix a right exact sequence
\begin{equation*}
 Q_X \xrightarrow{q_X} P_X \xrightarrow{p_X}  X \longrightarrow 0
\end{equation*}
such that $Q_X, P_X\in\mathrm{Ob}(\cP)$. Because $P_X$ is projective, we can also choose for any morphism $f: X\rightarrow\tilX$ in $\cD$ a morphism $g_f: P_X\rightarrow P_{\tilX}$ such that the diagram
\begin{equation*}
\xymatrixcolsep{3pc}
\xymatrixrowsep{2pc}
\xymatrix{
P_X \ar[d]^{p_X} \ar[r]^{g_f} & P_{\tilX} \ar[d]^{p_{\tilX}}\\
X \ar[r]^f & \tilX\\
}
\end{equation*}
commutes. For $X\in\mathrm{Ob}(\cP)$, we choose $P_X=X$, $p_X=\Id_X$ and $Q_X=0$, $q_X=0$. These choices force $g_f=f$ when $f$ is a morphism in $\cP$. More generally, for $X\in\mathrm{Ob}(\til{\cD})$, we choose $P_X$, $Q_X$, $p_X$, and $q_X$ such that $(\cG(X),\cG(p_X))$ is a cokernel of $\cG(q_X)$.

Now since $\cF$ should be a right exact functor that extends $\cG$, we define $\cF$ on objects by
\begin{equation*}
 \cF(X) = \coker\,\cG(q_X),
\end{equation*}
with cokernel surjection $c_X: \cG(P_X)\rightarrow\cF(X)$. For $X\in\mathrm{Ob}(\til{\cD})$, we take $\cF(X)=\cG(X)$ and $c_X=\cG(p_X)$. For any morphism $f: X\rightarrow\tilX$ in $\cD$, we want $\cF(f)$ to be the unique morphism, induced by the universal property of cokernels, such that the diagram
\begin{equation*}
 \xymatrix{
 \cG(P_X) \ar[r]^{\cG(g_f)} \ar[d]^{c_X} & \cG(P_{\tilX}) \ar[d]^{c_{\tilX}} \\
 \cF(X) \ar[r]^{\cF(f)} & \cF(\tilX)
 }
\end{equation*}
commutes. For $\cF(f)$ to exist, we need $c_{\tilX}\circ\cG(g_f\circ q_X)=0$. To check this, we first have
\begin{equation*}
 p_{\tilX}\circ g_f\circ q_X=f\circ p_X\circ q_X=0,
\end{equation*}
implying
\begin{equation*}
 \mathrm{Im}(g_f\circ q_X)\subseteq\ker p_{\tilX}=\mathrm{Im}\,q_{\tilX}.
\end{equation*}
Thus by projectivity of $Q_X$, there is a morphism $h_f: Q_X\rightarrow Q_{\tilX}$ such that the diagram
\begin{equation*}
\xymatrixcolsep{3pc}
 \xymatrix{
 Q_X \ar[r]^{h_f} \ar[d]^{q_X} & Q_{\tilX} \ar[d]^{q_{\tilX}} \\
 P_X \ar[r]^{g_f} & P_{\tilX}\\
 }
\end{equation*}
commutes. Thus
\begin{equation*}
 c_{\tilX}\circ\cG(g_f\circ q_X)=c_{\tilX}\circ\cG(q_{\tilX}\circ h_f) =0,
\end{equation*}
as required. When $X,\til{X}\in\mathrm{Ob}(\til{\cD})$,
\begin{equation*}
\cF(f)\circ c_X = c_{\til{X}}\circ\cG(g_f)=\cG(p_{\til{X}}\circ g_f)=\cG(f\circ p_X)=\cG(f)\circ c_X,
\end{equation*}
which implies $\cF(f)=\cG(f)$. Thus $\cF\vert_{\til{\cD}}=\cG$ on both objects and morphisms.

We claim that $\cF(f): \cF(X)\rightarrow\cF(\tilX)$ in $\cC$ is independent of the choice of morphism $g_f: P_X\rightarrow P_{\tilX}$ in $\cP$. That is,
we need to show that if $g_f,\tilg_f: P_X\rightarrow P_{\tilX}$ satisfy
 \begin{equation*}
  p_{\tilX}\circ g_f=f\circ p_X=p_{\tilX}\circ\tilg_f,
 \end{equation*}
then
\begin{equation*}
 c_{\tilX}\circ\cG(g_f)=c_{\tilX}\circ\cG(\tilg_f),
\end{equation*}
equivalently $c_{\tilX}\circ\cG(g_f-\tilg_f)=0$ since $\cG$ is $\CC$-linear. Indeed, $p_{\tilX}\circ(g_f-\tilg_f)=0$ yields
\begin{equation*}
 \mathrm{Im}(g_f-\tilg_f)\subseteq\ker p_{\tilX}=\mathrm{Im}\,q_{\tilX}.
\end{equation*}
Thus because $P_X$ is projective in $\cD$, there is a morphism $g: P_X\rightarrow Q_{\tilX}$ such that
\begin{equation*}
 g_f-\tilg_f=q_{\tilX}\circ g.
\end{equation*}
Thus 
\begin{equation*}
 c_{\tilX}\circ\cG(g_f-\tilg_f)=c_{\tilX}\circ\cG(q_{\tilX}\circ g) =0
\end{equation*}
since by definition $c_{\tilX}\circ\cG(q_{\tilX})=0$. This proves the claim.

It is now easy to check that $\cF: \cD\rightarrow\cC$ as defined above is a $\CC$-linear functor such that $\cF\vert_{\til{\cD}}=\cG$. Indeed, $\cF(\Id_X)=\Id_{\cF(X)}$ and $\cF(f_1\circ f_2)=\cF(f)\circ\cF(g)$ because $\cG$ is a functor and because the above claim shows that we may assume $g_{\Id_X}=\Id_{P_X}$ and $g_{f_1\circ f_2}=g_{f_1}\circ g_{f_2}$. Similarly, to show $\cF$ is $\CC$-linear, $\cF(c_1f_1+c_2f_2)=c_1\cF(f_1)+c_2\cF(f_2)$ for $c_1,c_2\in\CC$ because $\cG$ is $\CC$-linear and because we may assume $g_{c_1f_1+c_2f_2}=c_1 g_{f_1}+c_2 g_{f_2}$. To show that $\cF$ is the desired functor, we still need to show that $\cF$ is right exact, and that if $\til{\cF}: \cD\rightarrow\cC$ is any right exact functor such that $\til{\cF}\vert_{\til{\cD}}\cong\cG$, then $\til{\cF}\cong\cF$.

We now show that $\cF$ is right exact. Thus consider a right exact sequence
\begin{equation*}
 X\xrightarrow{f} \tilX \xrightarrow{c} C\rightarrow 0
\end{equation*}
in $\cD$, that is, $(C,c)$ is a cokernel of $f: X\rightarrow\tilX$. We need to show that $(\cF(C),\cF(c))$ is a cokernel of $\cF(f)$ in $\cC$. Recall that $\cF(f)$ and $\cF(c)$ are defined using morphisms $g_f$ and $g_c$ between projective modules in $\cD$, such that the diagram
\begin{equation*}
\xymatrixrowsep{1.5pc}
\xymatrixcolsep{2.5pc}
 \xymatrix{
 & Q_{\tilX} \ar[d]^{q_{\tilX}} & Q_C \ar[d]^{q_C} & \\
 P_X \ar[r]^{g_f} \ar[d]^{p_X} & P_{\tilX} \ar[d]^{p_{\tilX}} \ar[r]^{g_c} & P_C 
 \ar[d]^{p_C}& \\
 X \ar[r]^f & \tilX \ar[r]^c & C  \ar[d] \ar[r] & 0\\
 && 0 &\\
 }
\end{equation*}
commutes, with the bottom row and right column in particular right exact. Because $c\circ p_{\tilX}$ is surjective and $P_C$ is projective, there is a morphism $\tilg_c:P_C\rightarrow P_{\tilX}$ such that
\begin{equation*}
 p_C=c\circ p_{\tilX}\circ\tilg_c.
\end{equation*}
Since we get $0$ when we pre-compose both sides of this equation with $q_C$, and because $p_X$ is surjective,
\begin{equation*}
 \mathrm{Im}(p_{\tilX}\circ\tilg_c\circ q_C)\subseteq\ker c=\mathrm{Im}\,f=\mathrm{Im}(f\circ p_X).
\end{equation*}
Thus because $Q_C$ is projective, we get a morphism $h: Q_C\rightarrow P_X$ such that
\begin{equation*}
 p_{\tilX}\circ\tilg_c\circ q_C=f\circ p_X\circ h =p_{\tilX}\circ g_f\circ h.
\end{equation*}
Then
\begin{equation*}
 \mathrm{Im}(\tilg_c\circ q_C-g_f\circ h)\subseteq\ker p_{\tilX}=\mathrm{Im}\,q_{\tilX}.
\end{equation*}
Again because $Q_C$ is projective, there is a morphism $\tilh: Q_C\rightarrow Q_{\tilX}$ such that
\begin{equation*}
 \tilg_c\circ q_C-g_f\circ h = q_{\tilX}\circ\tilh.
\end{equation*}
In particular, recalling that $(\cF(\tilX),c_{\tilX})$ is the cokernel of $\cG(q_{\tilX})$, we get
\begin{equation}\label{eqn:F_exact_calc}
 c_{\tilX}\circ\cG(\tilg_c\circ q_C-g_f\circ h) =c_{\tilX}\circ\cG(q_{\tilX})\circ\cG(\tilh) =0.
\end{equation}

We can now show that $(\cF(C),\cF(c))$ is a cokernel of $\cF(f)$. Thus suppose that $F: \cF(\tilX)\rightarrow Y$ is a morphism in $\cC$ such that $F\circ\cF(f)=0$. We need to show that there is a unique morphism $G: \cF(C)\rightarrow Y$ such that
\begin{equation*}
 F=G\circ\cF(c).
\end{equation*}
For the existence, \eqref{eqn:F_exact_calc} and the definition of $\cF(f)$ imply that,
\begin{align*}
 F\circ c_{\tilX}\circ\cG(\tilg_c)\circ\cG(q_C) & = F\circ c_{\tilX}\circ\cG(g_f)\circ\cG(h) =F\circ\cF(f)\circ c_X\circ\cG(h)=0.
\end{align*}
Recalling that $(\cF(C),c_C)$ is by definition the cokernel of $\cG(q_C)$, this means there is a unique morphism $G: \cF(C)\rightarrow Y$ such that
\begin{equation*}
 F\circ c_{\tilX}\circ\cG(\tilg_c)=G\circ c_C.
\end{equation*}
The morphism $G$ also satisfies
\begin{align*}
 G\circ\cF(c)\circ c_{\tilX}= G\circ c_C\circ\cG(g_c) =F\circ c_{\tilX}\circ\cG(\tilg_c\circ g_c).
\end{align*}
Thus by surjectivity of $c_{\tilX}$, we will get $G\circ\cF(c)= F$ provided $F\circ c_{\tilX}\circ\cG(\tilg_c\circ g_c)=F\circ c_{\tilX}$. 

To show this, recall that by definition,
\begin{equation*}
c\circ p_{\tilX}\circ\tilg_c\circ g_c=p_C\circ g_c=c\circ p_{\tilX}.
\end{equation*}
Thus using the surjectivity of $p_X$,
\begin{equation*}
 \mathrm{Im}\,p_{\tilX}\circ(\tilg_c\circ g_c-\id_{P_{\tilX}})\subseteq\ker c=\mathrm{Im}\,f=\mathrm{Im}(f\circ p_X)=\mathrm{Im}(p_{\tilX}\circ g_f).
\end{equation*}
Then by projectivity of $P_{\tilX}$, there is a map $\tilg_f: P_{\tilX}\rightarrow P_X$ such that 
\begin{equation*}
 p_{\tilX}\circ(\tilg_c\circ g_c-\id_{P_{\tilX}}) =p_{\tilX}\circ g_f\circ\tilg_f,
\end{equation*}
and then
\begin{equation*}
 \mathrm{Im}(\tilg_c\circ g_c-\id_{P_{\tilX}}-g_f\circ\tilg_f)\subseteq\ker p_{\tilX}=\mathrm{Im}\,q_{\tilX}.
\end{equation*}
Again by projectivity of $P_{\tilX}$, there is a map $k: P_{\tilX}\rightarrow Q_{\tilX}$ such that
\begin{equation*}
 \tilg_c\circ g_c-\id_{P_{\tilX}}-g_f\circ\tilg_f=q_{\tilX}\circ k.
\end{equation*}
So now
\begin{align*}
 F\circ c_{\tilX}\circ\cG(\tilg_c\circ g_c) & =F\circ c_{\tilX}\circ\cG(\id_{P_{\tilX}}+g_f\circ\tilg_f+q_{\tilX}\circ k)\nonumber\\
 & = F\circ c_{\tilX}+F\circ c_{\tilX}\circ\cG(g_f)\circ\cG(\tilg_f)+F\circ c_{\tilX}\circ\cG(q_{\tilX})\circ\cG(k)\nonumber\\
 & =F\circ c_{\tilX}+F\circ\cF(f)\circ c_X\circ\cG(\tilg_f)\nonumber\\
 & =F\circ c_{\tilX},
\end{align*}
since $c_{\tilX}\circ\cG(q_{\tilX})=0$ and $F\circ\cF(f)=0$. This completes the proof that $F=G\circ\cF(c)$.

To show that $G$ is unique, it is enough to show that any $G$ satisfying $F=G\circ\cF(c)$ satisfies $F\circ c_{\tilX}\circ\cG(\tilg_c)=G\circ c_C$ as well (since $G$ is unique subject to this latter condition). In fact, the condition $F=G\circ\cF(c)$ implies
\begin{equation*}
 F\circ c_{\tilX}\circ\cG(\tilg_c)=G\circ\cF(c)\circ c_{\tilX}\circ\cG(\tilg_c)=G\circ c_C\circ\cG(g_c\circ\tilg_c),
\end{equation*}
so it is enough to show $c_C\circ\cG(g_c\circ\tilg_c)=c_C$. To show this, note that
\begin{equation*}
 p_C\circ g_c\circ\tilg_c =c\circ p_{\tilX}\circ\tilg_c=p_C
\end{equation*}
by the definitions, so
\begin{equation*}
 \mathrm{Im}(g_c\circ\tilg_c-\id_{P_C})\subseteq\ker p_C=\mathrm{Im}\,q_C.
\end{equation*}
Because $P_C$ is projective, there is thus a map $\tilk: P_C\rightarrow Q_C$ such that 
\begin{equation*}
 g_c\circ\tilg_c-\id_{P_C}=q_C\circ \tilk,
\end{equation*}
and then
\begin{equation*}
 c_C\circ\cG(g_c\circ\tilg_c)=c_C\circ\cG(\id_{P_C}+q_C\circ\tilk)=c_C+c_C\circ\cG(q_C)\circ\cG(\tilk)=c_C,
\end{equation*}
as required. This completes the proof that $(\cF(C),\cF(c))$ is a cokernel of $\cF(f)$ in $\cC$, that is, the functor $\cF$ is right exact.

Finally, we need to show that if $\widetilde{\cF}: \cD\rightarrow\cC$ is any right exact functor such that $\widetilde{\cF}\vert_{\til{\cD}}\cong\cG$, then $\widetilde{\cF}\cong\cF$. Thus fix a natural isomorphism $\beta:\widetilde{\cF}\vert_{\til{\cD}}\rightarrow\cG\rightarrow\cF\vert_{\til{\cD}}$. We will extend $\beta$ to a natural isomorphism $\alpha: \widetilde{\cF}\rightarrow\cF$. As before, for $X\in\mathrm{Ob}(\cD)$, we have a right exact sequence
\begin{equation*}
 Q_X\xrightarrow{q_X} P_X\xrightarrow{p_X} X\longrightarrow 0,
\end{equation*}
with $Q_X,P_X\in\mathrm{Ob}(\cP)\subseteq\mathrm{Ob}(\til{\cD})$, and $\cF(X)=\coker\,\cG(q_X)$ in $\cC$. Because $\widetilde{\cF}$ and $\cF$ are right exact, the commutative diagram
\begin{equation*}
 \xymatrixcolsep{3pc}
 \xymatrix{
 \widetilde{\cF}(Q_X) \ar[d]^{\beta_{Q_X}} \ar[r]^{\widetilde{\cF}(q_X)} & \widetilde{\cF}(P_X) \ar[d]^{\beta_{P_X}} \ar[r]^{\widetilde{\cF}(p_X)} & \widetilde{\cF}(X) \ar[r] \ar@{-->}[d]^{\exists\,!\,\alpha_X} & 0\\
 \cF(Q_X) \ar[r]^{\cF(q_X)} & \cF(P_X) \ar[r]^{\cF(p_X)} & \cF(X) \ar[r] & 0\\
 }
\end{equation*}
has right exact rows, where the existence of $\alpha_X$ follows from the naturality of $\beta$ and the universal property of the cokernel $(\til{\cF}(X),\til{\cF}(p_X))$. Because $\beta$ is a natural isomorphism and $(\cF(X),\cF(p_X)$ is a cokernel of $\cF(q_X)$, it easy to see that $\alpha_X$ is invertible.

To show that the isomorphisms $\alpha_X$ define a natural isomorphism $\alpha$, suppose $f: X\rightarrow\tilX$ is a morphism in $\cD$. Recall that there is a morphism $g_f: P_X\rightarrow P_{\til{X}}$ such that
\begin{equation*}
 f\circ p_X=p_{\tilX}\circ g_f.
\end{equation*}
Thus
\begin{align*}
 \cF(f)\circ\alpha_X\circ\widetilde{\cF}(p_X) & = \cF(f)\circ \cF(p_X)\circ\beta_{P_X} = \cF(p_{\til{X}})\circ\cF(g_f)\circ\beta_{P_X}\nonumber\\
 & = \cF(p_{\tilX})\circ\beta_{P_{\tilX}}\circ\widetilde{\cF}(g_f) = \alpha_{\tilX}\circ\widetilde{\cF}(p_{\tilX}\circ g_f)\nonumber\\
 & =\alpha_{\tilX}\circ\widetilde{\cF}(f\circ p_X) =\alpha_{\tilX}\circ\widetilde{\cF}(f)\circ\widetilde{\cF}(p_X).
\end{align*}
Since the cokernel morphism $\widetilde{\cF}(p_X)$ is surjective, it follows that $\cF(f)\circ\alpha_X=\alpha_{\tilX}\circ\widetilde{\cF}(f)$, that is, $\alpha:\widetilde{\cF}\rightarrow\cF$ is a natural isomorphism. This completes the proof of Theorem \ref{thm:extend_to_right_exact}.
\end{proof}

\begin{remark}\label{rem:beta_extends_uniquely}
The conclusion of the proof of Theorem \ref{thm:extend_to_right_exact} shows that any natural isomorphism $\beta:\til{\cF}\vert_{\til{\cD}}\rightarrow\cF\vert_{\til{\cD}}$ extends uniquely to a natural isomorphism $\alpha:\til{\cF}\rightarrow\cF$. Indeed, $\alpha\vert_{\til{\cD}}=\beta$ because for $X\in\mathrm{Ob}(\til{\cD})$, the diagrams
\begin{equation*}
\xymatrixcolsep{4pc}
\xymatrix{
  \widetilde{\cF}(P_X) \ar[d]^{\beta_{P_X}} \ar[r]^{\widetilde{\cF}(p_X)} & \widetilde{\cF}(X)  \ar[d]^{\alpha_X} \\
  \cF(P_X) \ar[r]^{\cF(p_X)} & \cF(X) \\
}\qquad\qquad \xymatrix{
  \widetilde{\cF}(P_X) \ar[d]^{\beta_{P_X}} \ar[r]^{\widetilde{\cF}(p_X)} & \widetilde{\cF}(X)  \ar[d]^{\beta_X} \\
  \cF(P_X) \ar[r]^{\cF(p_X)} & \cF(X) \\
}
\end{equation*}
commute by the construction of $\alpha_X$ and the naturality of $\beta$, respectively, and thus $\alpha_X=\beta_X$ by surjectivity of $\til{\cF}(p_X)$. The extension $\alpha$ is unique because if $\til{\alpha}:\til{\cF}\rightarrow\cF$ is any natural isomorphism such that $\til{\alpha}\vert_{\til{\cD}}=\beta$, then
\begin{equation*}
\til{\alpha}_X\circ\til{\cF}(p_X)=\cF(p_X)\circ\til{\alpha}_{P_X}=\cF(p_X)\circ\beta_{P_X} = \alpha_X\circ\til{\cF}(p_X)
\end{equation*}
for any $X\in\mathrm{Ob}(\cD)$. Thus $\til{\alpha}_X=\alpha_X$ since $\til{\cF}(p_X)$ is surjective.
\end{remark}

We now return to the general setting of Theorem \ref{thm:gen_extend_to_right_exact}: $\cD$ and $\cC$ are $\CC$-linear abelian categories, $\til{\cD}\subseteq\cD$ is an additive full subcategory, and $\cG: \til{\cD}\rightarrow \cC$ is a $\CC$-linear functor.
For any unordered triple $X_1,X_2,X_3\in\mathrm{Ob}(\cD)$, we have an abelian full subcategory $\cD_{X_1,X_2,X_3}$ that contains $X_1$, $X_2$, $X_3$ and has enough projectives; moreover, all projective objects in $\cD_{X_1,X_2,X_3}$ are objects of $\til{\cD}$, and
\begin{equation*}
\cD_{X_i,X_i,X_i}\subseteq\cD_{X_i,X_i, X_j}=\cD_{X_i,X_j,X_j}\subseteq\cD_{X_1,X_2,X_3}
\end{equation*}
for all $i,j=1,2,3$. To simplify notation, we set $\cD_X=\cD_{X,X,X}$ for any $X\in\mathrm{Ob}(\cD)$, and we set $\cD_{X_1,X_2}=\cD_{X_1,X_1,X_2}=\cD_{X_1,X_2,X_2}$ for any $X_1,X_2\in\mathrm{Ob}(\cC)$.

Now by Theorem \ref{thm:extend_to_right_exact}, for any $X_1,X_2,X_3\in\mathrm{Ob}(\cD)$, there is a unique right exact functor 
$\cF_{X_1,X_2,X_3}:\cD_{X_1,X_2,X_3}\rightarrow\cC$ such that 
\begin{equation*}\cF_{X_1,X_2,X_3}\vert_{\til{\cD}\cap\cD_{X_1,X_2,X_3}}\cong\cG\vert_{\til{\cD}\cap\cD_{X_1,X_2,X_3}}.
\end{equation*}
From Remark \ref{rem:equal_vs_nat_iso}, or from the proof of Theorem \ref{thm:extend_to_right_exact}, we may assume for convenience that this natural isomorphism of functors on $\til{\cD}\cap\cD_{X_1,X_2,X_3}$ is a precise equality. To simplify notation, set $\cF_X=\cF_{X,X,X}$ for $X\in\mathrm{Ob}(\cD)$, and $\cF_{X_1,X_2}=\cF_{X_1,X_1,X_2}=\cF_{X_1,X_2,X_2}$ for $X_1,X_2\in\mathrm{Ob}(\cD)$. Then for $i,j=1,2,3$,
\begin{align*}
(\cF_{X_1,X_2,X_3}\vert_{\cD_{X_i,X_j}})\vert_{\til{\cD}\cap\cD_{X_i,X_j}} & =(\cF_{X_1,X_2,X_3}\vert_{\til{\cD}\cap\cD_{X_1,X_2,X_3}})\vert_{\til{\cD}\cap\cD_{X_i,X_j}}\nonumber\\
& = (\cG\vert_{\til{\cD}\cap\cD_{X_1,X_2,X_3}})\vert_{\til{\cD}\cap\cD_{X_i,X_j}} = \cG\vert_{\til{\cD}\cap\cD_{X_i,X_j}}.
\end{align*}
Thus as in Remark \ref{rem:beta_extends_uniquely}, the identity natural isomorphism
\begin{equation*}
\Id: (\cF_{X_1,X_2,X_3}\vert_{\cD_{X_i,X_j}})\vert_{\til{\cD}\cap\cD_{X_i,X_j}} \longrightarrow\cF_{X_i,X_j}\vert_{\til{\cD}\cap\cD_{X_i,X_j}}
\end{equation*}
extends uniquely to a natural isomorphism
\begin{equation*}
\alpha^{(i,j)}: \cF_{X_1,X_2,X_3}\vert_{\cD_{X_i,X_j}}\longrightarrow\cF_{X_i,X_j}.
\end{equation*}
Similarly, there is a natural isomorphism
\begin{equation*}
\alpha^{(i)}: \cF_{X_i,X_j}\vert_{\cD_{X_i}}\longrightarrow\cF_{X_i}
\end{equation*}
uniquely extending the identity natural isomorphism of
\begin{equation*}
\cF_{X_i,X_j}\vert_{\til{\cD}\cap\cD_{X_i}} =\cG\vert_{\til{\cD}\cap\cD_{X_i}}=\cF_{X_i}\vert_{\til{\cD}\cap\cD_{X_i}}.
\end{equation*}
For any permutation $(i,j,k)$ of $(1,2,3)$, Remark \ref{rem:beta_extends_uniquely} shows that there is an equality of natural isomorphisms
\begin{equation}\label{eqn:nat_isos_coherent}
\alpha^{(i)}\circ\alpha^{(i,j)} = \alpha^{(i)}\circ\alpha^{(i,k)}: \cF_{X_1,X_2,X_3}\vert_{\cD_{X_i}}\longrightarrow\cF_{X_i},
\end{equation}
since both compositions restrict to the identity on $\til{\cD}\cap\cD_{X_i}$.

We now define the desired functor $\cF: \cD\rightarrow\cC$ by $\cF(X)=\cF_X(X)$ for any $X\in\mathrm{Ob}(\cD)$; for a morphism $f: X_1\rightarrow X_2$ in $\cD$, we define $\cF(f)$ by commutativity of the diagram
\begin{equation*}
\xymatrixcolsep{4pc}
\xymatrix{
\cF_{X_1,X_2}(X_1) \ar[r]^{\cF_{X_1,X_2}(f)} \ar[d]^{\alpha^{(1)}_{X_1}} & \cF_{X_1,X_2}(X_2) \ar[d]^{\alpha^{(2)}_{X_2}}\\
\cF_{X_1}(X_1) \ar[r]^{\cF(f)} & \cF_{X_2}(X_2)\\
}
\end{equation*}
Then $\cF(\Id_X)=\Id_{\cF(X)}$ for any $X\in\mathrm{Ob}(\cD)$, and $\cF$ defines a $\CC$-linear map on morphisms because $\cF_{X_1,X_2}$ does for all $X_1,X_2\in\mathrm{Ob}(\cD)$. To show that $\cF$ respects compositions, suppose $f_1: X_1\rightarrow X_2$ and $f_2: X_2\rightarrow X_3$ are morphisms in $\cD$. Then the definitions, the equation \eqref{eqn:nat_isos_coherent}, and naturality of $\alpha^{(1,2)}$ and $\alpha^{(2,3)}$ imply that
\begin{equation*}
\xymatrixcolsep{5.5pc}
\xymatrix{
\cF_{X_1,X_2,X_3}(X_1) \ar[r]^{\cF_{X_1,X_2,X_3}(f_1)} \ar[d]^{\alpha^{(1)}_{X_1}\circ\alpha^{(1,3)}_{X_1}=\alpha^{(1)}_{X_1}\circ\alpha^{(1,2)}_{X_1}} & \cF_{X_1,X_2,X_3}(X_2) \ar[r]^{\cF_{X_1,X_2,X_3}(f_2)} \ar[d]^{\alpha^{(2)}_{X_2}\circ\alpha^{(1,2)}_{X_2}=\alpha^{(2)}_{X_2}\circ\alpha^{(2,3)}_{X_2}} & \cF_{X_1,X_2,X_3}(X_3) \ar[d]^{\alpha^{(3)}_{X_3}\circ\alpha^{(2,3)}_{X_3}=\alpha^{(3)}_{X_3}\circ\alpha^{(1,3)}_{X_3}} \\
\cF_{X_1}(X_1) \ar[r]_{\cF(f_1)} & \cF_{X_2}(X_2) \ar[r]_{\cF(f_2)} & \cF_{X_3}(X_3) \\
}
\end{equation*}
commutes. This diagram, combined with 
\begin{equation*}
\cF_{X_1,X_2,X_3}(f_2)\circ\cF_{X_1,X_2,X_3}(f_1)=\cF_{X_1,X_2,X_3}(f_2\circ f_1)
\end{equation*}
 and the naturality of $\alpha^{(1,3)}$, implies that $\cF(f_2)\circ\cF(f_1)=\cF(f_2\circ f_1)$. Thus $\cF$ is a functor, and the same commutative diagram combined with the right exactness of $\cF_{X_1,X_2,X_3}$ implies that $\cF$ is right exact.
 
We now show that $\cF\vert_{\til{\cD}}=\cG$. For $X\in\mathrm{Ob}(\til{\cD})$, we have
\begin{equation*}
\cF(X)=\cF_X(X)=\cG(X)
\end{equation*}
since $\cF_X\vert_{\til{\cD}\cap\cD_X}=\cG\vert_{\til{\cD}\cap\cD_X}$. For a morphism $f: X_1\rightarrow X_2$ in $\til{\cD}$, we have
\begin{equation*}
\cF(f)=\alpha_{X_2}^{(2)}\circ\cF_{X_1,X_2}(f)\circ(\alpha_{X_1}^{(1)})^{-1}=\Id_{\cG(X_2)}\circ\cG(f)\circ\Id_{\cG(X_1)}=\cG(f)
\end{equation*}
since $\cF_{X_1,X_2}\vert_{\til{\cD}\cap\cD_{X_1,X_2}}=\cG$ and $\alpha^{(i)}\vert_{\til{\cD}\cap\cD_{X_i}}$ is the identity natural isomorphism of $\cG\vert_{\til{\cD}\cap\cD_{X_i}}$ for $i=1,2$. Thus $\cF\vert_{\til{\cD}}=\cG$ on both objects and morphisms.

Finally, we need to show that if $\til{\cF}:\cD\rightarrow\cC$ is any $\CC$-linear right exact functor such that $\til{\cF}\vert_{\til{\cD}}\cong\cG$, then $\til{\cF}\cong\cF$. Thus fix a natural isomorphism $\beta: \til{\cF}\vert_{\til{\cD}}\rightarrow\cG$. For any $X\in\mathrm{Ob}(\cD)$, $\beta$ restricts to a natural isomorphism 
\begin{equation*}
\til{\cF}\vert_{\til{\cD}\cap\cD_X}\longrightarrow\cG\vert_{\til{\cD}\cap\cD_X} =\cF_X\vert_{\til{\cD}\cap\cD_X},
\end{equation*}
so by the uniqueness assertion in Theorem \ref{thm:extend_to_right_exact} and Remark \ref{rem:beta_extends_uniquely}, $\beta\vert_{\til{\cD}\cap\cD_X}$ extends uniquely to a natural isomorphism
\begin{equation*}
\eta^{(X)}: \til{\cF}\vert_{\cD_X}\longrightarrow\cF_X.
\end{equation*}
Thus for all $X\in\mathrm{Ob}(\cD)$, we can define
\begin{equation*}
\eta_X =\eta^{(X)}_X: \til{\cF}(X)\longrightarrow\cF_X(X)=\cF(X).
\end{equation*}
To show that the $\eta_X$ define a natural isomorphism $\eta: \til{\cF}\rightarrow\cF$, let $f: X_1\rightarrow X_2$ be a morphism in $\cD$. Since $\cD_{X_i}$ for $i=1,2$ has enough projectives, we can fix surjections $p_i: P_i\rightarrow X_i$ such that $P_i$ is projective in $\cD_{X_i}$, and thus is also an object of $\til{\cD}$. Since $\alpha^{(i)}$ for $i=1,2$ extends the identity isomorphism of $\cG\vert_{\til{\cD}\cap\cD_{X_i}}$, we have a commutative diagram
\begin{equation*}
\xymatrixcolsep{4pc}
\xymatrix{
\til{\cF}(P_i) \ar[r]^{\beta_{P_i}} \ar[d]^{\til{\cF}(p_i)} & \cG(P_i) \ar[d]^{\cF_{X_i}(p_i)} \ar[rd]^{\cF_{X_1,X_2}(p_i)} &\\
\til{\cF}(X_i) \ar[r]_{\eta_{X_i}} & \cF_{X_i}(X_i) & \cF_{X_1,X_2}(X_i) \ar[l]^{\alpha^{(i)}_{X_i}} \\
}
\end{equation*}
Now fix a surjection $p_{1,2}: P_{1,2}\rightarrow X_1$ such that $P_{1,2}\in\mathrm{Ob}(\til{\cD}\cap\cD_{X_1,X_2})$ is projective in $\cD_{X_1,X_2}$. By projectivity, there are morphisms $\til{p}_1: P_{1,2}\rightarrow P_1$ and $g_f: P_{1,2}\rightarrow P_2$ such that
\begin{equation*}
p_{1,2}= p_1\circ\til{p}_1,\qquad f\circ p_{1,2}=p_2\circ g_f.
\end{equation*}
Then we calculate
\begin{align*}
\cF(f)\circ\eta_{X_1}\circ\til{\cF}(p_{1,2}) & = \alpha^{(2)}_{X_2}\circ\cF_{X_1,X_2}(f)\circ(\alpha^{(1)}_{X_1})^{-1}\circ\eta_{X_1}\circ\til{\cF}(p_1)\circ\til{\cF}(\til{p}_1)\nonumber\\
& =\alpha^{(2)}_{X_2}\circ\cF_{X_1,X_2}(f)\circ\cF_{X_1,X_2}(p_1)\circ\beta_{P_1}\circ\til{\cF}(\til{p}_1)\nonumber\\
& =\alpha^{(2)}_{X_2}\circ\cF_{X_1,X_2}(f\circ p_1)\circ\cG(\til{p}_1)\circ\beta_{P_{1,2}}\nonumber\\
& =\alpha^{(2)}_{X_2}\circ\cF_{X_1,X_2}(p_2)\circ\cG(g_f)\circ\beta_{P_{1,2}}\nonumber\\
& =\cF_{X_2}(p_2)\circ\beta_{P_2}\circ\til{\cF}(g_f)\nonumber\\
& =\eta_{X_2}\circ\til{\cF}(p_2)\circ\til{\cF}(g_f) = \eta_{X_2}\circ\til{\cF}(f)\circ\til{\cF}(p_{1,2}).
\end{align*}
Since $p_{1,2}$ is surjective and $\til{\cF}$ is right exact, $\til{\cF}(p_{1,2})$ is also surjective, and thus $\cF(f)\circ\eta_{X_1}=\eta_{X_2}\circ\til{\cF}(f)$. That is, $\eta$ is a natural isomorphism. This completes the proof of Theorem \ref{thm:gen_extend_to_right_exact}.

\section{Proof of Theorem \ref{thm:ext_is_(br)_tens}}\label{app:ext_is_(br)_tens}

We assume $\cD$ and $\cC$ are tensor categories with right exact fusion products, $\til{\cD}$ is a monoidal subcategory of $\cD$, and $\cF:\cD\rightarrow\cC$ is a right exact functor such that $\cF\vert_{\til{\cD}}$ is a tensor functor. This means there is an isomorphism $\varphi:\cF(\vac_\cD)\rightarrow\vac_\cC$ and a natural isomorphism
\begin{equation*}
G: \tens_\cC\circ(\cF\vert_{\til{\cD}}\times\cF\vert_{\til{\cD}})\rightarrow(\cF\circ\tens_\cD)\vert_{\til{\cD}\times\til{\cD}}
\end{equation*}
which are suitably compatible with the unit and associativity isomorphisms of $\cD$ and $\cC$. We also assume that any $X\in\mathrm{Ob}(\cD)$ is contained in an abelian full subcategory $\cD_{X,X}\subseteq\cD$ which has enough projectives, such that every projective object of $\cD_{X,X}$ is contained in $\til{\cD}$. Thus we may fix a right exact sequence
\begin{equation*}
Q_X\xrightarrow{q_X} P_X\xrightarrow{p_X} X\longrightarrow 0
\end{equation*}
such that $Q_X$ and $P_X$ are projective in $\cD_{X,X}$, and thus are also objects of $\til{\cD}$.

Also, for any morphism $f: X_1\rightarrow X_2$ in $\cD$, we are assuming that $X_1$ and $X_2$ are contained in an abelian full subcategory $\cD_{X_1,X_2}$ which contains $\cD_{X_1,X_1}$ and $\cD_{X_2,X_2}$ and has enough projectives, such that every projective object of $\cD_{X_1,X_2}$ is  contained in $\til{\cD}$. Thus we may fix a surjection $p_f: P_f\twoheadrightarrow X_1$ such that $P_f$ is projective in $\cD_{X_1,X_2}$. By projectivity, there are morphisms $q^{(i)}_{f}: P_f\rightarrow P_{X_i}$ for $i=1,2$ such that the diagram
\begin{equation}\label{eqn:gf_def}
\xymatrixcolsep{4pc}
\begin{matrix}
\xymatrix{
 & P_f \ar[ld]_{q_f^{(1)}} \ar[r]^{q_f^{(2)}} \ar[d]^{p_f} & P_{X_2} \ar[d]^{p_{X_2}} \\
 P_{X_1} \ar[r]^(.52){p_{X_1}} & X_1 \ar[r]^{f} & X_2 \\
}
\end{matrix}
\end{equation}
commutes.

Now for $W, X\in\mathrm{Ob}(\cD)$, we want to define an isomorphism $F_{W,X}: \cF(W)\tens_\cC\cF(X)\rightarrow\cF(W\tens_\cD X)$ by commutativity of the diagram
\begin{equation*}
\xymatrixrowsep{2pc}
\xymatrixcolsep{1.5pc}
\xymatrix{
(\cF(Q_W)\tens_\cC \cF(P_X))\oplus(\cF(P_W)\tens_\cC\cF(Q_X)) \ar[r]^(.55){\Gamma} \ar[d]^{(\cF(q_W)\tens\Id)\circ\pi_1+(\Id\tens\cF(q_X))\circ \pi_2} & \cF((Q_W\tens_\cD P_X)\oplus(P_W\tens_\cD Q_X)) \ar[d]^{\cF((q_W\tens\Id)\circ\pi_1+(\Id\tens q_X)\circ\pi_2)}\\
 \cF(P_W)\tens_\cC\cF(P_X) \ar[r]_{G_{P_W,P_X}}\ar[d]^{\cF(p_W)\tens_\cC\cF(p_X)} & \cF(P_W\tens_\cD P_X) \ar[d]^{\cF(p_W\tens_\cD p_X)}\\ \cF(W)\tens_\cC\cF(X) \ar@{-->}[r]_{\exists\,!\,F_{W,X}} \ar[d] & \cF(W\tens_\cD X) \ar[d] \\ 0 & 0\\
}
\end{equation*}
Here $\pi_1$ and $\pi_2$ are the obvious projections, and
\begin{equation*}
\Gamma = \cF(\varepsilon_1)\circ G_{Q_W,P_X}\circ\pi_1 + \cF(\varepsilon_2)\circ G_{P_W,Q_X}\circ\pi_2
\end{equation*}
is an isomorphism, where $\varepsilon_1$ and $\varepsilon_2$ are the obvious inclusions. The columns of the diagram are right exact by Lemma \ref{lem:cokernel_of_tens_prod} and the right exactness of $\tens_\cC$, $\tens_\cD$, and $\cF$. Thus the existence and uniqueness of $F_{W,X}$ follows from the universal property of the cokernel $(\cF(W)\tens_\cC\cF(X), \cF(p_W)\tens_\cC\cF(p_X))$. Since $\Gamma$ and $G_{P_W,P_X}$ are invertible, so is $F_{W,X}$.

To show that the isomorphisms $F_{W,X}$ define a natural isomorphism, consider morphisms $f: W_1\rightarrow W_2$ and $g: X_1\rightarrow X_2$ in $\cD$. Then using \eqref{eqn:gf_def}, 
\begin{align*}
F_{W_2,X_2}\circ(\cF(f)&\,\tens_\cC\cF(g))\circ(\cF(p_f)\tens_\cC\cF(p_g))\nonumber\\
&  = F_{W_2,X_2}\circ(\cF(p_{W_2})\tens_\cC\cF(p_{X_2}))\circ(\cF(q_f^{(2)})\tens_\cC\cF(q_g^{(2)}))\nonumber\\
& =\cF(p_{W_2}\tens_\cD p_{X_2})\circ G_{P_{W_2},P_{X_2}}\circ(\cF(q_f^{(2)})\tens_\cC\cF(q_g^{(2)}))\nonumber\\
& =\cF(p_{W_2}\tens_\cD p_{X_2})\circ\cF(q_f^{(2)}\tens_\cD q_g^{(2)})\circ G_{P_{f},P_{g}}\nonumber\\
& =\cF(f\tens_\cD g)\circ\cF(p_{W_1}\tens_\cD p_{X_1})\circ\cF(q_f^{(1)}\tens_\cD q_g^{(1)})\circ G_{P_{f},P_{g}}\nonumber\\
& = \cF(f\tens_\cD g)\circ\cF(p_{W_1}\tens_\cD p_{X_1})\circ G_{P_{W_1}, P_{X_1}}\circ(\cF(q_f^{(1)})\tens_\cC\cF(q_g^{(1)}))\nonumber\\
& = \cF(f\tens_\cD g)\circ F_{W_1,X_1}\circ(\cF(p_{W_1})\tens_\cC\cF(p_{X_1}))\circ(\cF(q_f^{(1)})\tens_\cC\cF(q_g^{(1)}))\nonumber\\
& =\cF(f\tens_\cD g)\circ F_{W_1,X_1}\circ(\cF(p_f)\tens_\cC\cF(p_g)).
\end{align*}
Since $\cF(p_f)\tens_\cC\cF(p_g)$ is surjective, due to the right exactness of $\cF$ and $\tens_\cC$, this implies $F$ is natural, as required.

To prove that the natural isomorphism $F$ is compatible with the left unit isomorphisms of $\cD$ and $\cC$, we need to show that the diagram
\begin{equation*}
\xymatrixcolsep{4pc}
\xymatrix{
\cF(\vac_\cD)\tens_\cC\cF(X) \ar[r]^{F_{\vac_{\cD},X}} \ar[d]^{\varphi\tens_\cC\Id_{\cF(X)}} & \cF(\vac_\cD\tens_\cD X) \ar[d]^{\cF(l_X)} \\
\vac_\cC\tens_\cC\cF(X) \ar[r]^{l_{\cF(X)}} & \cF(X)\\
}
\end{equation*}
commutes for all $X\in\mathrm{Ob}(\cD)$. In fact, because $G$ is natural and compatible with the left unit isomorphisms in $\til{\cD}$ and $\cC$,
\begin{align*}
\cF(l_X)  \circ F_{\vac_\cD,X} & \circ(\cF(p_{\vac_\cD})\tens_\cC\cF(p_X)) = \cF(l_X)\circ\cF(p_{\vac_\cD}\tens_\cD p_X)\circ G_{P_{\vac_\cD}, P_X}\nonumber\\
& =\cF(l_X)\circ\cF(\Id_{\vac_\cD}\tens_\cD p_X)\circ\cF(p_{\vac_\cD}\tens_\cD\Id_{P_X})\circ G_{P_{\vac_\cD},P_X}\nonumber\\
& =\cF(p_X)\circ\cF(l_{P_X})\circ G_{\vac_\cD,P_X}\circ(\cF(p_{\vac_\cD})\tens_\cC\Id_{\cF(P_X)})\nonumber\\
& =\cF(p_X)\circ l_{\cF(P_X)}\circ((\varphi\circ\cF(p_{\vac_\cD}))\tens_\cC\Id_{\cF(P_X)})\nonumber\\
& =l_{\cF(X)}\circ(\varphi\tens_\cC\Id_{\cF(X)})\circ(\cF(p_{\vac_\cD})\tens_\cC\cF(p_X)).
\end{align*}
Thus because $\cF(p_{\vac_\cD})\tens_\cC\cF(p_X)$ is surjective (by the right exactness of $\cF$ and $\tens_\cC$),
\begin{equation*}
\cF(l_X)\circ F_{\vac_\cD,X}=l_{\cF(X)}\circ(\varphi\tens_\cC\Id_{\cF(X)}).
\end{equation*}
Similarly, $F$ is compatible with the right unit, associativity, and braiding isomorphisms (if any) of $\cD$ and $\cC$, so $F$ and $\varphi$ give $\cF$ the structure of a (braided) tensor functor.

Finally, if $\cD$ and $\cC$ have ribbon twists such that $\cF(\theta_P)=\theta_{\cF(P)}$ for all $P\in\mathrm{Ob}(\til{\cD})$, then naturality of the twist implies that
\begin{align*}
\cF(\theta_X)\circ\cF(p_X)=\cF(p_X)\circ\cF(\theta_{P_X})=\cF(p_X)\circ\theta_{\cF(P_X)} =\theta_{\cF(X)}\circ\cF(p_X)
\end{align*}
for all $X\in\mathrm{Ob}(\cD)$. This implies $\cF(\theta_X)=\theta_{\cF(X)}$ since $\cF(p_X)$ is surjective by the right exactness of $\cF$. This completes the proof of Theorem \ref{thm:ext_is_(br)_tens}.

\end{document}